\documentclass{article}
\usepackage{amssymb,amsmath,bm}
\usepackage{amsthm}
\usepackage{hyperref}
\usepackage{mathrsfs}
\usepackage{textcomp}
\usepackage{enumerate}      
\usepackage{graphicx}        
\usepackage{caption}

\newcommand{\cp}[1]{\vcenter{\hbox{#1}}}
\newcommand{\C}{\mathcal C}
\newcommand{\R}{\mathbb R}
\newcommand{\T}{\mathcal T}
\newcommand{\X}{\mathcal{X}}
\newcommand{\g}{\gamma}
\newcommand{\s}{\Sigma}
\newcommand{\AS}{\mathcal{AS}}

\newtheorem{theorem}{Theorem}[section]
\newtheorem{lemma}[theorem]{Lemma}
\newtheorem{proposition}[theorem]{Proposition}
\newtheorem{definition}[theorem]{Definition}
\newtheorem{corollary}[theorem]{Corollary}
\newtheorem{conjecture}[theorem]{Conjecture}

\theoremstyle{remark}
\newtheorem{remark}[theorem]{Remark}

\theoremstyle{remark}
\newtheorem{example}[theorem]{Example}

\numberwithin{equation}{section}

\begin{document}

\title{\textbf{The skein algebra of arcs and links and the decorated Teichm\"uller space}
\bigskip}

\author{\medskip Julien Roger and Tian Yang}

\date{}

\maketitle
\begin{abstract} We define an associative $\mathbb{C}[[h]]$--algebra $\AS_h(\s)$ generated by framed arcs and links over a punctured surface $\s$ which is a quantization of the Poisson algebra $\C(\s)$ of arcs and curves on $\s$. We then construct a Poisson algebra homomorphism from $\C(\s)$ to the algebra of smooth functions on the decorated Teichm\"uller space endowed with the Weil-Petersson Poisson structure. The construction relies on a collection of geodesic lengths identities in hyperbolic geometry which generalize Penner's Ptolemy relation, the trace identities and Wolpert's cosine formula. As a consequence, we derive an explicit formula for the geodesic lengths functions in terms of the edge lengths of an ideally triangulated decorated hyperbolic surface.
\end{abstract}


\section{Introduction}

The skein module $\mathcal{S}_q(M)$ of a 3--manifold $M$ was introduced independently by Turaev \cite{Turaev1} and Przytycki \cite{Przytycki} as a generalization of the Jones polynomial of a link in $S^3$, using as a key ingredient the Kauffman bracket skein relation. If the 3--manifold is the product $\s\times [0,1]$ of a surface $\s$ by an interval, it has a natural structure of an algebra, and is at the heart of the combinatorial approach to constructing a TQFT developped in \cite{BHMV}. This construction in turn has had many applications in low dimensional topology (see \cite{FWW} for example) and is presumed to be equivalent to the geometric approach to TQFT coming from conformal theory (see the recent work of Andersen and Ueno \cite{AU}).

On the other hand, the skein algebra turns out to be deeply related to the $SL_2$--geometry of the undelying surface. More precisely, following the work of Turaev \cite{Turaev2}, Bullock, Frohman and Kania-Bartoszy\'nska\,\cite{Bu,BFK} and Przytycki and Sikora \cite{PS}, the skein algebra can be understood as a quantization of the $SL_2(\mathbb{C})$--character variety of $\s$.

The goal of the present paper is the following: First, we extend the notion of skein algebra for a punctured surface by allowing for framed arcs between the punctures. In order to show that this construction is well-defined we revisit arguments found in \cite{G,Przytycki} and extend them when necessary. Second, following the approach of \cite{BFK}, we relate our construction to geometric structures of $\s$. In our approach, the relevant object turns out to be the decorated Teichm\"uller space and the notion of $\lambda$--length introduced by Penner \cite{Penner1} . In order to establish a connection between the two constructions, we derive a series of identities in hyperbolic geometry which we hope will be of interest in themselves.

To motivate our approach let us first recall some of the steps in the construction of the skein algebra and its relationship to the character variety. Let $q$ be the formal power series $e^\frac{h}{4}\in\mathbb{C}[[h]]$. The skein algebra $\mathcal{S}_h(\s)$, introduced by Przytycki\,\cite{Przytycki} and Turaev\,\cite{Turaev1}, is the $\mathbb{C}[[h]]$--algebra generated by isotopy classes of framed links in $\s\times\left[0,1\right]$ subject to the Kauffman bracket skein relation
\[\cp{\includegraphics[width=1cm]{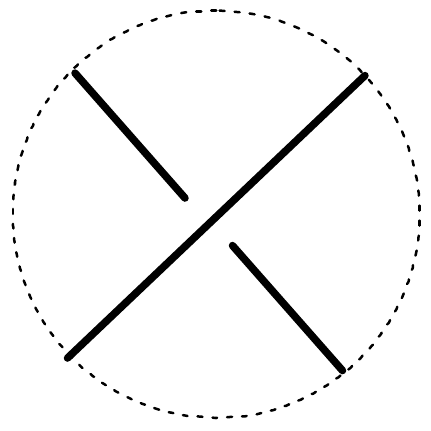}}\ =\ q\ \cp{\includegraphics[width=1cm]{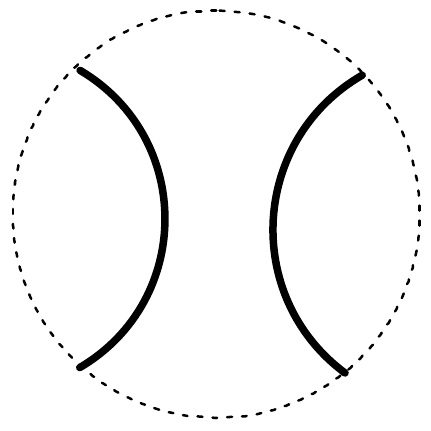}}\  +\ q^{-1}\ \cp{\includegraphics[width=1cm]{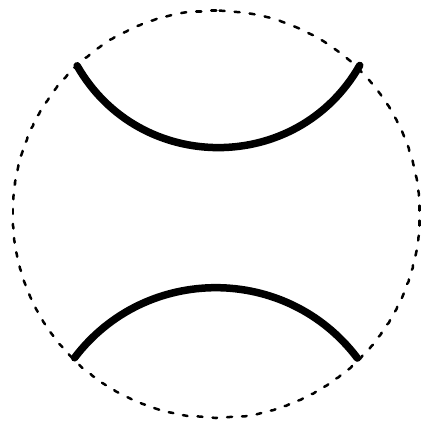}}\]
as well as the framing relation $\cp{\includegraphics[width=0.4cm]{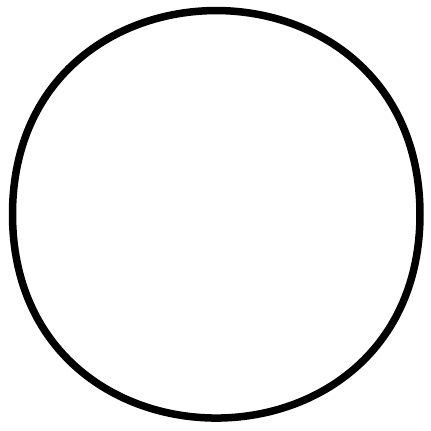}} = -q^2-q^{-2}$. In \cite{Turaev2}, Turaev studied the relationship between the skein algebra and the Lie algebra of curves on $\s$ introduced by Goldman\,\cite{G}. In turn, in the work of Goldman, the Lie bracket on curves is related to the Weil-Petersson Poisson structure on the $SL_2(\mathbb{C})$--character variety $\X(\s)$ of $\s$, that is, the space of conjugacy classes of representations $\rho\colon\pi_1(\s)\to SL_2(\mathbb{C})$. A direct relationship between the skein algebra and the character variety was described by Bullock\,\cite{Bu} who constructed a surjective homomorphism from the commutative algebra $\mathcal{S}_0(\s)$ to the coordinate ring $\mathbb{C}[\X(\s)]$ of $\X(\s)$. This map turns out to be an isomorphism by the work of Przytycki and Sikora\,\cite{PS} (see more recently \cite{CM} for a direct proof). Up to a sign, it assigns to each free homotopy class of curves $\g$ in $\s$ its trace function $tr_\g$ in $\mathbb{C}[\X(\s)]$, given by taking the trace of representations evaluated at $\g$. One of the key ingredients is then given by the trace identities which relate the product of traces of two intersecting curves with the traces of their resolutions at one point. These identities, in turn, come from the classical formula $tr A\cdot tr B=tr AB + tr AB^{-1}$ relating traces in $SL_2(\mathbb{C})$. Using this isomorphism, Bullock, Frohman and Kania-Bartoszy\'nska\,\cite{BFK} showed that the skein algebra is in fact a quantization of the character variety for the Goldman-Weil-Petersson bracket, in the sense of deformation of Poisson algebras. What this means is that $\mathcal{S}_0(\s)$, endowed with the Poisson bracket inherited from the commutator on $\mathcal{S}_h(\s)$, is isomorphic \emph{as a Poisson algebra} to $\mathbb{C}[\X(\s)]$.

Our goal is to extend this construction by introducing framed arcs in the definition of the skein algebra for a surface with punctures. We define a generalized framed link to be an embedding of a collection of annuli and strips in $\s\times [0,1]$, so that the ends of the strips are above the punctures (see Section~\ref{section:2} for a precise definition). A component given by a strip will be called a framed arc.
\[\includegraphics[width=3cm]{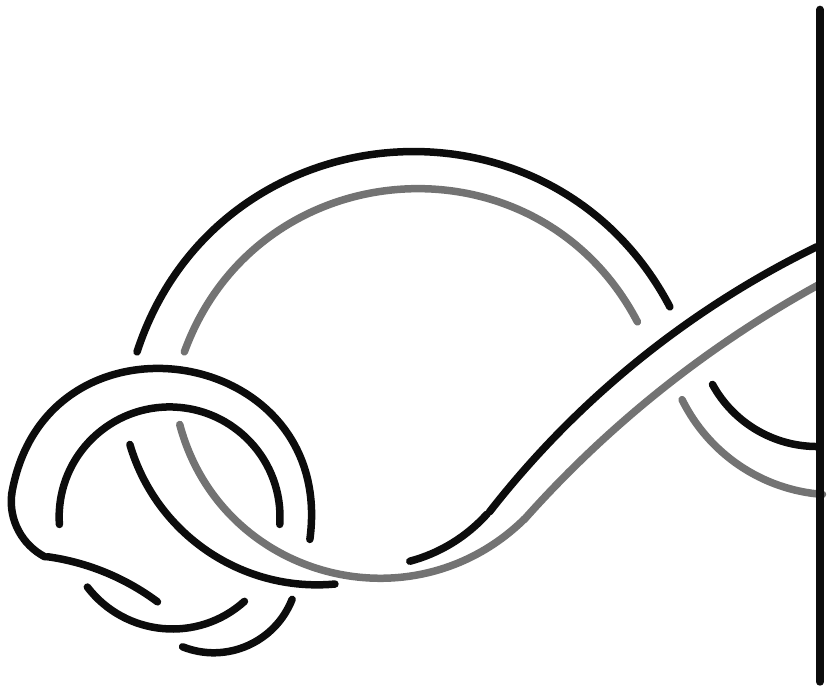}v\times [0,1]\]
The \emph{skein algebra of arcs and links} $\AS_h(\s)$ of $\s$ will then be generated by isotopy classes of generalized framed links. The usual skein relation still applies for crossings occurring above $\s$, where we allow some of the strands to be arcs. When two arcs meet at a puncture we introduce the so-called \emph{puncture-skein relation}
\[\cp{\includegraphics[width=1cm]{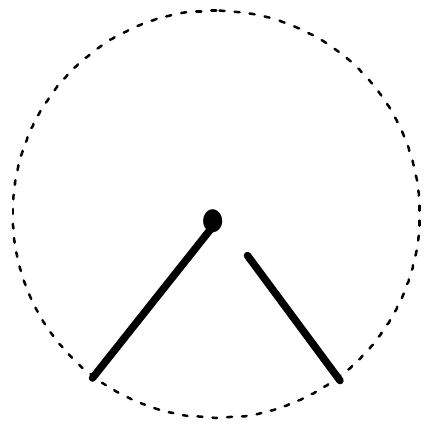}}\ =\ \frac{1}{v}\Big(q^{\frac{1}{2}}\ \cp{\includegraphics[width=1cm]{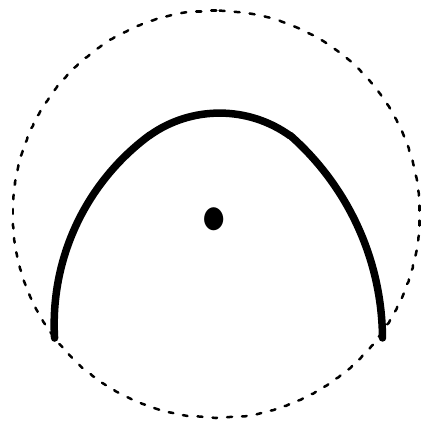}}\  +\ q^{-\frac{1}{2}}\ \cp{\includegraphics[width=1cm]{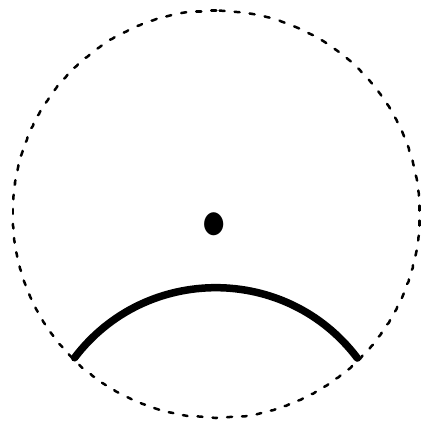}}\ \Big).\]
Here $v$ is a central element associated to the puncture which turns out to be essential when trying to interpret this relation geometrically. In addition, the framing relation still applies and we also impose the \emph{puncture relation} $\cp{\includegraphics[width=0.4cm]{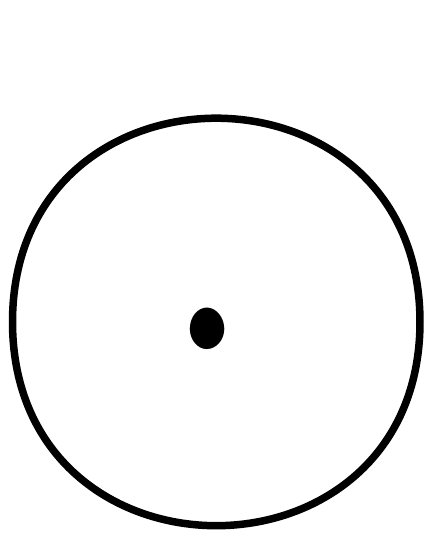}}\ =\ q+q^{-1}$.

In the non-quantum case, we consider the algebra $\C(\s)$ generated by arcs and loops on $\s$ itself subject to (classical) skein relations. It admits a Poisson bracket decribed in terms of resolutions of intersections inside the surface and at the punctures which generalizes Goldman's Lie bracket on loops. Using arguments similar to the ones in \cite{BFK}, we show that this bracket comes from the commutator in $\AS_h(\s)$, In other words
\begin{theorem}
\label{intro1}
$\AS_h(\s)$ is a quantization of $\C(\s)$.
\end{theorem}

The next step of our construction is to relate the algebra $\C(\s)$ to the $SL_2$--geometry of the surface $\s$, in the case when $\chi(\s)<0$ and the set of punctures $V$ is non-empty. In this context, the relevant object is the decorated Teichm\"uller space $\T^d(\s)$ introduced by Penner \,\cite{Penner1}. It is defined as a bundle over the usual Teichm\"uller space $\mathcal{T}(\s)$ with fiber $\mathbb{R}^V_{>0}$. Given a hyperbolic metric $m\in\mathcal{T}(\s)$, the choice of a point in the fiber corresponds to fixing the length of a horocycle at each of the punctures of $\s$. This assignment, in turn, permits to measure the length $l(\alpha)$ of each arc $\alpha$ between horocycles. A more relevant quantity in our context is the $\lambda$-length of $\alpha$ given by $\lambda(\alpha)=e^\frac{l(\alpha)}{2}$. They satisfy the well-known Ptolemy relation
\[\lambda(e)\lambda(e')=\lambda(a)\lambda(c)+\lambda(b)\lambda(d)\]
where $a$, $b$, $c$, $d$ are the consecutive edges of a square and $e$ and $e'$ are its diagonals. Graphically, the Ptolemy relation can the be rewritten:
\[\cp{\includegraphics[width=1.4cm]{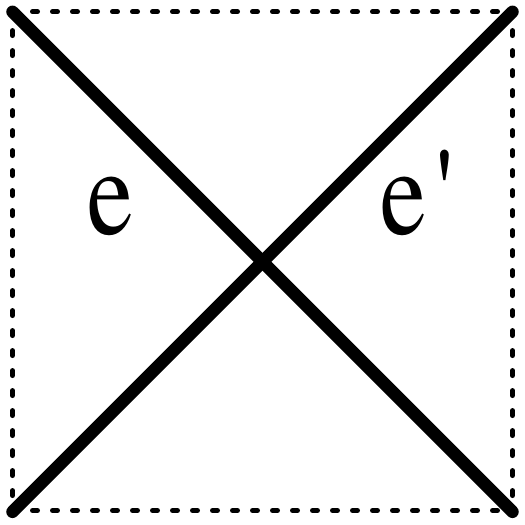}}\ \ =\ \ \cp{\includegraphics[width=1.4cm]{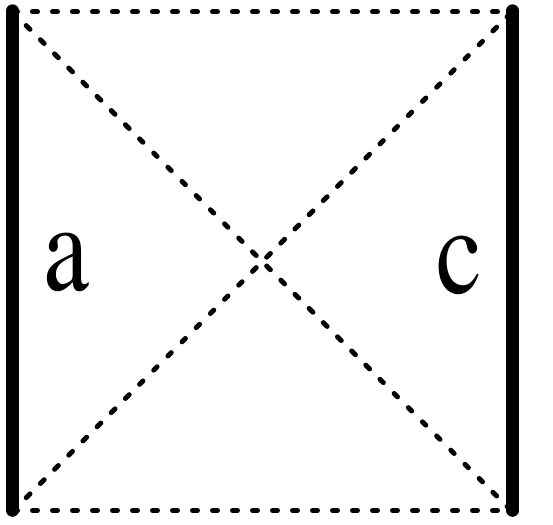}}\ \ +\ \ \cp{\includegraphics[width=1.4cm]{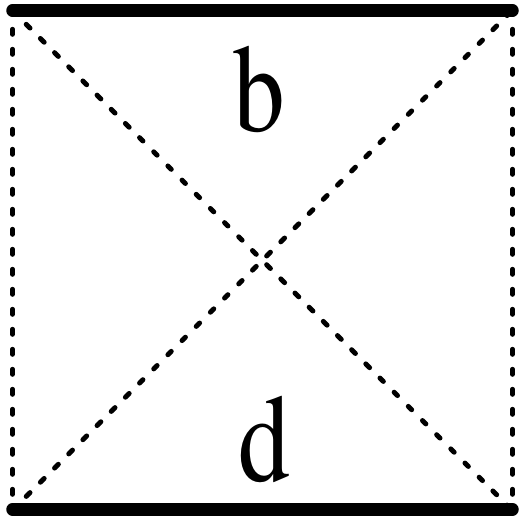}}\ \]
which we want to interpret as a (non-quantum) skein relation.

Using these notions, we obtain the following theorem.
\begin{theorem}
\label{intro2}
There is a well-defined homomorphism of Poisson algebras
\[\Phi\colon\C(\s)\to C^\infty(\mathcal{T}^d(\s))\,.\]
\end{theorem}

Up to signs, this map sends loops to their trace functions, arcs to their $\lambda$-lengths and punctures to horocycle lengths around them. The Poisson structure on $\mathcal{T}^d(\s)$ is an extension of the usual Weil-Petersson Poisson bracket on $\mathcal{T}(\s)$ and was described by Mondello\,\cite{Mondello}. The proof of the theorem relies on a collection of length identities which generalize Penner's Ptolemy relation, the trace identities and Wolpert's cosine formula\,\cite{Wolpert} for the Poisson bracket of two trace functions. These identities are derived in turn from a set of ``cosine laws'' which can be found in the Appendix A of \cite{GL} by Guo and Luo.

Combining Theorems~\ref{intro1} and \ref{intro2}, it is tempting to interpret $\AS_h(\s)$ as a quantization of the decorated Teichm\"uller space. Following \cite{Bu} and \cite{PS}, we first conjecture that the homomorphism $\Phi$ in the theorem above is injective. Another important step would then be to understand what is the correct ``algebra of functions'' on $\T^d(\s)$. In their work, Bullock, Frohman and Kania-Bartoszy\'nska make use of the fact that the character variety is an algebraic variety, and as such they are lead to use its coordinate ring $\mathbb{C}[\X(\s)]$. This choice is made natural by the fact that $\mathbb{C}[\X(\s)]$ is generated by trace functions. In our context, the image of the homomorphism $\Phi$ is essentially the subalgebra generated by trace functions and $\lambda$--lengths. A natural question to ask is then if $\T^d(\s)$ has a natural structure of an algebraic variety for which its coordinate ring coincides with this subalgebra. An important observation which makes this approach sensible is the fact that the $\lambda$--lengths associated to the edges of an ideal triangulation form a coordinate system on $\T^d(\s)$ in which every trace function can be written as a Laurent polynomial (see Proposition~\ref{Laurent} and Appendix~\ref{B}). The way this fact translates at the level of the skein algebra $\AS_h(\s)$ however remains an intriguing problem.

The fact that the trace identity and the Ptolemy relation can be combined into generalized skein relations involving both arcs and curves has been used recently in works of Dupont and Palesi\,\cite{DP} and Musiker and Williams\,\cite{MW}, in the context of cluster algebras associated to triangulated surfaces. It would be interesting to see if our work applies to the context of quantum cluster algebras as defined by Berenstein and Zelevinsky\,\cite{BZ}. Closely related to these considerations is the construction of quantum trace functions in the context of the quantization of Teichm\"uller space. This problem was solved recently by Bonahon and Wong in \cite{BW1, BW2} using the skein relation in a crucial way. In turn, their construction is based on the use of shear coordinates\,\cite{Bo} which are closely related to $\lambda$-lengths. As such, we hope that our work could shed new light on the relationship between the skein algebra and the quantum Teichm\"uller space.
\medskip

\textbf{Acknowledgments:} We are grateful to M. Freedman, C. Frohman, J. Przytycki and Z-H.Wang for showing interest in this work and for their helpful comments. We would like also to thank F. Bonahon and F. Luo for their continuous support and for many conversations which have lead to this work. The second author is partly supported by an NSF research fellowship.

\section{Algebraic aspects}\label{section:2}

\subsection{The skein algebra of arcs and links}

We consider a surface $\s$ obtained from a closed oriented surface $\overline{\s}$ by removing a possibly empty finite subset $V$. Elements of $V$ will be called the \emph{punctures} of $\s$.

\begin{definition} A continuous map $\alpha=\coprod_i\alpha_i\sqcup\coprod_j l_j$ from a domain $D$ consisting of a finite collection of strips $\coprod_i\ [0,1]\times(-\epsilon,\epsilon)$ and annuli $\coprod_j S^1\times(-\epsilon,\epsilon)$ into $\overline{\s}\times [0,1]$ is called a \emph{generalized framed link} in $\s\times [0,1]$ if 
\begin{enumerate}[(1)]
\item $\alpha$ is an injection into $\overline{\s}\times (0,1)$;

\item each $l_j$ is an embedding into $\s\times[0,1]$;

\item the restriction of each $\alpha_i$ to ${(0,1)\times(-\epsilon,\epsilon)}$ is an embedding into $\s\times [0,1]$;

\item the restriction of each $\alpha_i$ to ${\{0,1\}\times(-\epsilon,\epsilon)}$ is an orientation preserving embedding into $V\times [0,1].$
\end{enumerate}
Two generalized framed links $\alpha$ and $\beta$ are \emph{isotopic} if there exists a continuous map $H\colon D\times[0,1]\rightarrow\overline{\s}\times[0,1]$ such that $H_0=\alpha$ and $H_1=\beta$, and $H_t$ is a generalized framed link for each $t\in(0,1)$. 
\end{definition}

Each strip $\alpha_i$ in a generalized link $\alpha$ will be called a \emph{(framed) arc}, with the understanding that such a component can be ``knotted''. Condition $(1)$ implies that each arc component of a generalized link ends at a different height above the punctures of $\s$. Condition $(4)$ prevents a framed arc from doing a ``half-twist'' between two punctures.

Some conventions are needed when considering a diagram of a generalized link projected onto $\s$. First, modulo an isotopy, we can assume that the framing of a generalized link always points in the vertical direction. On a link diagram this will correspond to the direction pointing toward the reader. Second, we use the usual convention to encode which strand of a generalized link passes over another in $\s\times[0,1]$ and we assume that the diagram only possesses ordinary double points in $\s$. However, we cannot impose this convention for intersections occurring at a puncture, since more than two arcs can meet at a puncture and such intersections cannot be resolved via an isotopy of $\s$. If two strands of arcs meet above a puncture, we consider the following diagram:
\[\includegraphics[width=1.5cm]{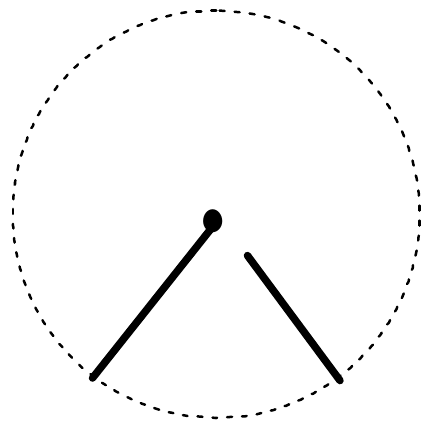}\]
Here the left strand ends above the right one at the puncture, and no other strand ends in between. We call such a configuration a \emph{consecutive crossing}. However, we do not necessarily sketch the strands ending above or below. In some cases, we need to study relations involving more than two strands ending at a puncture, then a picture such as this one,
\[\includegraphics[width=1.5cm]{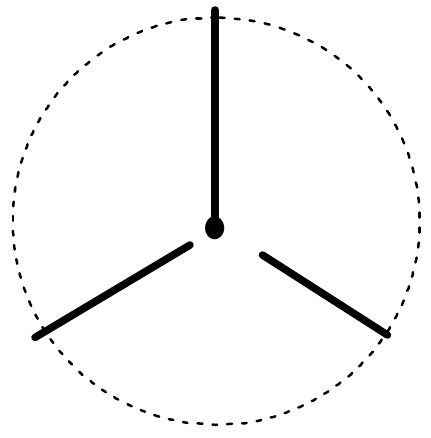}\]
will be supplemented with an explanation of the respective positions of the strands lying under the top one.

As is well known, two diagrams correspond to isotopic framed links if and only if one can be obtained from the other by a sequence of Reidemeister Moves II and III. This is also true in the case of a generalized link if we add the moves described in Figure~\ref{RII'}, which we will call \emph{Reidemeister  Moves II$'$}. 
\begin{figure}[htbp]\centering
\includegraphics[width=9cm]{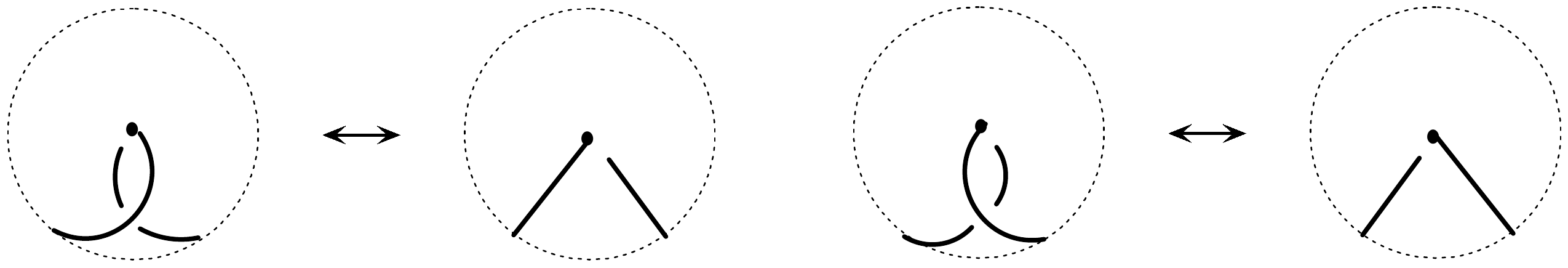}\\
\caption{Reidemeister Move II$'$.}\label{RII'}
\end{figure}
Indeed, this move is obtained by replacing one of the crossings in Reidemeister Move II by a crossings at a puncture. The only possible move obtained by replacing both of the crossings in Reidemeister Move II by crossings at punctures is a composition of Reidemeister Moves II and II$'$ as follows:
\[\includegraphics[width=7cm]{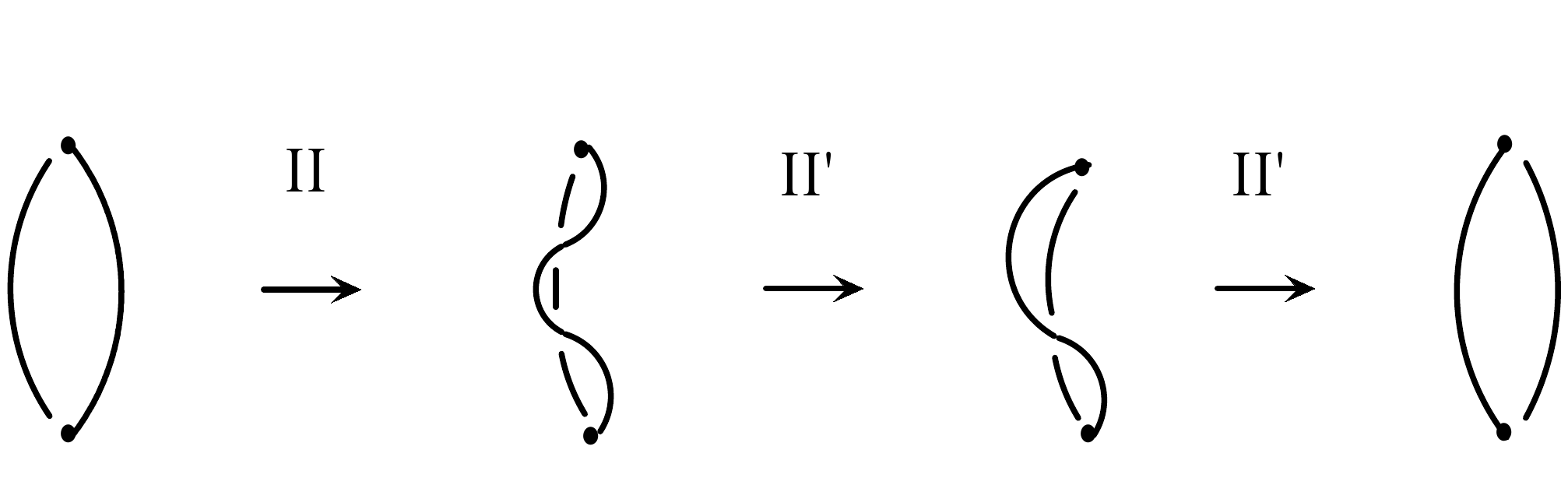}\]
There are no analogous moves to Reidemeister Move III, since a strand of a generalized link cannot be isotoped through a puncture.

Given two generalized links $\alpha$ and $\beta$ in $\s\times[0,1]$ we define the \emph{stacking of} $\alpha$ \emph{over} $\beta$ to be the link in $\s\times[0,1]$ obtained by rescaling $\alpha$ to $\s\times [\frac{1}{2},1]$ and $\beta$ to $\s\times[0,\frac{1}{2}]$ and taking there union. This operation is compatible with isotopies.

Let $\mathbb{C}[[h]]$ denote the ring of power series in $h$. We endow it with the $h$--adic topology and, following the approach of Bullock, Frohman and Kania-Bartoszy\'nska \cite{BFK} we will work over the category of topological $\mathbb{C}[[h]]$-modules. We refer to their paper as well as to the standard reference $\cite{Kassel}$ for details.

We let $\mathcal{L}$ be the set of isotopy classes of generalized framed links in $\s\times [0,1]$ together with the empty link and let $V^{\pm 1}$ be the set of punctures $v$ of $\Sigma$ and there formal inverses $v^{-1}$. We consider the $\mathbb{C}$--vector space $\mathbb{C}[\mathcal{L},V^{\pm1}]$ with basis $\mathcal{L}\cup V^{\pm1}$ and define a product on this space via
\begin{enumerate}[(1)]
\item the product $\alpha\cdot\beta$ of $\alpha$ and $\beta$ in $\mathcal{L}$ is obtained by stacking $\alpha$ over $\beta$;
\item the elements of $V^{\pm 1}$ are central and $v\cdot v^{-1}=1$ for each $v\in V$.
\end{enumerate}
The empty link is the identity for this operation.

We can then form the set $\mathbb{C}[\mathcal{L},V^{\pm1}][[h]]$ of formal power series with coefficients in $\mathbb{C}[\mathcal{L},V^{\pm1}]$ which inherits a natural structure of $\mathbb{C}[[h]]$-module. The multiplication on $\mathbb{C}[\mathcal{L},V^{\pm1}]$ extends naturally to $\mathbb{C}[\mathcal{L},V^{\pm1}][[h]]$ and turns it into a topological algebra.

We are now ready to introduce the main object of this article.

\begin{definition}\label{arcskein} Let $q$ be the formal power series $e^{\frac{h}{4}}\in\mathbb{C}[[h]]$. The \emph{skein algebra of arcs and links} $\AS_h(\s)$ is the quotient of $\mathbb{C}[\mathcal{L},V^{\pm 1}][[h]]$ by the sub--module generated by the following relations:
\begin{enumerate}[(1)]
\item \emph{\bf Kauffman Bracket Skein Relation:} For a crossing in the surface, we have
\begin{equation*}
\cp{\includegraphics[width=1cm]{crossing}}\ =\ q\ \cp{\includegraphics[width=1cm]{+crossing}}\  +\ q^{-1}\ \cp{\includegraphics[width=1cm]{ncrossing}}\ ;
\end{equation*}

\item \emph{\bf Puncture-Skein Relation:}  For a consecutive crossing at a puncture $v$, we have 
\begin{equation*}
\cp{\includegraphics[width=1cm]{v}}\ =\ \frac{1}{v}\Big(q^{\frac{1}{2}}\ \cp{\includegraphics[width=1cm]{+v}}\  +\ q^{-\frac{1}{2}}\ \cp{\includegraphics[width=1cm]{nv}}\Big);
\end{equation*}

\item \emph{\bf Framing Relation:} For the isotopy class of a trivial loop, we have
\begin{equation*}
\cp{\includegraphics[width=0.8cm]{trivial}}\ =\ -q^2-q^{-2};
\end{equation*} 

\item \emph{\bf Puncture Relation:} For the isotopy class of a circle around a puncture, we have
\begin{equation*}
\cp{\includegraphics[width=0.8cm]{curve-around-puncture}}\ =\ q+q^{-1}.
\end{equation*} 
\end{enumerate}
The multiplication $\cdot$ on $\AS_h(\s)$ is induced by the stacking operation on $\mathbb{C}[\mathcal{L},V^{\pm 1}][[h]]$.
\end{definition}

Some comments are in order to justify this definition. First, note that if $V$ is empty, then $\AS_h(\s)$ coincides with the algebra defined in $\cite{BFK}$, which is a topological version of the usual Kauffman bracket skein algebra $\mathcal{S}_q(\Sigma)$ over a formal parameter $q$ as defined in \cite{Przytycki} and \cite{Turaev1}. Second, The choice of the coefficients $q^{\pm\frac{1}{2}}$ in the puncture-skein relation turns out to be essential in proving that this algebra is well-defined and that the product is associative. It will also have a geometrical justification which will be explained in Section \ref{section:3}. Finally, the central elements $v$ associated to the punctures are not essential from the algebraic point of view but will play an important r\^ole in the geometric interpretation given in Section~\ref{section:3}, where they will be related to the choice of a horocycle at each puncture.

We recall that a $\mathbb{C}[[h]]$--module $M$ is called \emph{topologically free} if there exists a $\mathbb{C}$-vector space $\mathcal{V}$ so that $M$ is isomorphic to $\mathcal{V}[[h]]$ (see for example \cite{Kassel}). We have the following

\begin{theorem} The skein algebra $\big(\AS_h(\s),\ \cdot\ \big)$ is a well-defined topologically free associative $\mathbb{C}[[h]]$--algebra.  
\end{theorem}

\begin{proof} In order to verify the well-definition of the multiplication, it suffices to show that it is invariant under Reidemeister Moves II, II$'$ and III. The invariance under Reidemeister Moves II and III follows from the same arguments as in \cite{Przytycki}.  For Reidemeister Move II$'$, we first calculate that
\begin{equation*}
\begin{split}
\cp{\includegraphics[width=1cm]{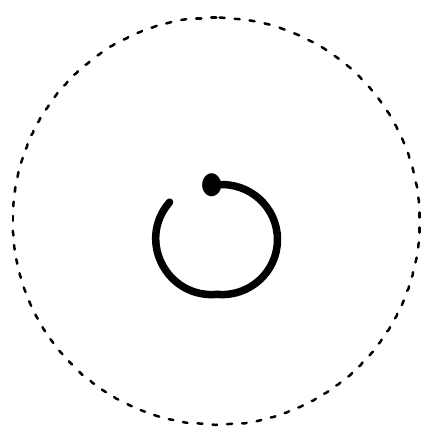}}\ &=\ \frac{1}{v}\big(q^{\frac{1}{2}}\ \cp{\includegraphics[width=1cm]{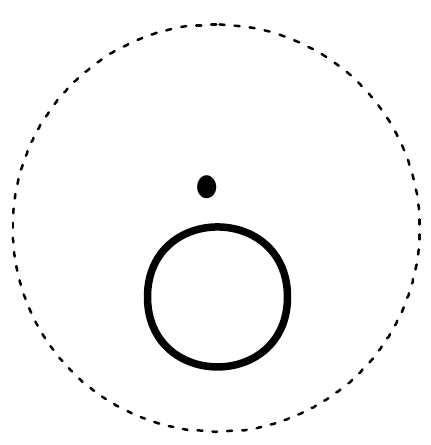}}\ +\ q^{-\frac{1}{2}}\ \cp{\includegraphics[width=1cm]{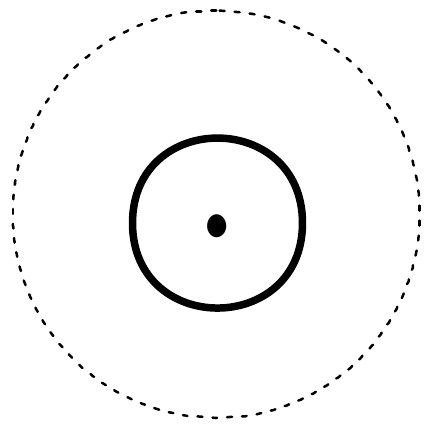}}\,\big)\\
&=\frac{1}{v}\big(q^{\frac{1}{2}}(-q^2-q^{-2})+q^{-\frac{1}{2}}(q+q^{-1})\big)=\frac{1}{v}(q^{\frac{1}{2}}-q^{\frac{5}{2}}),
\end{split}
\end{equation*}
where $v$ is the puncture and the second equality follows from the framing and the puncture relations. With this, we obtain
\begin{equation*}
\begin{split}
\cp{\includegraphics[width=1cm]{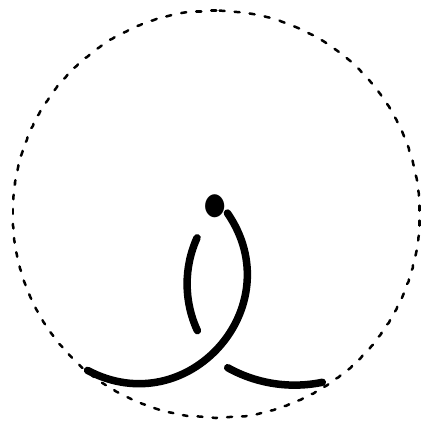}}\ &=\ q\ \cp{\includegraphics[width=1cm]{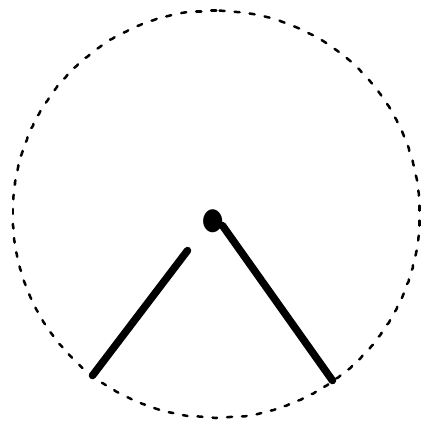}}\ +\ q^{-1}\ \cp{\includegraphics[width=1cm]{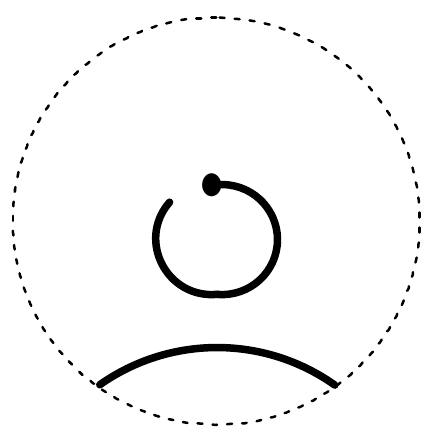}}\\
&=\ \frac{1}{v}q\big(q^{\frac{1}{2}}\ \cp{\includegraphics[width=1cm]{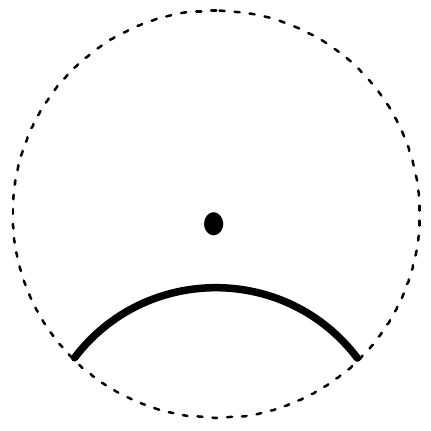}}\  +\   q^{-\frac{1}{2}}\ \cp{\includegraphics[width=1cm]{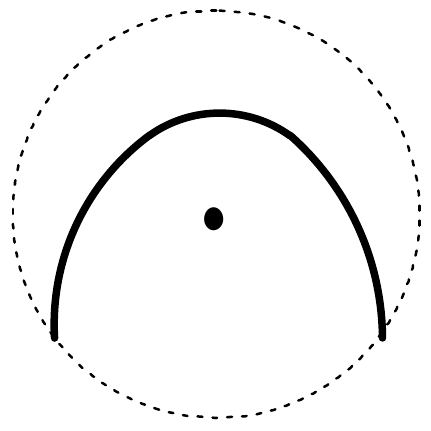}}\,\big)+ \frac{1}{v}q^{-1}(q^{\frac{1}{2}}-q^{\frac{5}{2}})\ \cp{\includegraphics[width=1cm]{3}}\\
&=\  \frac{1}{v}\Big(q^{\frac{1}{2}}\ \cp{\includegraphics[width=1cm]{2}} \ +\  q^{-\frac{1}{2}}\ \cp{\includegraphics[width=1cm]{3}}\Big)=\  \cp{\includegraphics[width=1cm]{1}}\ ,
\end{split}
\end{equation*}
where the first equality follows from the Kauffman bracket skein relation and the second equality from the puncture-skein relation and the previous calculation. The well-definition under the other Reidemeister Move II$'$ is verified similarly. 

To show that $\alpha\cdot\bigodot=\bigodot\cdot\ \alpha=(q+q^{-1})\alpha$, the only case we need to consider is when $\bigodot$ is a circle around a puncture $v$ and $\alpha$ is a framed arc with $v$ one of its end points. For $\alpha\cdot\bigodot$, we have
\begin{equation*}
\begin{split}
\cp{\includegraphics[width=1cm]{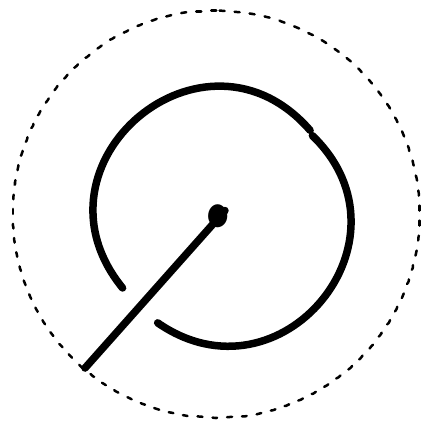}}\ &=\ q\ \cp{\includegraphics[width=1cm]{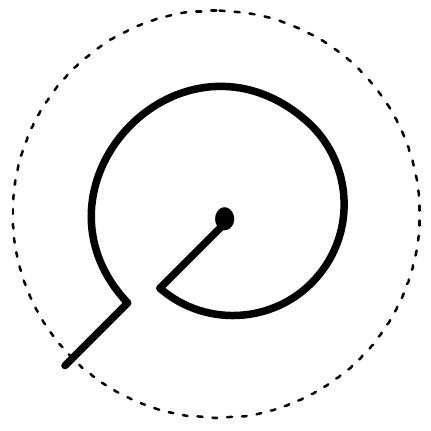}}\ +\ q^{-1}\ \cp{\includegraphics[width=1cm]{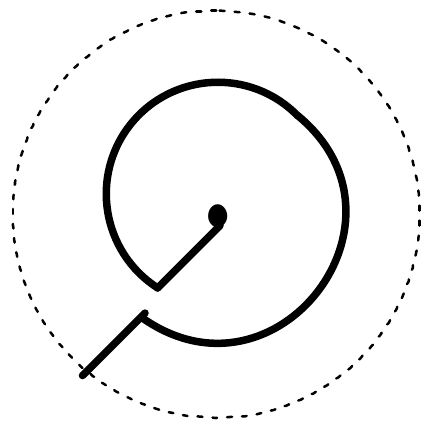}}\ =\ \big(q+q^{-1}\big)\ \cp{\includegraphics[width=1cm]{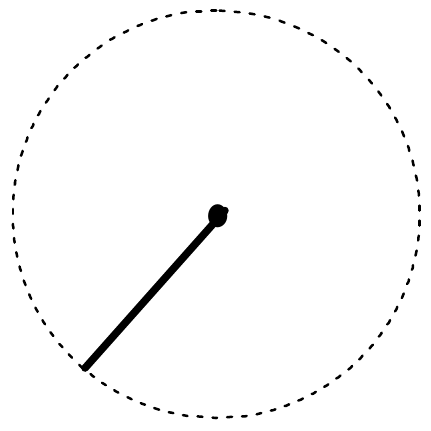}}\ ;\\
\end{split}
\end{equation*}
and similarly for $\bigodot\cdot\ \alpha=(q+q^{-1})\alpha$.

When three links cross inside the surface, the associativity follows from the same arguments as in \cite{Przytycki}, and similarly if some intersections happen at a puncture as long as there are no triple points. If three arcs $\alpha$, $\beta$ and $\gamma$ meet at a puncture $v$, say in counterclockwise order, we have for $(\alpha\cdot\beta)\cdot\gamma$ that
\begin{equation*}
\begin{split}
_{\beta}\cp{\includegraphics[width=1cm]{3valent}}_{\gamma}^{\alpha}\ &=\ \frac{1}{v}\Big(q^{\frac{1}{2}}\ \cp{\includegraphics[width=1cm]{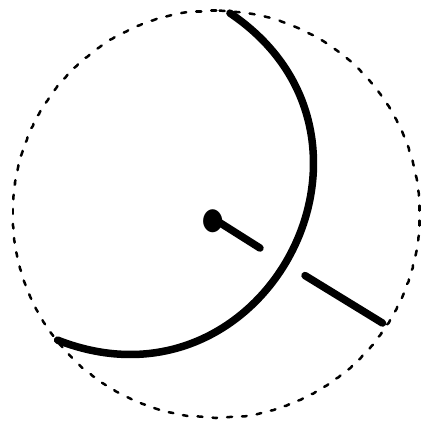}}\ +\ q^{-\frac{1}{2}}\ \cp{\includegraphics[width=1cm]{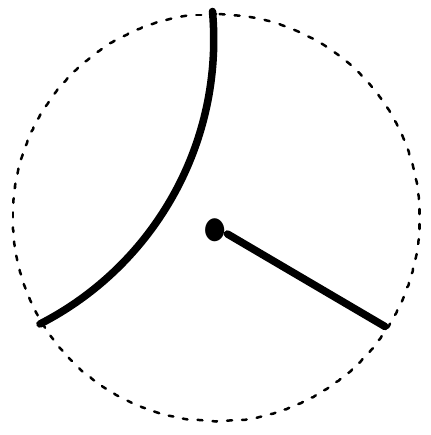}}\Big)\\
&=\frac{1}{v}\Big(q^{\frac{3}{2}}\ \cp{\includegraphics[width=1cm]{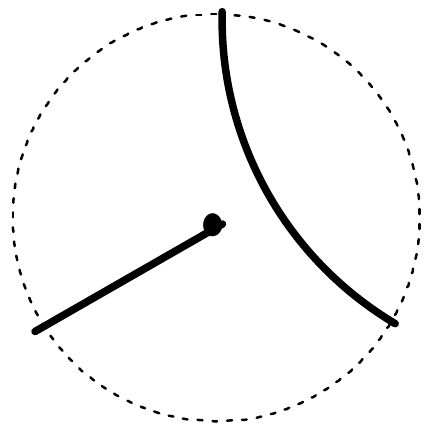}}\ +\ q^{-\frac{1}{2}}\ \cp{\includegraphics[width=1cm]{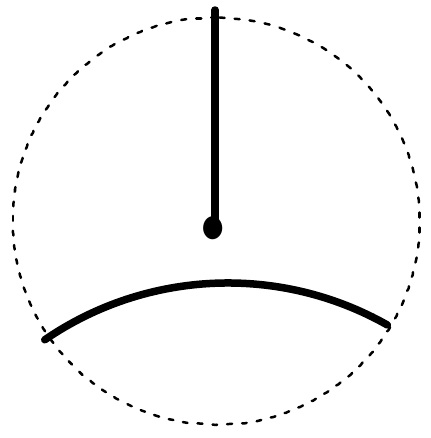}}\ +\ q^{-\frac{1}{2}}\ \cp{\includegraphics[width=1cm]{c}}\Big);
\end{split}
\end{equation*}
and for $\alpha\cdot(\beta\cdot\gamma)$ that
\begin{equation*}
\begin{split}
_{\beta}\cp{\includegraphics[width=1cm]{3valent}}_{\gamma}^{\alpha}\ &=\frac{1}{v}\Big( q^{\frac{1}{2}}\ \cp{\includegraphics[width=1cm]{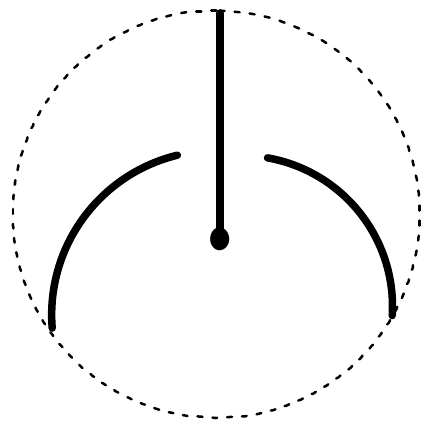}}\ +\ q^{-\frac{1}{2}}\ \cp{\includegraphics[width=1cm]{e}}\Big)\\
&=\frac{1}{v}\Big( q^{\frac{3}{2}}\ \cp{\includegraphics[width=1cm]{d}}\ +\ q^{-\frac{1}{2}}\ \cp{\includegraphics[width=1cm]{c}}\ +\ q^{-\frac{1}{2}}\ \cp{\includegraphics[width=1cm]{e}}\Big).
\end{split}
\end{equation*}
The case when $\alpha$, $\beta$ and $\gamma$ are ordered clockwise is similar.

The proof that $\AS_h(\s)$ is topologically free follows from the same arguments as in \cite{BFK}. In our case, a diagram of a generalized framed link in $\s\times (0,1)$ is a graph in $\overline{\s}$ which is four-valent in $\s$ and many-valent at $V$ with crossings and vertical framing. Two diagrams represent the same generalized framed link if and only if they differ by a sequence of isotopies of $\s$ and Reidemeister Moves II, II$'$ and III. We consider the vector space $\mathcal{W}$ over $\mathbb{C}$ whose basis consists of all diagrams which have no crossing in $\overline{\s}$, no trivial loops and no loops bounding a puncture and let $\mathcal{V}=\mathcal{W}\otimes\mathbb{C}[V^{\pm 1}]$. To any element in $\mathbb{C}[\mathcal{L},V^{\pm}]$ one can associate an element of $\mathcal{V}[[h]]$ by considering one of its associated diagrams and first resolving intersections at the punctures using the puncture-skein relation (2), then resolving intersections in $\s$ using the skein relation (1) and finally sending each trivial loops to $-q^2-q^{-2}$ and loops around punctures to $q+q^{-1}$. This process converges when extended to power series in $\mathbb{C}[\mathcal{L},V^{\pm}][[h]]$ and can be seen to descend to a well-defined homomorphism of topological algebras $\Psi\colon \AS_h(\s)\to\mathcal{V}[[h]]$ whose inverse is given by considering the inclusion of $\mathcal{V}[[h]]$ in $\mathbb{C}[\mathcal{L}, V^{\pm}][[h]]$ and taking the quotient.
\end{proof}

\begin{remark}
In the rest of this paper, we call an element $\mathcal{S}$ of $\mathcal{V}=\mathcal{W}\otimes\mathbb{C}[V^{\pm}]$ a \emph{state}. Recall that $\mathcal{W}$ consists of all the diagrams on $\s$ with no crossings and no loops bounding a disk or a puncture.
\end{remark}


\subsection{The Poisson algebra of curves on a punctured surface}

The classical counterpart of the skein algebra can be defined in terms of curves on the surface $\s$ itself. We define a \emph{generalized curve} to be a union of immersed loops and arcs on $\s$ with ends at the punctures. We also consider the empty set as a curve. Any two generalized curves will be identified if they differ by a \emph{regular isotopy}, that is, if one can be obtained from the other by a sequence of isotopies of $\s$ and Reidemeister Moves II, II$'$ and III. We do not however identify curves differing by a Reidemeister Move I. 

\begin{definition}\label{def:curve} The \emph{algebra of curves} $\mathcal{C}(\s)$ on $\s$ is the quotient of the $\mathbb{C}$--algebra generated by the regular isotopy classes of generalized curves on $\s$, the punctures of $\s$ and their formal inverses, modulo the subspace generated by the following relations:
\begin{enumerate}[(1$'$)]
\item \emph{\bf Skein Relation:} $\cp{\includegraphics[width=0.8cm]{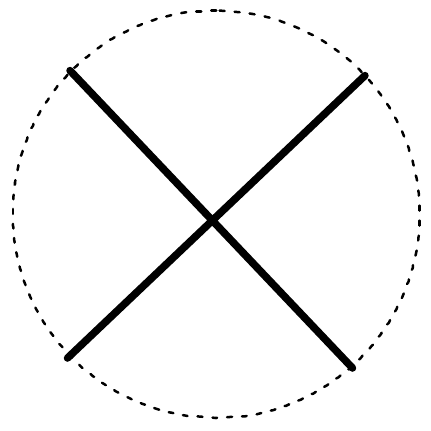}}\ =\ \cp{\includegraphics[width=0.8cm]{+crossing}}\  +\ \cp{\includegraphics[width=0.8cm]{ncrossing}}\ $ for an intersection in $\s$;

\item \emph{\bf Puncture-Skein Relation:} $\cp{\includegraphics[width=0.8cm]{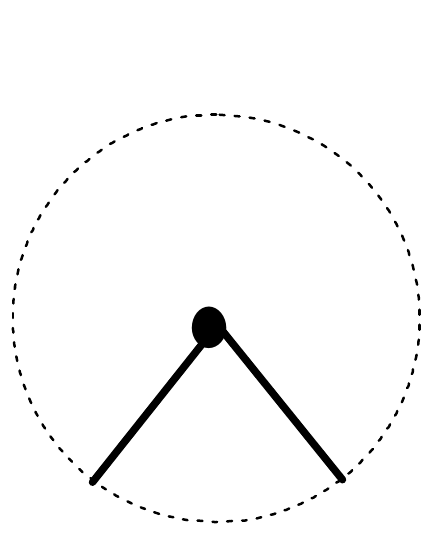}}\ =\ v^{-1}\big(\,\cp{\includegraphics[width=0.8cm]{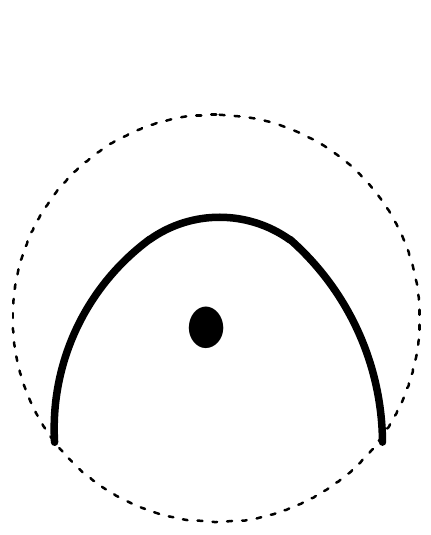}}\  +\ \cp{\includegraphics[width=0.8cm]{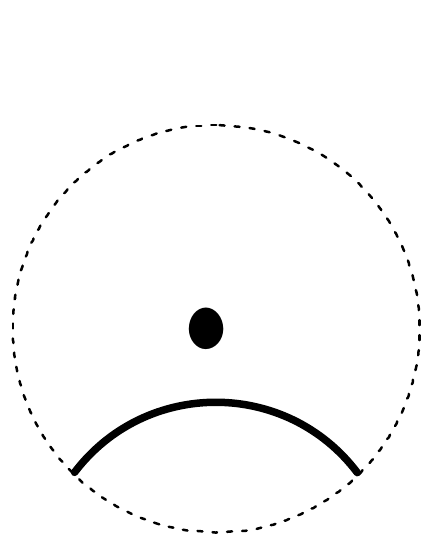}}\,\big)\ $ for an intersection at $v$;

\item \emph{\bf Framing Relation:} $\cp{\includegraphics[width=0.7cm]{trivial}}\ =\ -2$;

\item \emph{\bf Puncture Relation:} $\cp{\includegraphics[width=0.7cm]{curve-around-puncture}}\ =\ 2$.
\end{enumerate}
The product $\alpha\cdot\beta$ of two generalized curves $\alpha$ and $\beta$ is obtained by taking their union with unit the empty set.
\end{definition}

The fact that $\mathcal{C}(\s)$ is a well-defined commutative algebra follows from the same arguments as for $\AS_h(\s)$. These two algebras are related naturally as follows: let $p\colon\AS_h(\s)\to\mathcal{C}(\s)$ be the map which to a generalized link in $\s\times[0,1]$ with vertical framing associates its projection on $\s$. We also let $p(h)=0$ and let $p(v^{\pm})=v^{\pm}$ for each puncture $v$. Since $p$ maps relations (1)\,--\,(4) to the corresponding relations (1$'$)\,--\,(4$'$), and maps the stacking of generalized framed links in $\s\times[0,1]$ to the union of generalized curves on $\s$, it is a well-defined surjective $\mathbb{C}$--algebra homomorphism.

\begin{proposition}\label{prop:iso} The map $\overline{p}\colon\AS_h(\s)/h\AS_h(\s)\to\mathcal{C}(\s)$ induced by $p$ is an isomorphism of $\mathbb{C}$--algebras.
\end{proposition}

\begin{proof} Since $\AS_h(\s)\cong\mathcal{V}[[h]]$ is topologically free, each element $a\in \AS_h(\s)$ can be uniquely written as a power series $\sum a_kh^k$ with coefficients $a_i\in\mathcal{V}$. By the definition of $p$, we have $p(a)=p(a_0)$. Remember that the elements of $\mathcal{V}$ are diagrams without crossings either in $\s$ or at $V$, hence $p$ is injective on $\mathcal{V}$. Since $a_0\in\mathcal{V}$, we have $p(a_0)=0$ if and only if $a_0=0$. As a consequence, $\ker p=h\AS_h(\s)$ and $p$ induces a $\mathbb{C}$--algebra isomorphism $\overline{p}\colon\AS_h(\s)/h\AS_h(\s)\to\mathcal{C}(\s)$.
\end{proof}

In \cite{G}, Goldman defines a Lie bracket on the free algebra generated by free homotopy classes of curves on $\s$. It can be described in terms of resolutions of intersections and is of a purely topological nature. Generalizing this construction, we consider the \emph{Goldman bracket} on $\C(\s)$ to be the bilinear map $\{\ ,\ \}\colon\C(\s)\times \C(\s)\rightarrow \C(\s)$ defined as follows:
\begin{enumerate}[(1)]
\item for a puncture $v$ and a generalized curve $\alpha$, we let $\{v,\alpha\}=0$;

\item for two generalized curves $\alpha$ and $\beta$, we let
$$\{\alpha,\beta\}=\frac{1}{2}\sum_{p\in\alpha\cap\beta\cap\s}(\alpha_p\beta^+-\alpha_p\beta^-)+\frac{1}{4}\sum_{v\in\alpha\cap\beta\cap V}\frac{1}{v}(\alpha_v\beta^+-\alpha_v\beta^-).$$
\end{enumerate}

In the first sum, the \emph{positive resolution} $\alpha_p\beta^+$ of $\alpha$ and $\beta$ at $p$ is obtained by going along $\alpha$ toward $p$ then turning left at $p$ before going along $\beta$, and the \emph{negative resolution} $\alpha_p\beta^-$ is obtained similarly by turning right.
\begin{equation*}
_{\alpha}\includegraphics[width=0.9cm]{crossing-c}_{\beta}\quad\quad\quad\quad\includegraphics[width=0.9cm]{+crossing}\ \ _{\alpha_p\beta^+}\quad\quad\includegraphics[width=0.9cm]{ncrossing}\ \ _{\alpha_p\beta^-}
\end{equation*}

In the second sum, we introduce positive and negative resolutions of intersections at a puncture in a similar manner. However, we have to consider several cases depending on the number and the relative positions of the ends of $\alpha$ and $\beta$ meeting at $v$. Given an end of $\alpha$ and an end of $\beta$ at $v$, a positive resolution consists in going along the corresponding strand of $\alpha$ toward $v$ then turning left around $v$ before going along the strand of $\beta$. In this process the other ends of $\alpha$ and $\beta$, if any, are left untouched. The \emph{positive resolution} $\alpha_v\beta^+$ is then defined to be the sum of all the positive resolutions between the ends of $\alpha$ and the ends of $\beta$. The \emph{negative resolution} $\alpha_v\beta^-$ is defined accordingly by turning right around $v$. The possible configurations and resolutions are given in the table bellow.
\bigskip

\centerline{\begin{tabular}{c|rr} 
& \large$\alpha_v\beta^+$ & \large$\alpha_v\beta^-$ \\ 
\hline
& & \\
$\cp{\includegraphics[width=1cm]{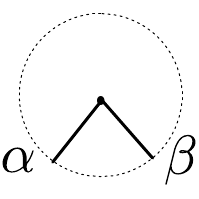}}$ & $\cp{\includegraphics[width=0.9cm]{2}}$ & $\cp{\includegraphics[width=0.9cm]{3}}$ \\
& & \\
$\cp{\includegraphics[width=1cm]{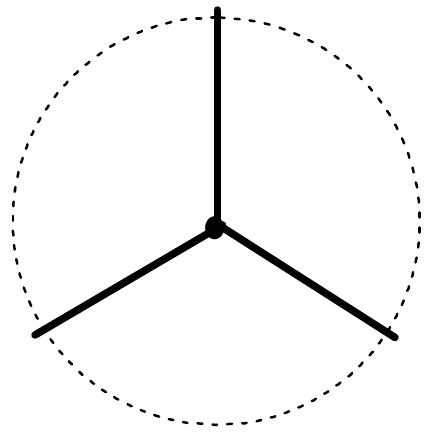}}$ & $\cp{\includegraphics[width=0.9cm]{c}}+\cp{\includegraphics[width=0.9cm]{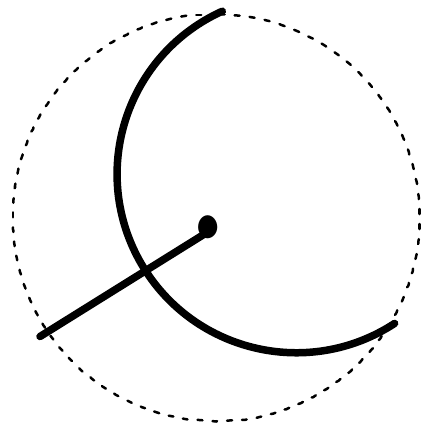}}$ & $\cp{\includegraphics[width=0.9cm]{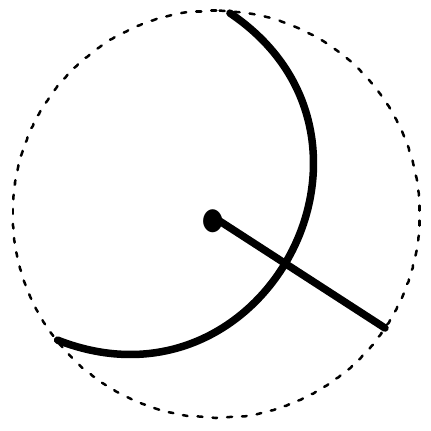}}+\cp{\includegraphics[width=0.9cm]{d}}$ \\
& & \\
$\cp{\includegraphics[width=1cm]{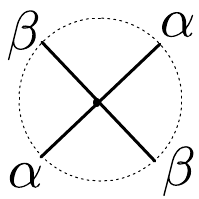}}$& $\cp{\includegraphics[width=0.9cm]{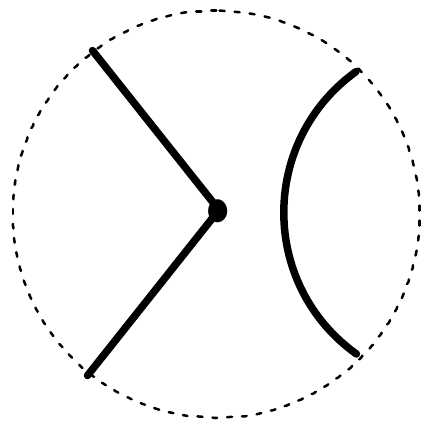}}+\cp{\includegraphics[width=0.9cm]{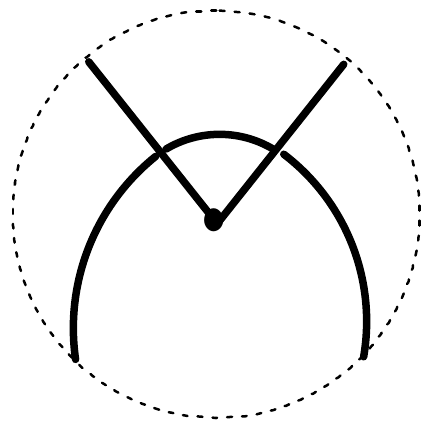}}+\cp{\includegraphics[width=0.9cm]{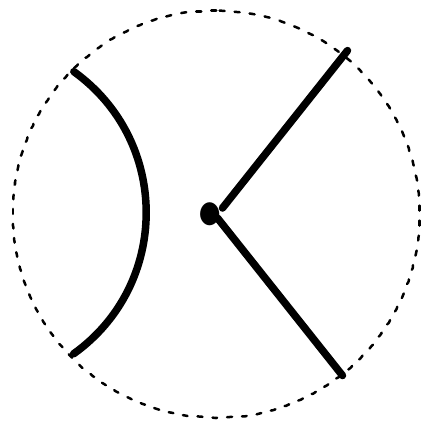}}+\cp{\includegraphics[width=0.9cm]{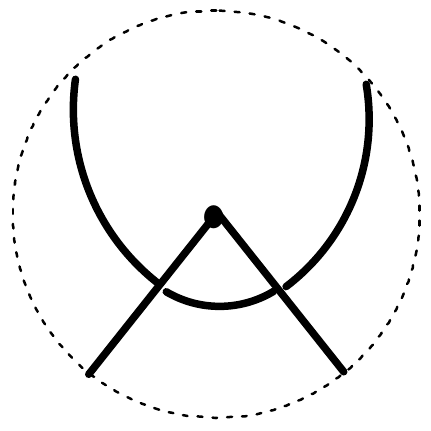}}$ & $\cp{\includegraphics[width=0.9cm]{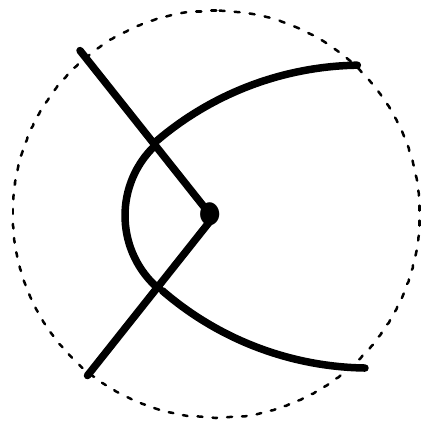}}+\cp{\includegraphics[width=0.9cm]{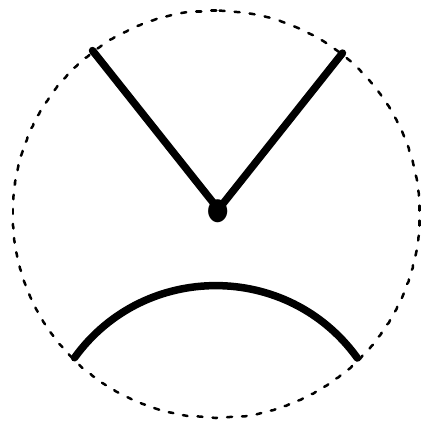}}+\cp{\includegraphics[width=0.9cm]{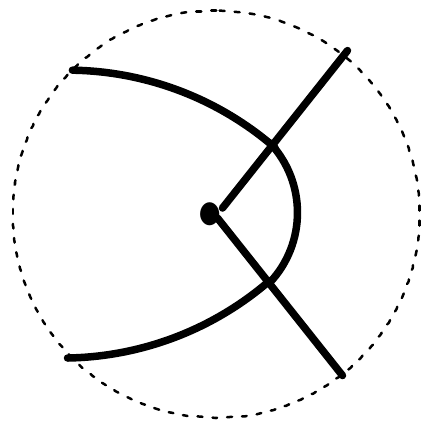}}+\cp{\includegraphics[width=0.9cm]{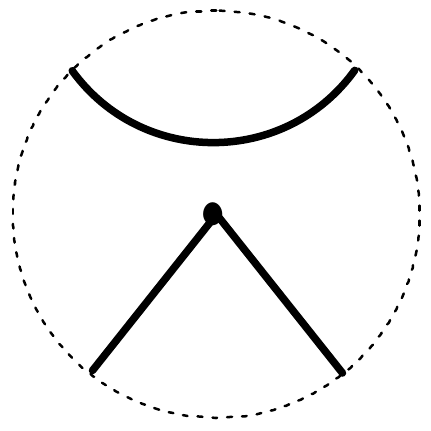}}$ \\ 
& & \\
$\cp{\includegraphics[width=1cm]{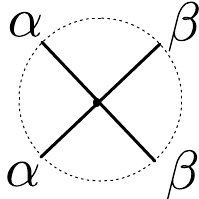}}$ & $\cp{\includegraphics[width=0.9cm]{34}}+\cp{\includegraphics[width=0.9cm]{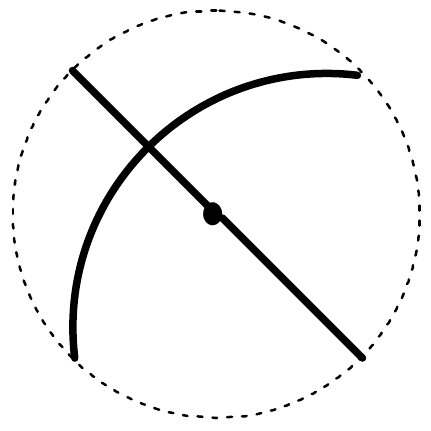}}+\cp{\includegraphics[width=0.9cm]{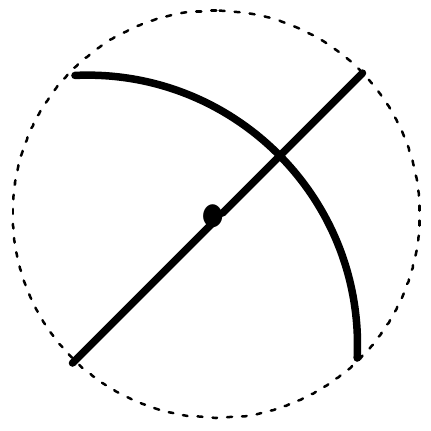}}+\cp{\includegraphics[width=0.9cm]{41}}$ & $\cp{\includegraphics[width=0.9cm]{40}}+\cp{\includegraphics[width=0.9cm]{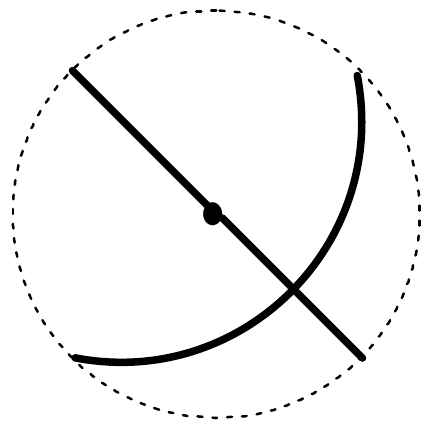}}+\cp{\includegraphics[width=0.9cm]{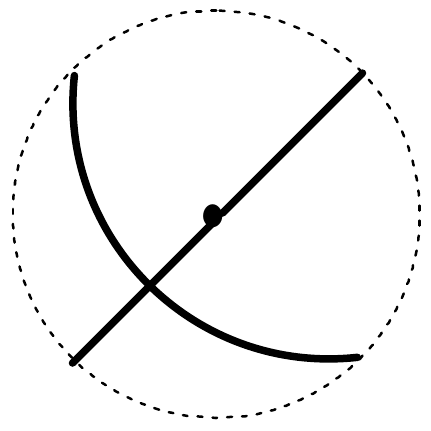}}+\cp{\includegraphics[width=0.9cm]{35}}$ \\ 
\end{tabular}}
\bigskip


Similarly to the skein algebra, most of the facts about this bracket rely on computations done locally around intersections. As such, diagrams of these types
\[
\Big\{\cp{\includegraphics[width=1cm]{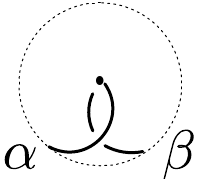}}\Big\}\hspace{2cm}
\Big\{\cp{\includegraphics[width=0.9cm]{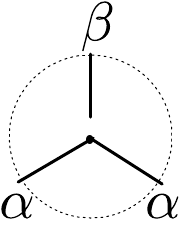}}\Big\}\]
will be used to denote the sum of the terms in $\{\alpha,\beta\}$ coming from their intersections in the dotted circle. Note in particular that the order in the bracket will be encoded using over and under crossings, even though the curves do intersect in $\overline{\s}$.

With these definitions, we have the following theorem.

\begin{theorem}\label{Poisson} The algebra $\big(\C(\s), \cdot, \{\ ,\ \}\big)$ is a well-defined Poisson algebra.
\end{theorem}

To prove Theorem \ref{Poisson}, we need the following lemma.

\begin{lemma}\label{lemma} The following identities hold in $\mathcal{C}(\s)$:
\begin{enumerate}[(1)]
\item $\cp{\includegraphics[width=0.9cm]{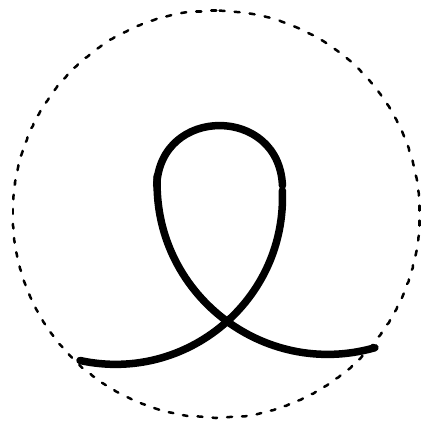}}=-\cp{\includegraphics[width=0.9cm]{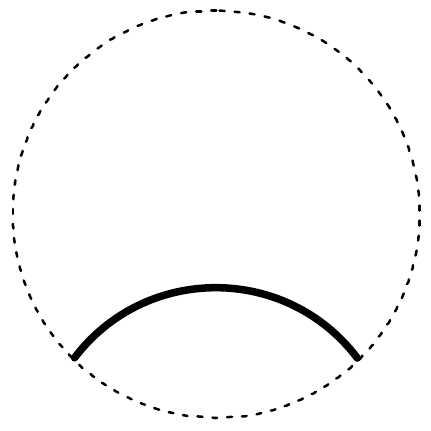}}\,$;

\item $\cp{\includegraphics[width=0.9cm]{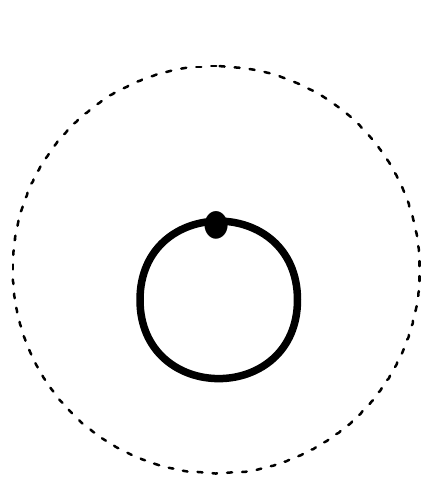}}=0\,$; 

\item $\cp{\includegraphics[width=0.9cm]{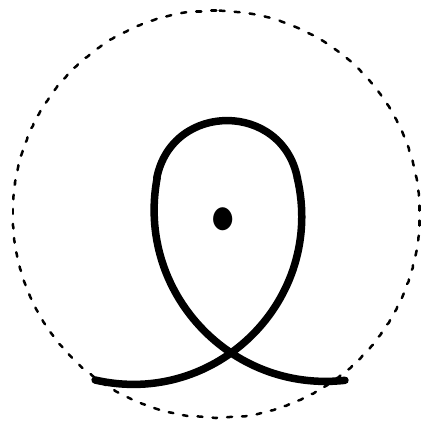}}=\cp{\includegraphics[width=0.9cm]{2}}+2\ \cp{\includegraphics[width=0.9cm]{3}}\;$.\hfill$\square$
\end{enumerate}
\end{lemma}

\begin{proof} This is a simple computation using the relations in Definition~\ref{def:curve}. As an illustration, for relation $(2)$ we have
\begin{equation*}
\cp{\includegraphics[width=0.9cm]{11}}\ =\ \frac{1}{v}\Big(\ \cp{\includegraphics[width=0.9cm]{13}}\ +\ \cp{\includegraphics[width=0.9cm]{12}}\ \Big) =\frac{1}{v}(2-2)=0.
\end{equation*}
\end{proof}

\begin{proof}[Proof of Theorem \ref{Poisson}] In order to verify the well-defintion, it suffices to show that $\{\ ,\ \}$ is invariant under Reidemeister Moves II, II$'$ and III and relations (1$'$)\,--\,(4$'$). The invariance under Reidemeister Moves II and III and relations (1$'$), (3$'$) and (4$'$) follows from the same arguments as in \cite{G}. For the invariance under Reidemeister Move II$'$, We have, with the pictorial conventions set earlier,
\begin{equation*}
\begin{split}
\Big\{\cp{\includegraphics[width=1cm]{4bis}}\Big\}
&=\,\frac{1}{2}\Big(\cp{\includegraphics[width=1cm]{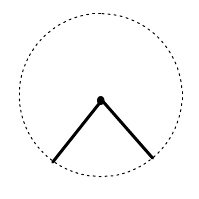}}-\cp{\includegraphics[width=0.9cm]{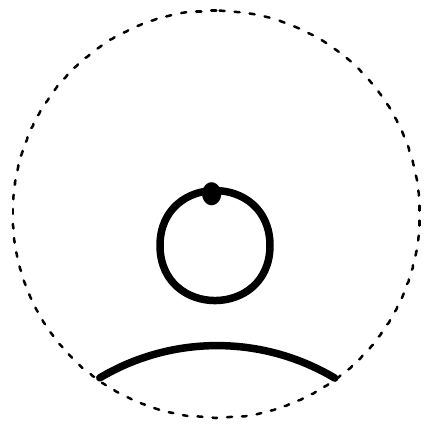}}\Big)+\frac{1}{4v}\Big(\cp{\includegraphics[width=0.9cm]{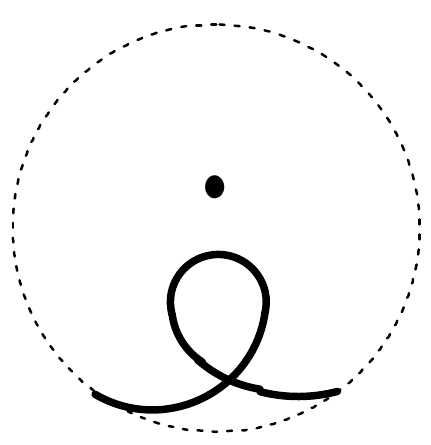}}- \cp{\includegraphics[width=0.9cm]{10}}\Big)\\
&=\,\frac{1}{2v}\Big(\cp{\includegraphics[width=0.9cm]{3}}+ \cp{\includegraphics[width=0.9cm]{2}}\Big)+\frac{1}{4v}\Big(-3\cp{\includegraphics[width=0.9cm]{3}}- \cp{\includegraphics[width=1cm]{2}}\Big)\\
&=\,\frac{1}{4v}\Big(\cp{\includegraphics[width=0.9cm]{2}}- \cp{\includegraphics[width=0.9cm]{3}}\Big)= \Big\{\cp{\includegraphics[width=0.9cm]{1}}\Big\},
\end{split}
\end{equation*}
where the second equality follows from Lemma \ref{lemma}. For the invariance under the puncture-skein relation (3$'$), we have to verify the following three cases:
\begin{enumerate}[(i)]
\item
$\Big\{\cp{\includegraphics[width=0.9cm]{24}}\Big\}\ =v^{-1}\Big(\Big\{\cp{\includegraphics[width=0.9cm]{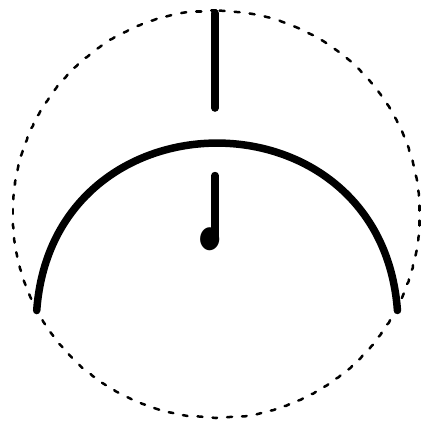}}\Big\}+\Big\{\cp{\includegraphics[width=0.9cm]{e}}\Big\}\Big)$,

\item 
$\Big\{\cp{\includegraphics[width=0.9cm]{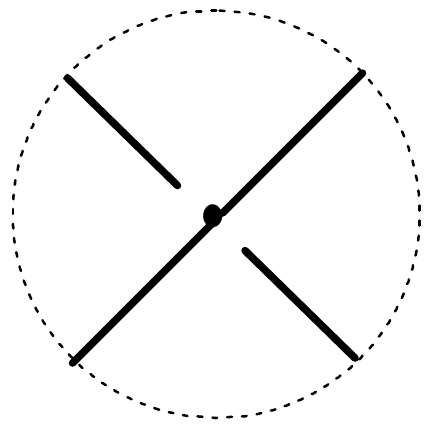}}\Big\}=v^{-1}\Big(\Big\{\cp{\includegraphics[width=0.9cm]{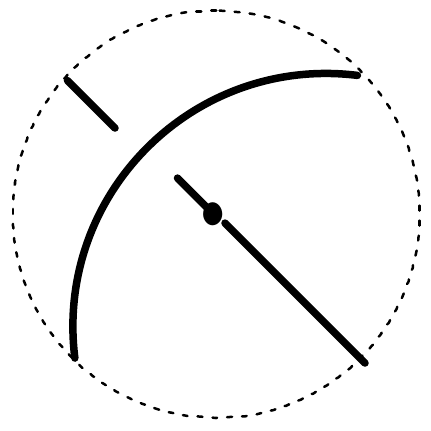}}\Big\}+\Big\{\cp{\includegraphics[width=0.9cm]{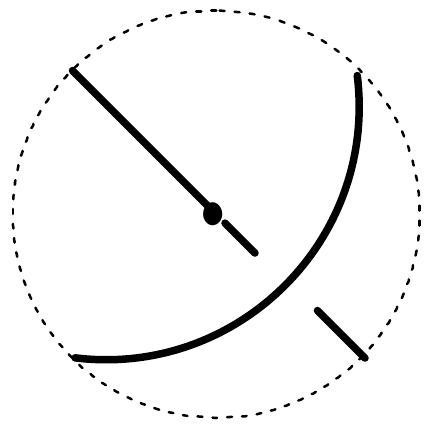}}\Big\}\Big)$,

\item
$\Big\{\cp{\includegraphics[width=0.9cm]{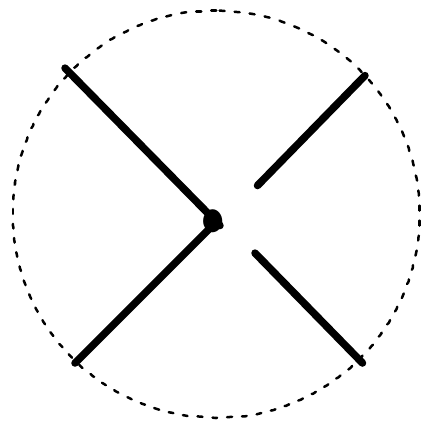}}\Big\}=v^{-1}\Big(\Big\{\cp{\includegraphics[width=0.9cm]{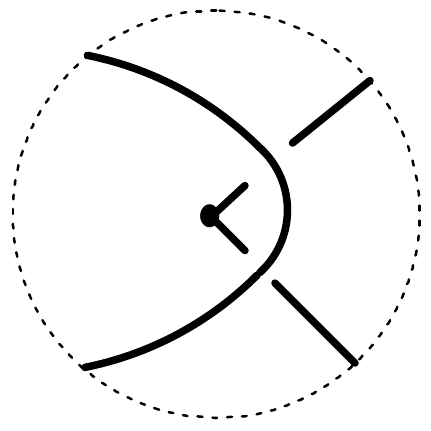}}\Big\}+ \Big\{\cp{\includegraphics[width=0.9cm]{37}}\Big\}\Big)$.
\end{enumerate}
For (i), if we apply to the left hand side the definition of the bracket then resolve the intersections in the surface using the skein relation, we obtain 
\begin{equation*}
\begin{split}
\Big\{\cp{\includegraphics[width=0.9cm]{24}}\Big\}
=&\,\frac{1}{4v}\Big(\cp{\includegraphics[width=0.9cm]{c}}+\cp{\includegraphics[width=0.9cm]{26}}-\cp{\includegraphics[width=0.9cm]{25}}-\cp{\includegraphics[width=0.9cm]{d}}\Big)\\
=&\,\frac{1}{4v}\Big(\cp{\includegraphics[width=0.9cm]{c}}+ \cp{\includegraphics[width=0.9cm]{e}}+ \cp{\includegraphics[width=0.9cm]{c}}-\cp{\includegraphics[width=0.9cm]{e}}-\cp{\includegraphics[width=0.9cm]{d}}-\cp{\includegraphics[width=0.9cm]{d}}\Big)\\
=&\,\frac{1}{2v}\Big(\cp{\includegraphics[width=0.9cm]{c}}- \cp{\includegraphics[width=0.9cm]{d}}\Big),
\end{split}
\end{equation*}
while by definition the second term on the right end side vanishes and so
\begin{equation*}
\frac{1}{v}\Big(\Big\{\cp{\includegraphics[width=0.9cm]{h}}\Big\}+\Big\{\cp{\includegraphics[width=0.9cm]{e}}\Big\}\Big)=\,\frac{1}{2v}\Big(\cp{\includegraphics[width=0.9cm]{c}}-\cp{\includegraphics[width=0.9cm]{d}}\Big).
\end{equation*}
Note that since the left hand side of (ii)  and (iii) differ by a Reidemeister Move II$'$, it suffices to verify either of them, which follows from a computation similar to that for (i) and is left to the reader. The anti-symmetry of $\{\ ,\ \}$ follows from the fact that $\alpha_x\beta^{\pm}=\beta_x\alpha^{\mp}$ for each $x\in\alpha\cap\beta$ either in the surface or at the punctures.

The verification of the Jacobi identity is in the spirit of Goldman\,\cite{G} separating the following two cases: 
\begin{enumerate}[(1)]
\item $\alpha\cap\beta\cap\gamma\cap V=\emptyset$, and 

\item $\alpha\cap\beta\cap\gamma\cap V\neq\emptyset$. 
\end{enumerate}
In case (1), we let $\alpha$, $\beta$ and $\gamma$ be three generalized curves on $\s$. We let $c(x,y)=\frac{1}{4}$ if $x$, $y\in \s$ and $c(x,y)=\frac{1}{16}x^{-1}y^{-1}$ if $x$, $y\in V$; and if only one of $x$ and $y$, say $x$, is a puncture of $\s$, we let $c(x,y)=\frac{1}{8}x^{-1}$. Then we have
\begin{equation*}
\begin{split}
&\{\{\alpha,\beta\},\gamma\}\\
=&\sum_{\bf \scriptsize\begin{split}x&\in\alpha\cap\beta\\y&\in\beta\cap\gamma\end{split}}c(x,y)\big((\alpha_x\beta^+)_y\gamma^+-(\alpha_x\beta^+)_y\gamma^--(\alpha_x\beta^-)_y\gamma^++(\alpha_x\beta^-)_y\gamma^-\big)\\
+&\sum_{\bf \scriptsize\begin{split}x&\in\alpha\cap\beta\\z&\in\gamma\cap\alpha\end{split}}c(x,z)\big((\alpha_x\beta^+)_z\gamma^+-(\alpha_x\beta^+)_z\gamma^--(\alpha_x\beta^-)_z\gamma^++(\alpha_x\beta^-)_z\gamma^-\big),
\end{split}
\end{equation*}
and
\begin{equation*}
\begin{split}
&\{\{\beta,\gamma\},\alpha\}\\
=&\sum_{\bf \scriptsize\begin{split}y&\in\beta\cap\gamma\\z&\in\gamma\cap\alpha\end{split}}c(y,z)\big((\beta_y\gamma^+)_z\alpha^+-(\beta_y\gamma^+)_z\alpha^--(\beta_y\gamma^-)_z\alpha^++(\beta_y\gamma^-)_z\alpha^-\big)\\
+&\sum_{\bf \scriptsize\begin{split}y&\in\beta\cap\gamma\\x&\in\alpha\cap\beta\end{split}}c(y,x)\big((\beta_y\gamma^+)_x\alpha^+-(\beta_y\gamma^+)_x\alpha^--(\beta_y\gamma^-)_x\alpha^++(\beta_y\gamma^-)_x\alpha^-\big).
\end{split}
\end{equation*}
By definition, we have that $(\alpha_x\beta^+)_y\gamma^+=(\beta_y\gamma^+)_x\alpha^-$, $
(\alpha_x\beta^+)_y\gamma^-=(\beta_y\gamma^-)_x\alpha^-$, $
(\alpha_x\beta^-)_y\gamma^+=(\beta_y\gamma^+)_x\alpha^+$ and $(\alpha_x\beta^-)_y\gamma^-=(\beta_y\gamma^-)_x\alpha^+$ for each $x\in\alpha\cap\beta$ and $y\in\beta\cap\gamma$, so the summands in the first row of the expansion of $\{\{\alpha,\beta\},\gamma\}$ cancel out the summands in the second row of the expansion of $\{\{\beta,\gamma\},\alpha\}$. Similarly, the summands in the second row of the expansion of $\{\{\alpha,\beta\},\gamma\}$ and the first row of the expansion of $\{\{\beta,\gamma\},\alpha\}$ cancel out the summands in the expansion of $\{\{\gamma,\alpha\},\beta\}$. Hence $\{\{\alpha,\beta\},\gamma\}+\{\{\beta,\gamma\},\alpha\}+\{\{\gamma,\alpha\},\beta\}=0$. In case (2), we let $v\in\alpha\cap\beta\cap\gamma$. If $v$ is a self-intersection of one of $\alpha$, $\beta$ or $\gamma$, say $\alpha$, then by the well-definition of $\{\ ,\ \}$ we can resolve $\alpha$ at $v$ to reduce to case (1). If $v$ is a self-intersection of none of $\alpha$, $\beta$ or $\gamma$, then we may without loss of generality assume that $\alpha$, $\beta$ and $\gamma$ are counterclockwise ordered at $v$. Then all the summands in $\{\{\alpha,\beta\},\gamma\}+\{\{\beta,\gamma\},\alpha\}+\{\{\gamma,\alpha\},\beta\}$ cancel out in pairs as in case (1) except three summands around $v$ which are from $\frac{1}{4}v^{-1}\{\alpha_v\beta^+,\gamma\}$, $\frac{1}{4}v^{-1}\{\beta_v\gamma^+,\alpha\}$ and $\frac{1}{4}v^{-1}\{\gamma_v\alpha^+,\beta\}$ respectively; and for the sum of them, we have
\begin{equation*}
\begin{split}
&\Big\{\cp{\includegraphics[width=0.9cm]{b}}\Big\}+\Big\{\cp{\includegraphics[width=0.9cm]{h}}\Big\}+\Big\{\cp{\includegraphics[width=0.9cm]{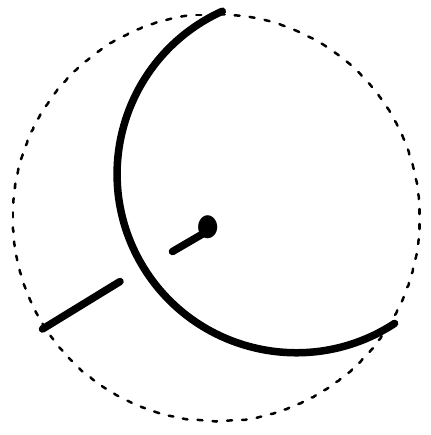}}\Big\}\\
=&\, \frac{1}{2}\Big(\cp{\includegraphics[width=0.9cm]{d}}-\cp{\includegraphics[width=0.9cm]{e}}\Big)+\frac{1}{2}\Big(\cp{\includegraphics[width=0.9cm]{c}}-\cp{\includegraphics[width=0.9cm]{d}}\Big)+\frac{1}{2}\Big(\cp{\includegraphics[width=0.9cm]{e}}-\cp{\includegraphics[width=0.9cm]{c}}\Big)=0.
\end{split}
\end{equation*}
The Leibniz rule follows directly from the definition and the fact that $(\alpha\cdot\beta)\cap\gamma=(\alpha\cup\beta)\cap\gamma=(\alpha\cap\gamma)\cup(\beta\cap\gamma)$.
\end{proof}

Following the approach of \cite{BFK} (see also \cite{KS}), we recall that a topologically free $\mathbb{C}[[h]]$--algebra $A_h$ is called a \emph{(formal) quantization} of a Poisson algebra $A$ if there is a $\mathbb{C}$--algebra isomorphism $\Theta\colon A_h/hA_h\to A$ such that $$\Theta\Big(\frac{\bar\alpha\cdot\bar\beta-\bar\beta\cdot\bar\alpha}{h}\Big)=\{\alpha,\beta\}$$
for any $\bar\alpha\in\Theta^{-1}(\alpha)$ and $\bar\beta\in\Theta^{-1}(\beta)$. Using the isomorphism from Proposition~\ref{prop:iso}, we obtain the following theorem.

\begin{theorem}\label{dequantization} The $\mathbb{C}[[h]]$--algebra $\AS_h(\s)$ is a quantization of $\C(\s)$ via the $\mathbb{C}$--algebra isomorphism $\overline{p}$.
\end{theorem}

\begin{proof} This is an analogue of the arguments in \cite{BFK}. Given a diagram on the surface, we let $p^{\pm}(\mathcal{S})$ respectively be the number of positive and negative resolutions in the surface used to obtain the state $\mathcal{S}$, and let $v^{\pm}(\mathcal{S})$ respectively be the number of positive and negative resolutions at the punctures used to obtain $\mathcal{S}$. If none of $\alpha$ and $\beta$ has a self-intersection at the punctures, then keeping track of the crossings, we have
$$\{\alpha,\beta\}=\sum_{\mathcal{S}}\Big(\frac{1}{2}\big(p^+(\mathcal{S})-p^-(\mathcal{S})\big)+\frac{1}{4}\big(v^+(\mathcal{S})-v^-(\mathcal{S})\big)\Big)\mathcal{S,}$$
where the summation is taken over all states $\mathcal{S}$ obtained from resolving $\alpha\cup\beta$, and
\begin{equation*}
\begin{split}
\bar{\alpha}\cdot\bar{\beta}-\bar{\beta}\cdot\bar{\alpha}=\sum_{\mathcal{S}}\Big(&q^{(p^+(\mathcal{S})-p^-(\mathcal{S}))+\frac{1}{2}(v^+(\mathcal{S})-v^-(\mathcal{S}))}\\
&-q^{-(p^+(\mathcal{S})-p^-(\mathcal{S}))-\frac{1}{2}(v^+(\mathcal{S})-v^-(\mathcal{S}))}\Big)\mathcal{S},
\end{split}
\end{equation*}
in which the coefficient of $h$ is exactly $\{\alpha,\beta\}$. If one of $\alpha$ or $\beta$, say $\alpha$, has a self-intersection at a puncture $v\in V$, then we let $\alpha_1$ and $\alpha_2$ be the resolutions of $\alpha$ at $v$. Let ${\sigma}_1$ and ${\sigma}_2$ respectively be the set of states obtained by resolving $\alpha_1\cup\beta$ and $\alpha_2\cup\beta$, and let $n_{\mathcal{S}}=\frac{1}{2}(p^+(\mathcal{S})-p^-(\mathcal{S}))+\frac{1}{4}(v^+(\mathcal{S})-v^-(\mathcal{S}))$. By the previous calculation, we have  
\begin{equation*}
\begin{split}
\{\alpha,\beta\}=&\,\frac{1}{v}\big(\{\alpha_1,\beta\}+\{\alpha_2,\beta\}\big)=\,\frac{1}{v}\sum_{\mathcal{S}\in{\sigma}_1\cup\sigma_2}n_{\mathcal{S}}\,\mathcal{S}.
\end{split}
\end{equation*}
If $\bar{\alpha}_1$ and $\bar{\alpha}_2$ respectively are the positive and negative resolutions of $\bar{\alpha}$ at $v$, then by the puncture-skein relation $\bar{\alpha}=v^{-1}(q^{\frac{1}{2}}\bar{\alpha}_1+q^{-\frac{1}{2}}\bar{\alpha}_2)$, and we have that
\begin{equation*}
\begin{split}
\bar{\alpha}\cdot\bar{\beta}-\bar{\beta}\cdot\bar{\alpha}=&\,\frac{1}{v}\Big(q^{\frac{1}{2}}\big(\bar{\alpha}_1\cdot\bar{\beta}-\bar{\beta}\cdot\bar{\alpha}_1\big)+q^{-\frac{1}{2}}\big(\bar{\alpha}_2\cdot\bar{\beta}-\bar{\beta}\cdot\bar{\alpha}_2\big)\Big)\\
=&\,\frac{1}{v}\Big(q^{\frac{1}{2}}\sum_{\mathcal{S}\in{\sigma}_1}\big(q^{2n_{\mathcal{S}}}-q^{-2n_{\mathcal{S}}}\big)\mathcal{S}+q^{-\frac{1}{2}}\sum_{\mathcal{S}\in{\sigma}_2}\big(q^{2n_{\mathcal{S}}}-q^{-2n_{\mathcal{S}}}\big)\mathcal{S}\Big)
\end{split}
\end{equation*}
in which the coefficient of $h$ is $\{\alpha,\beta\}$.
\end{proof}

In particular, the proof of Theorem~\ref{dequantization} explains the relationship between the coefficients $q^{\pm\frac{1}{2}}$ in the puncture-skein relation used in the definition of $\AS_h(\s)$ and the coefficient $\frac{1}{4}$ in front of the puncture terms in the Goldman bracket on $\C(\s)$. Both of these choices were essential at some point in the well-definition of $\AS_h(\s)$ and $\{\ ,\ \}$ and turn out to be related to the geometric aspects of the theory described in the next section.



\section{Relationship with hyperbolic geometry}\label{section:3}

An essential aspect of the skein algebra is its relationship with the $SL_2$--character variety $\X(\s)$. A first step in this direction can be found in the work of Turaev\,\cite{Turaev2} and the full picture was unraveled by the work of Bullock, Frohman and  Kania-Bartoszy\'nska \cite{Bu,BFK,BFK2} and Przytycki and Sikora\,\cite{PS}. In our context, the corresponding framework will be that of the decorated Teichm\"uller space and the notion of $\lambda$-lengths, which as we will see can be understood as generalized trace functions.

\subsection{The decorated Teichm\"uller space and its Poisson structure}

As before, we let $\s$ be a surface with a set of punctures $V=\{v_1,\ldots,v_s\}$. In order to work in the hyperbolic setting, we suppose in addition that $\chi(\s)<0$. We consider the \emph{cusped Teichm\"uller space} $\T_c(\s)$ defined as the set of isotopy classes of complete hyperbolic metrics on $\s$ with finite area. A \emph{decoration} $r\in\R^V_{>0}$ is given by a choice of a positive real number $r_i=r(v_i)$ associated to each puncture. Geometrically, given a metric $m\in\T_c(\s)$, a decoration should be interpreted as a choice of a horocycle of length $r_i$ at each puncture $v_i$ of $\s$. The \emph{decorated Teichm\"uller space} $\T^d(\s)$, introduced by Penner in\,\cite{Penner1}, is then defined to be the space of \emph{decorated hyperbolic metrics} $(m,r)$, $m\in\T_c(\s)$. Topologically, it is the fiber bundle $\T^d(\s)=\T_c(\s)\times\R^V_{>0}$ over the cusped Teichm\"uller space.

One of the reasons for introducing decorations is to be able to define a notion of length of an arc between punctures. More precisely, Let $\alpha$ be an arc between two punctures of $\s$, possibly with self-intersections. Given a decorated hyperbolic metric $(m,r)$, consider a geodesic lift $\widetilde{\alpha}$ of $\alpha$ to the universal cover $\mathbb{H}^2$ of $(\s,m)$. The \emph{length} $l(\alpha)$ of $\alpha$ for $(m,r)$ is then defined to be the signed length of the segment of $\widetilde{\alpha}$ between the horocycles given by the decoration, where the sign is chosen to be positive if the horocycles do not intersect and negative if they do. A number of properties concerning lengths of arcs are in fact best expressed in terms of the associated \emph{$\lambda$-length} $\lambda(\alpha)=e^\frac{l(\alpha)}{2}$. In particular, if $a$, $b$, $c$ and $d$ are the consecutive sides of a square in $\s$ and $e$ and $e'$ are its diagonals, they satisfy the Ptolemy relation\,\cite{Penner1}
\[\lambda(e)\lambda(e')=\lambda(a)\lambda(c)+\lambda(b)\lambda(d).\]

Let $T$ be an \emph{ideal triangulation} of $\s$, that is, a maximal collection of isotopy classes of simple arcs between punctures in $\s$ which decomposes the surface into \emph{ideal triangles}. We let $E$ be the set of edges of $T$. If $\s$ is a sphere with at least three punctures or a surface of genus $g>0$ with at least one puncture then such a triangulation exists. In this case, the associated lengths $l(e)$, $e\in E$, form a coordinate system on $\T^d(\s)$. In these coordinates, Mondello\,\cite{Mondello} introduced a Poisson bi-vector field on $\T^d(\s)$ defined as follows: on a decorated hyperbolic surface $\s$, given an end of an edge $\alpha$ and an end of an edge $\beta$ meeting at a puncture $v$, we  define the \emph{generalized angle} from the end of $\alpha$ to the end of $\beta$ to be the length of the horocycle segment between them going in the positive direction for the orientation of $\s$. Then we let $\theta_v$ be the sum of the generalized angles from each end of $\alpha$ to each end of $\beta$ meeting at $v$, and we let $\theta'_v$ be the sum of the generalized angles from the ends of $\beta$ to the ends of $\alpha$ (see the figure below for an example).
\begin{figure}[htbp]\centering
\includegraphics[width=3cm]{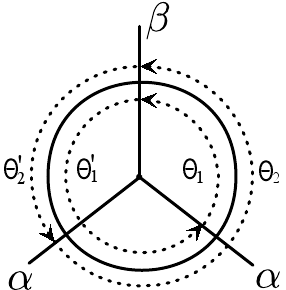}\\
\caption{here $\theta_v=\theta_1+\theta_2$ and $\theta'_v=\theta'_1+\theta'_2$.}
\end{figure}

We then consider the following bi-vector field
$$\Pi_{WP}=\frac{1}{4}\sum_{v\in V}\sum_{\substack{\alpha,\beta\in E\\ \alpha\cap\beta=v}}\frac{\theta'_v-\theta_v}{r(v)}\frac{\partial}{\partial l(\alpha)}\wedge\frac{\partial}{\partial l(\beta)}$$
on the decorated Teichm\"uller space. It is a Poisson bi-vector field which is directly related to the pull-back of the Weil-Petersson symplectic structure on $\T_c(\s)$ as described by Penner in \cite{Penner2} (see Prop 4.7 in \cite{Mondello}) and as such we call it the \emph{Weil-Petersson Poisson bi-vector} on $\T^d(\s)$. It can also be shown by a direct computation that this bi-vector is invariant under a diagonal exchange, and as a consequence is independent of the choice of the ideal triangulation $T$.

If $\alpha$ is a loop on $\s$ and $m\in\T_c(\s)$, we consider the quantity $\lambda(\alpha)=2\cosh\frac{l(\alpha)}{2}$ where $l(\alpha)$ is the length of the geodesic representative of $\alpha$ in $m$. Up to a sign, it is equal to the trace $tr(\rho(\alpha))$ of the monodromy representation $\rho\colon\pi_1(\s)\to PSL_2(\R)$ associated to $m$. We purposefully used the same notations as for $\lambda$-lengths and call $\lambda(\alpha)$ the \emph{generalized trace} of $\alpha$, where $\alpha$ can be an arc or a loop on $\s$.

The goal of this section is to construct a map from the algebra of curves $\C(\s)$ to the algebra of functions over $\T^d(\s)$ by associating to a generalized curve the product of the generalized traces of its components. One issue, however, is the fact that elements of $\C(\s)$ are not identified up to Reidemeister Move I, and as a consequence we need to introduce the following definition.

\begin{definition}  For a curve with a self-intersection as on the left of Figure \ref{curl-killing}, we call the topologically trivial loop $c$ a \emph{curl} and the intersection point $p$ its \emph{vertex}. If $\alpha$ is an arc or a non-null-homotopic  loop on $\s$, then the \emph{curling number} $c(\alpha)$ of $\alpha$ is the maximal number of \emph{one way Reidemeister Moves I} (Figure \ref{curl-killing}) that $\alpha$ carries. If $\alpha$ is a null-homotopic loop, then $c(\alpha)$ is defined to be the same number plus 1.
\end{definition}

\begin{figure}[htbp]\centering
\includegraphics[width=5cm]{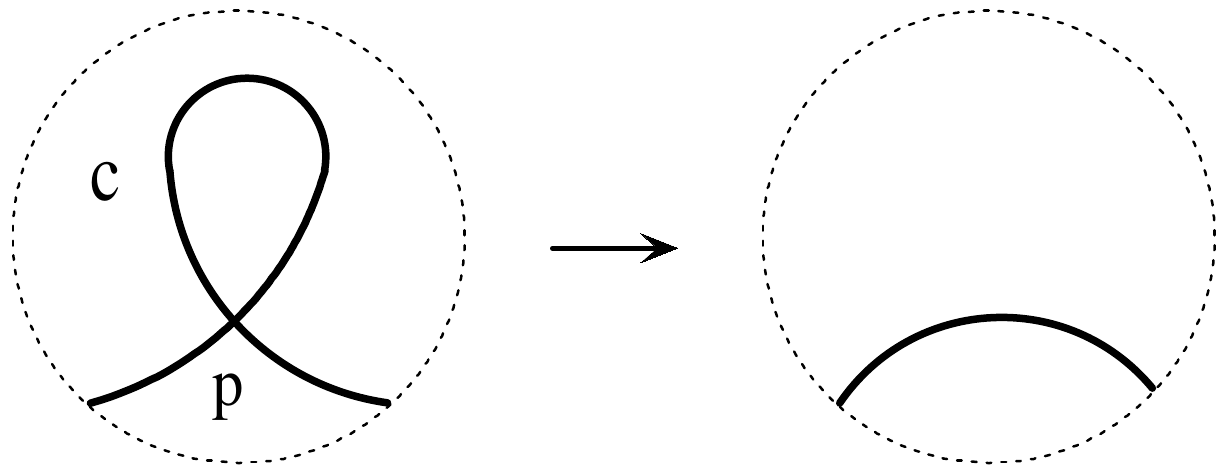}\\
\caption{One way Reidemeister Move I.}\label{curl-killing}
\end{figure}

\begin{example}By definition
$c(\bigodot)=0$ and $c(\bigcirc)=1$.
\end{example}

\begin{example} Since a geodesic minimizes self-intersections, its curling number is necessarily 0.
\end{example}

Using this definition, if $\alpha=\alpha_1\cup\cdots\cup\alpha_n$ is a generalized curve, that is, a union of equivalence classes of arcs and loops on $\s$, then we let $c(\alpha)=\sum_ic(\alpha_i)$ and $\lambda(\alpha)=\prod_i\lambda(\alpha_i)$. We recall that, if $(m,r)$ is a decorated hyperbolic metric, then $r(v)$ denotes the length of the horocycle at the puncture $v$ of $\s$.

\begin{theorem}\label{main} The map $\Phi\colon\C(\s)\rightarrow C^{\infty}(\T^d(\s))$ defined on the generators by 
$\Phi(v)=r(v)$ if $v$ is a puncture and
$\Phi(\alpha)=(-1)^{c(\alpha)}\lambda(\alpha)$ if $\alpha $ is a generalized curve is a well-defined Poisson algebra homomorphism with respect to the Goldman bracket $\{\ ,\ \}$ on $\C(\s)$ and the Weil-Petersson Poisson bracket on $C^{\infty}(\T^d(\s))$ associated to the bi-vector field $\Pi_{WP}$.
\end{theorem}

The remainder of this article will be dedicated to the proof of this theorem. The first step, done in Section~\ref{lengths}, will be to derive a series of lengths identities in hyperbolic geometry which generalize the Ptolemy relation, the trace identity and Wolpert's cosine formula for the Weil-Petersson Poisson bracket of length functions. In Section~\ref{algebra}, Together with an analysis of the behavior of the curling number under resolutions, half of these identities will be combined into generalized trace identities which imply that the map $\Phi$ is an algebra homomorphism. Finally, in Section~\ref{Poissonalg}, combined with a lemma about the expression of generalized trace functions in terms of the $\lambda$-lengths associated to the edges of a fixed ideal triangulation, the other half will be used to show that this homomorphism respects the Poisson structures.

\subsection{The lengths identities}\label{lengths}

In this section we are going to derive a series of identities involving geodesic lengths of curves and arcs between horocycles which are the heart of the proof of Theorem~\ref{main}. They rely on a set of ``cosine laws'' for various types of generalized hyperbolic triangles which can be found in Appendix A of \cite{GL}. We will use the results and notations found in this paper throughout this section. Some ``twisted versions'' of these laws are also needed and are included in Appendix \ref{A} of this paper. In Lemma~\ref{3.1} through \ref{3.4}, we study the relationship between the lengths of two intersecting curves $\alpha$ and $\beta$, the lengths of their possible resolutions and the (generalized) angle from $\alpha $ to $\beta$. There are several cases depending on whether $\alpha$ or $\beta$ is an arc or a loop and whether the intersection happens inside of $\s$ or at the puncture. Similarly, in Lemma~\ref{3.6} through \ref{3.8} we study the relationship between the length of a curve $\alpha$ with a self-intersection and the lengths of its possible resolutions. A complication here comes from the behavior of the curling number of the resolutions of a geodesic curve at a self-intersection. This is treated in Lemma~\ref{lemma:nonzero}.

Throughout this section, we will fix a decorated hyperbolic metric $(m,r)\in\T^d(\s)$. We recall that if $\alpha$ and $\beta$ are two geodesics on $\s$ for $(m,r)$, then the angle from $\alpha$ to $\beta$ at $p\in\alpha\cap\beta$ in $\s$ is the angle measured from $\alpha$ to $\beta$ for the orientation of $\s$, and the generalized angle from $\alpha$ to $\beta$ at $v\in\alpha\cap\beta$, $v\in V$ is the length of the horocycle segment measured from $\alpha$ to $\beta$ for the orientation of $\s$.

We start with the case of two loops intersecting in $\s$. The following formulae are well-known, Part (1) is the trace identity and Part (2) can be interpreted as Wolpert's cosine formula applied to trace functions. We nonetheless give a proof of these formulae for completeness.

\begin{lemma}\label{3.1} Let $\alpha$ and $\beta$ be two closed geodesics of lengths $a$ and $b$, and let $\theta$ be the angle from $\alpha$ to $\beta$ at $p\in \alpha\cap\beta$. If $x$ and $y$ are the lengths of the geodesic representatives of $\alpha_p\beta^+$ and $\alpha_p\beta^-$ respectively, then we have
\begin{enumerate}[(1)]
\item $\cosh\frac{x}{2}+\cosh\frac{y}{2}=2\cosh\frac{a}{2}\cosh\frac{b}{2}, $

\item $\cosh\frac{x}{2}-\cosh\frac{y}{2}=2\sinh\frac{a}{2}\sinh\frac{b}{2}\cos\theta.$
\end{enumerate}
\end{lemma}

\begin{figure}[htbp]\centering
\includegraphics[width=10.5cm]{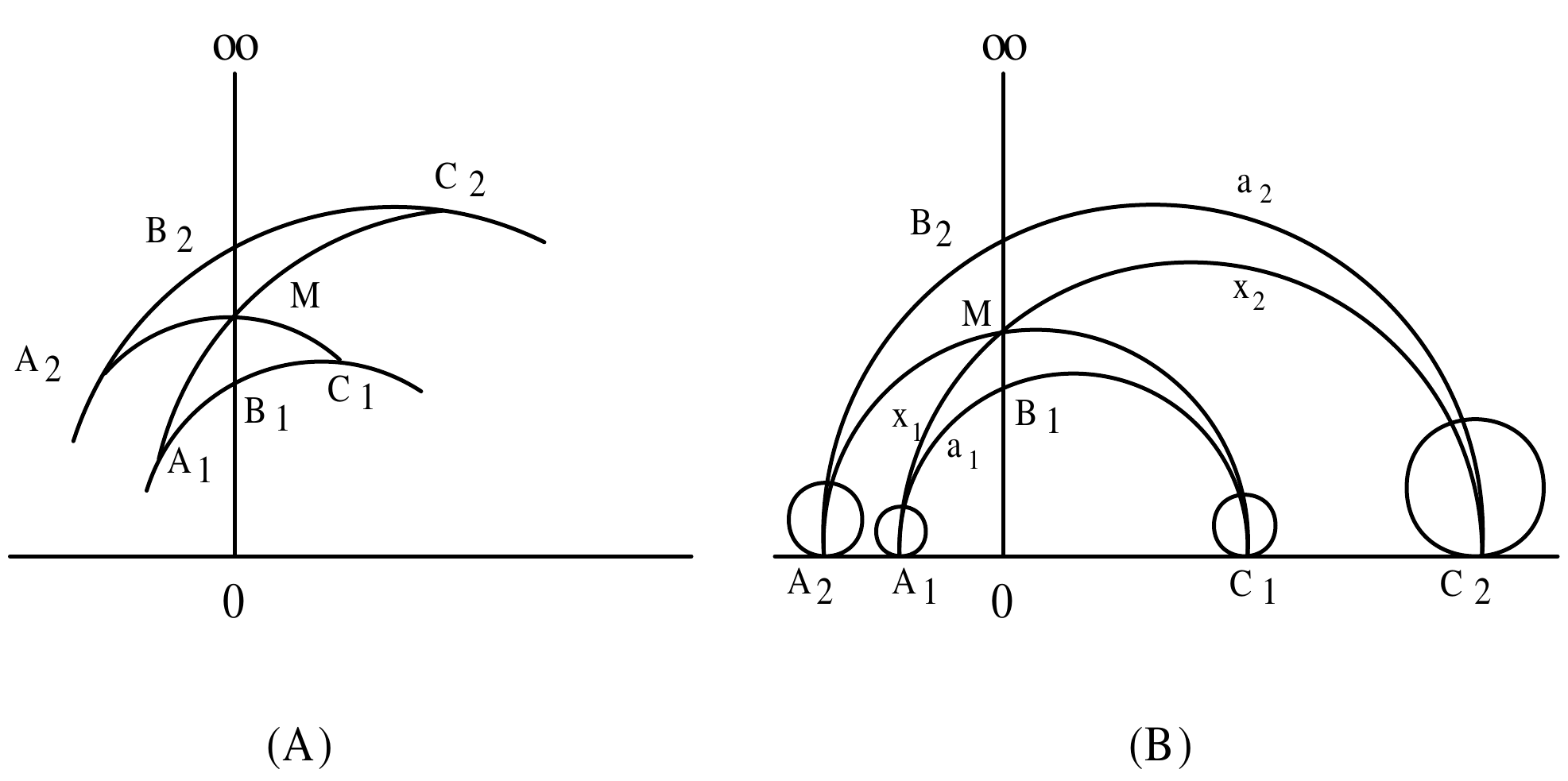}\\
\caption{}\label{curve-curve}
\end{figure}

\begin{proof} We consider Figure \ref{curve-curve} (A). Let $\overline{0\infty}$ be a lift of the geodesic $\beta$ in the universal cover $\mathbb{H}^2$ of $\s$. Let $\{B_i\}_{i\in\mathbb{Z}}$ be the lifts of $p$ on $\overline{0\infty}$ so that $|B_iB_{i+1}|=b$. Let $A_i$ and $C_i$ for $i=1,2$ be the points on the lift of the geodesic $\alpha$ passing through $B_i$ so that $|A_iB_i|=|B_iC_i|=\frac{a}{2},$ hence $|A_iC_i|=a$ and $\overline{A_iC_i}$ is a lift of $\alpha$ for $i=1,2$. Now take the mid-point $M$ of $\overline{B_1B_2}$, and connect $A_1$ and $C_2$ to $M$ by geodesics. Since $|A_1B_1|=|B_2C_2|=\frac{a}{2}$ and $|B_1M|=|MB_2|=\frac{b}{2}$, and $\angle A_1B_1M=\angle MB_2C_2=\pi-\theta$, the triangles $A_1B_1M$ and $MB_2C_2$ are isometric, hence the anlges $\angle A_1MB_1=\angle B_2MC_2$ and $\angle B_1A_1M=\angle MC_2B_2$. Therefore, the points $A_1$, $M$ and $C_2$ are on the geodesic representing a lift of $\alpha_p\beta^+$, and $|A_1M|=|MC_2|=\frac{x}{2}$. By the same argument, we have that $A_2$, $M$ and $C_1$ are on a lift of $\alpha_p\beta^-$ and $|A_2M|=|MC_1|=\frac{y}{2}.$ Applying the cosine law to the triangles $A_1B_1M$ and $A_2B_2M$ respectively, we have \[\cos(\pi-\theta)=\frac{-\cosh\frac{x}{2}+\cosh\frac{a}{2}\cosh\frac{b}{2}}{\sinh\frac{a}{2}\sinh\frac{b}{2}}\ \text{ and }\ \cos\theta=\frac{-\cosh\frac{y}{2}+\cosh\frac{a}{2}\cosh\frac{b}{2}}{\sinh\frac{a}{2}\sinh\frac{b}{2}}\,.\]
 Since $\cos(\pi-\theta)=-\cos\theta$, the sum of the two equalities implies Part (1) and the difference implies Part (2).
\end{proof}

\begin{lemma}\label{3.2} Let $\alpha$ be a geodesic arc of length $a$ and let $\beta$ be a closed geodesic of length $b$. Let $\theta$ be the angle from $\alpha$ to $\beta$ at $p\in \alpha\cap\beta$.  If $x$ and $y$ are the lengths of geodesic representatives of the arcs $\alpha_p\beta^+$ and $\alpha_p\beta^-$ respectively, then we have
\begin{enumerate}[(1)]
\item $e^{\frac{x}{2}}+e^{\frac{y}{2}}=2e^{\frac{a}{2}}\cosh\frac{b}{2},$
\item $e^{\frac{x}{2}}-e^{\frac{y}{2}}=2 e^{\frac{a}{2}}\sinh\frac{b}{2}\cos\theta\,.$
\end{enumerate}
\end{lemma}

\begin{proof} Let us look at Figure \ref{curve-curve} (B). Let $\overline{0\infty}$ be a lift of $\beta$ in the universal cover $\mathbb{H}^2$. Let $\{B_i\}_{i\in\mathbb{Z}}$ be the lifts of $p$ on $\overline{0\infty}$ so that $|B_i,B_{i+1}|=b$, and let $A_i$ and $C_i$ for $i=1,2$ be the end points of the lifts of $\beta$ passing through $B_i$. Let $M$ be the intersection of $\overline{0\infty}$ and the geodesic connecting $A_1$ and $C_2$. Let $a_1$ be the distance from $B_i$ to the horocycle centered at $A_i$ and let $a_2$ be the distance from $B_i$ to $C_i$ for $i=1,2$ so that $a_1+a_2=a$, and let $x_1$ be the distance from $M$ to the horocycle centered at $A_1$ and let $x_2$ be the distance from $M$ to the horocycle centered at $C_2$ so that $x_1+x_2=x$. Since $\angle A_1B_1M=\angle C_2B_2M$ and $\angle A_1MB_1=\angle C_2MB_2$, we have that the ideal triangles $A_1B_1M$ and $C_2B_2M$ of type $(0,1,1)$ are isometric which implies that $|B_1M|=|MB_2|=\frac{b}{2}$. Applying the cosine law to the triangle $A_1B_1M$, we have
\[\cos(\pi-\theta)=\frac{-e^{x_1}+e^{a_1}\cosh\frac{b}{2}}{e^{a_1}\sinh\frac{b}{2}}.\] Applying the sine law to the triangles $A_1B_1M$ and $C_2B_2M$, we have \[\frac{e^{a_1}}{e^{x_1}}=\frac{\sin\angle A_1MB_1}{\sin\angle A_1B_1M}=\frac{\sin\angle  C_2MB_2}{\sin\angle C_2B_2M}=\frac{e^{a_2}}{e^{x_2}}\ \text{ hence }\ \frac{a_2-a_1}{2}=\frac{x_2-x_1}{2}\,.\]
Using this, the cosine law above becomes \[\cos(\pi-\theta)=\frac{-e^{\frac{x}{2}}+e^{\frac{a}{2}}\cosh\frac{b}{2}}{e^{\frac{a}{2}}\sinh\frac{b}{2}}.\] By the same argument applied to the generalized triangles $A_2B_2M$ and $B_1C_1M$, we obtain \[\cos\theta=\frac{-e^{\frac{y}{2}}+e^{\frac{a}{2}}\cosh\frac{b}{2}}{e^{\frac{a}{2}}\sinh\frac{b}{2}}\,.\]
Part (1) is obtained by taking the sum of the two equalities above and Part (2) by taking their difference.\end{proof}

The following lemma involving $\lambda$--length can be found in \cite{Penner2} (Lemma A1). Part (1) was proved first by Penner in \cite{Penner1} and is the celebrated Ptolemy relation.

\begin{lemma}(\textbf{Penner\,\cite{Penner2}})\label{3.3} Let $\alpha$ and $\beta$ be two geodesic arcs of lengths $a$ and $b$, and let $\theta$ be the angle from $\alpha$ to $\beta$ at $p\in\alpha\cap\beta$. If $x$ and $x'$ respectively are the lengths of the geodesic representatives of the components of $\alpha_p\beta^+$, and $y$ and $y'$ respectively are the lengths of the geodesic representatives of the components of $\alpha_p\beta^-$, then we have
\begin{enumerate}[(1)]
\item $e^{\frac{x}{2}}e^{\frac{x'}{2}}+e^{\frac{y}{2}}e^{\frac{y'}{2}}=e^{\frac{a}{2}}e^{\frac{b}{2}},$

\item $e^{\frac{x}{2}}e^{\frac{x'}{2}}-e^{\frac{y}{2}}e^{\frac{y'}{2}}= e^{\frac{a}{2}}e^{\frac{b}{2}}\cos\theta.$\hfill$\square$
\end{enumerate}
\end{lemma}

\begin{lemma}\label{3.4} Let $\alpha$ and $\beta$ be two geodesic arcs of lengths $a$ and $b$ each having exactly one end at a puncture $v$, and let $\theta$ be the generalized angle from $\alpha$ to $\beta$ and $\theta'$ be the generalized angle from $\beta$ to $\alpha$. Let $r$ be the length of the horocycle centered at $v$, and let $x$ and $y$ be the lengths of the geodesic representatives of $\alpha_v\beta^+$ and $\alpha_v\beta^-$ respectively. Then we have
\begin{enumerate}[(1)]
\item $e^{\frac{x}{2}}+e^{\frac{y}{2}}=re^{\frac{a}{2}}e^{\frac{b}{2}},$

\item $e^{\frac{x}{2}}-e^{\frac{y}{2}}=(\theta'-\theta)e^{\frac{a}{2}}e^{\frac{b}{2}}.$
\end{enumerate}
\end{lemma}

\begin{figure}[htbp]\centering
\includegraphics[width=4.5cm]{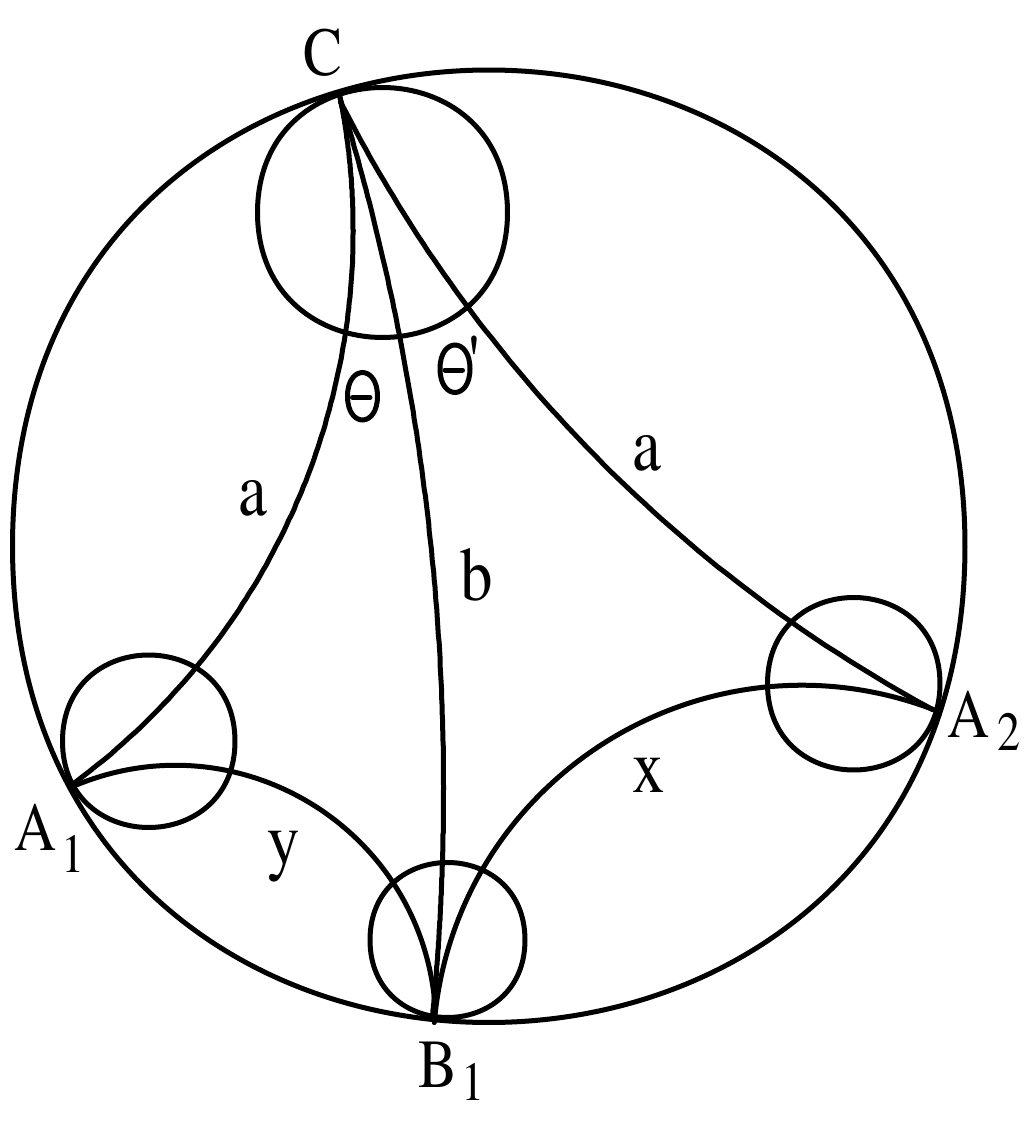}\\
\caption{}\label{puncture-crossing}
\end{figure}

\begin{proof} Let us look at Figure \ref{puncture-crossing}. Let $C$ be a lift of $v$ in the universal cover $\mathbb{H}^2$, and let $\overline{A_1C}$ and $\overline{B_1C}$ be the respective lifts of $\alpha$ and $\beta$ passing through $C$. Then $\overline{A_2B_1}$ and $\overline{A_1B_1}$ are lifts of $\alpha_v\beta^+$ and $\alpha_v\beta^-$, respectively. Applying the cosine law to the ideal triangles $CA_2B_1$ and $CA_1B_1$, we have \[\theta'=e^{\frac{x-a-b}{2}}\ \text{ and }\ \theta=e^{\frac{y-a-b}{2}}\,.\]
Since $r=\theta+\theta'$, the sum of the two equalities above implies part (1), and the difference implies part (2).\end{proof}

\begin{remark}\label{multiple}
Note that if $\alpha$ or $\beta$ have several of there ends meeting at $v$, similar formulas hold replacing $x$ and $y$ by all the possible positive and negative resolutions between $\alpha$ and $\beta$, taking the sum of their $\lambda$--lengths instead and considering the sums of the generalized angles between their ends. See the definition of the Goldman bracket on $\C(\s)$ and that of the Weil-Petersson Poisson bi-vector $\Pi_{WP}$ for comparison.
\end{remark}

For each point of self-intersection $p$ of an arc or a loop $\alpha$, one of its two resolutions at $p$ is connected and the other one is not. We call the former one the \emph{non-separating resolution} of $\alpha$ and the later one the \emph{separating resolution} of $\alpha$. Note that if $\alpha$ is an arc, then the separating resolution of $\alpha$ consists of an arc and a loop, which we call the \emph{arc component} and the \emph{loop component} respectively.

Although the curling number of a geodesic is always 0, when resolving a self-intersection the curling number of one of its resolutions may be 1. As such we have the following lemma.

\begin{lemma}\label{lemma:nonzero} Let $\alpha$ be a closed geodesic or a geodesic arc. Then the curling number $c(\beta)$ of the non-separating resolution $\beta$ of $\alpha$ at each of its points of self-intersection is at most $1$; and the only possibility that $c(\beta)=1$ is as shown in Figure \ref{nonezero}.
\end{lemma}

\begin{figure}[htbp]\centering
\includegraphics[width=6cm]{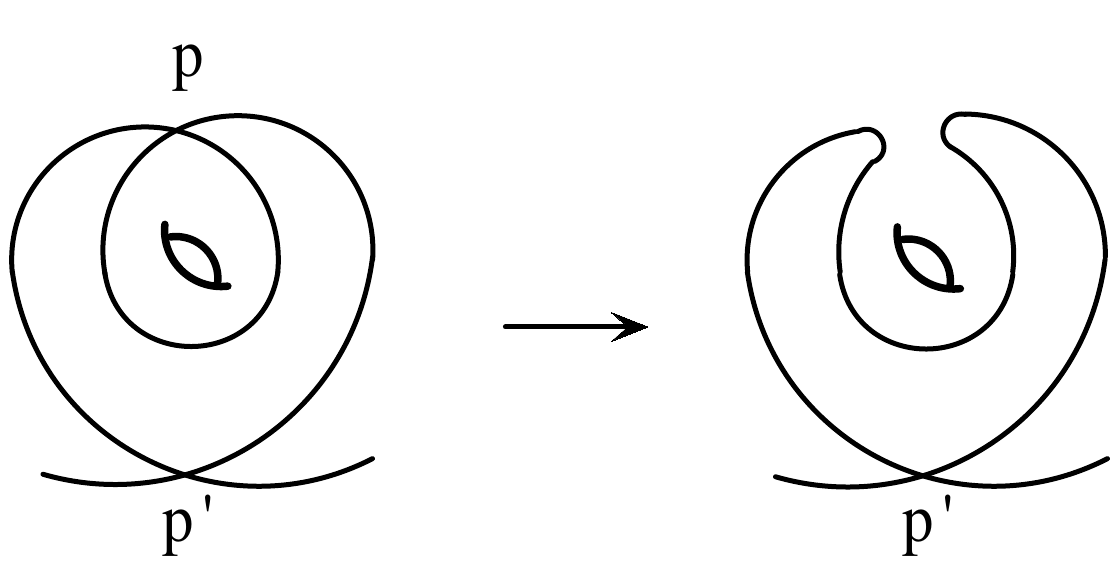}\\
\caption{The case $c(\beta)=1$.}\label{nonezero}
\end{figure}

\begin{proof} If $c(\beta)>0$, let $c$ be one of its curl and let $p'$ be the vertex of $c$. Let also $\alpha_1$ and $\alpha_2$ be the components of $\alpha-p$. We have the following two cases:
\begin{enumerate}[(a)]
\item $p'\in \alpha_i$ and $c\subset \alpha_i$ for $i=1$ or $2$, 

\item $p'\in \alpha_i$ but $c\nsubseteq\alpha_i$ for $i=1$ or $2$.
\end{enumerate}
However, if (a) occurred, then $\alpha$ itself would contain either a curl or a bi-gon with vertices $p$ and $p'$, which is excluded since $\alpha$ is a geodesic. Hence the only possibility is (b), in which case if it is true for say $i=1$, then $\alpha_2\subset c$ and $c$ is the only curl in $\beta$, otherwise $\alpha$ would also contain one. One can see then that the only possible configuration is that of Figure~\ref{nonezero}.\end{proof}

\begin{lemma}\label{3.6} Let $\alpha$ be a closed geodesic of length $a$ and $p$ be one of its self-intersection points. Let $x$ and $y$ respectively be the lengths of the geodesic representatives of the two components of the separating resolution of $\alpha$, and let $z$ be the length of the geodesic representative of the non-separating resolution $\beta$ of $\alpha$. 
\begin{enumerate}[(1)]
\item If $c(\beta)=0$, then $$\cosh\frac{a}{2}=2\cosh\frac{x}{2}\cosh\frac{y}{2}+\cosh\frac{z}{2}.$$

\item If $c(\beta)=1$, then $$\cosh\frac{a}{2}=2\cosh\frac{x}{2}\cosh\frac{y}{2}-\cosh\frac{z}{2}.$$
\end{enumerate}
In addition, the formulae hold when some components of the resolutions of $\alpha$ are circles around a puncture.

\end{lemma}

\begin{figure}[htbp]\centering
\includegraphics[width=9.5cm]{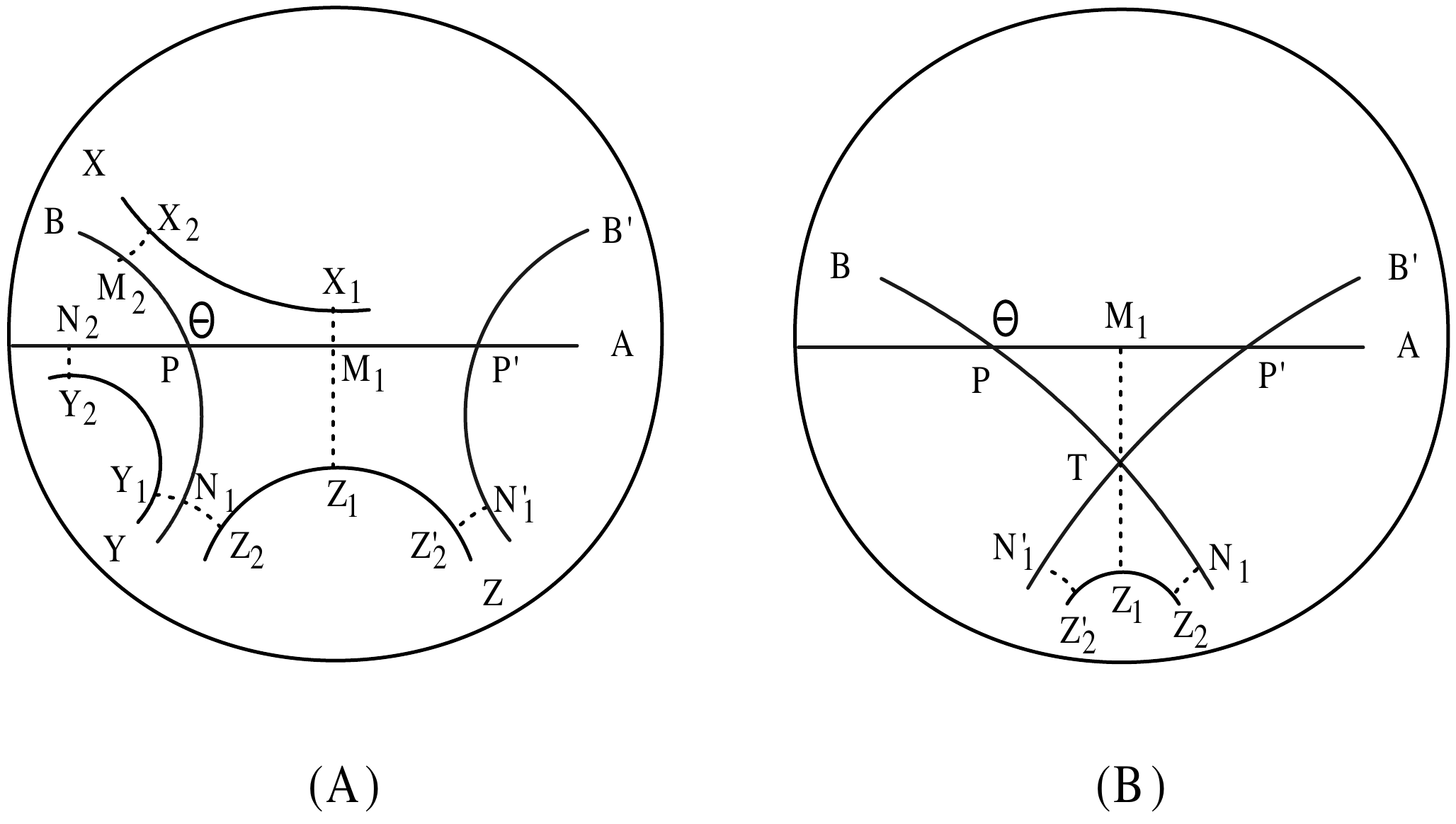}\\
\caption{}\label{self-intersection}
\end{figure}

\begin{proof}  For (1), let us look at (A) of Figure \ref{self-intersection}. Let $P$ be a lift of $p$ in the universal cover $\mathbb{H}^2$, and let $\theta$ the angle between the two lifts  $A$ and $B$ of $\alpha$ passing through $P$. Let $X$ and $Y$ be the corresponding lifts of the components of the separating resolution of $\alpha$, and let $Z$ be the corresponding lift of the non-separating resolution of $\alpha$. Let $M_1\in A$, $N_1\in B$ and $Z_1$, $Z_2\in Z$ be such that $|M_1Z_1|$ realizes the distance $d(A,Z)$ and $|N_1Z_2|$ realizes the distance $d(B,Z)$. Let $P'$ be the lift of $p$ on $A$ next to $P$ and let $B'$ be the other lift of $\alpha$ passing through $P'$. Let $N_1'\in B'$ and $Z_2'\in Z$ be such that $|N_1'Z_2'|$ realizes $d(B',Z)$. Then $\overline{N_1'Z_2'}$ and $\overline{N_1Z_2}$ are the lifts of the same geodesic segment, which implies that $|N_1'Z_2'|=|N_1Z_2|$. Since $\angle N_1'P'M_1=\angle N_1PM_1$, the generalized triangles $M_1Z_1Z_2'N_1'P'$ and $M_1Z_1Z_2N_1P$ of type $(1,-1,-1)$ are isometric. Therefore, the lengths $|Z_2'Z_1|=|Z_1Z_2|=\frac{z}{2}$ and $|P'M_1|=|PM_1|\doteq a_1$. Let $M_1''\in A$ and $X_1\in X$ be such that $|M_1''X_1|$ realizes $d(A,X)$, and let $N_1''\in B$ and $Y_1\in Y$ be such that $|N_1''Y_1|$ realizes $d(B,Y)$. By a similar argument as above, we see that $|PM_1''|=\frac{1}{2}|P'P|=a_1$ and $PN_1''=PN_1\doteq a_2$. Let $M_2\in B$, $N_2\in A$, $X_2\in X$ and $Y_2\in Y$ be such that $|M_2X_2|$ realizes $d(B,X)$ and $|N_2Y_2|$ realizes $d(A,Y)$. Then as above, we have $|PM_2|=a_1$ and $|PN_2|=a_2$. Since $|PM_1|=|PM_2|=\frac{1}{2}|P'P|$, the points $M_1$ and $M_2$ cover the same point on the surface $\s$, hence $X_1$ and $X_2$ cover the same point on $\s$ and $|X_1X_2|=x$. The same argument implies $|Y_1Y_2|=y$. Applying the cosine law to the generalized triangles $PM_1X_1X_2M_2$ and $PN_1Y_1Y_2N_2$ of type $(1,-1,-1)$, we have \[\cos\theta=\frac{-\cosh x+\sinh^2\frac{a_1}{2}}{\cosh^2\frac{a_1}{2}}=\frac{-\cosh y+\sinh^2\frac{a_2}{2}}{\cosh^2\frac{a_2}{2}}\,,\]
which implies \[\sin^2\frac{\theta}{2}=\frac{\cosh\frac{x}{2}\cosh\frac{y}{2}}{\cosh\frac{a_1}{2}\cosh\frac{a_2}{2}}\,.\]
Applying the cosine law to the generalized triangle $PM_1Z_1Z_2N_1$ of the same type, we have \[\cos(\pi-\theta)=\frac{-\cosh\frac{z}{2}+\sinh\frac{a_1}{2}\sinh\frac{a_2}{2}}{\cosh\frac{a_1}{2}\cosh\frac{a_2}{2}}\,.\]
From the last two equations and the identity $\cos(\pi-\theta)=2\sin^2\frac{\theta}{2}-1$ we obtain the result. Note that when some components of the resolutions of $\alpha$ are curves around a puncture, then the corresponding lengths $x$, $y$ or $z$ are taken to be $0$, and the corresponding generalized triangles become union of generalized ideal triangles of type $(0,1,1)$. Applying the cosine law for such triangles we obtain formula (1) in these degenerated cases. 

For (2), let us look at (B) of Figure \ref{self-intersection}. Applying similarly the cosine law to the generalized triangles $PM_1X_1X_2M_2$ and $PN_1Y_1Y_2N_2$ of type $(1,-1,-1)$, we obtain
\[\cos\theta=\frac{-\cosh x+\sinh^2\frac{a_1}{2}}{\cosh^2\frac{a_1}{2}}=\frac{-\cosh y+\sinh^2\frac{a_2}{2}}{\cosh^2\frac{a_2}{2}}\,,\]
which implies \[\sin^2\frac{\theta}{2}=\frac{\cosh\frac{x}{2}\cosh\frac{y}{2}}{\cosh\frac{a_1}{2}\cosh\frac{a_2}{2}}\,.\]
Since $c(\beta)=1$, there is an intersection $T$ between $B$ and $B'$ and the generalized triangle $PM_1Z_1Z_2N_1$ of type $(1,-1,-1)$ is twisted. Applying the cosine law for such generalized triangle (see Appendix~\ref{A}), we have \[\cos(\pi-\theta)=\frac{\cosh\frac{z}{2}+\sinh\frac{a_1}{2}\sinh\frac{a_2}{2}}{\cosh\frac{a_1}{2}\cosh\frac{a_2}{2}}\,,\]
and the identity $\cos(\pi-\theta)=2\sin^2\frac{\theta}{2}-1$ implies the result. \end{proof}

\begin{lemma}\label{3.7} Let $\alpha$ be a geodesic arc of length $a$ and $p$ be one of its points of self-intersection. For the separating resolution of $\alpha$, let $x$ and $y$ be the lengths of the geodesic representatives of the loop and of the arc component respectively. We also let  $z$ be the length of the geodesic representative of the non-separating resolution $\beta$ of $\alpha$.
\begin{enumerate}[(1)]
\item If $c(\beta)=0$, then 
$$e^{\frac{a}{2}}=2\cosh\frac{x}{2}e^{\frac{y}{2}}+e^{\frac{z}{2}}.$$

\item If $c(\beta)=1$, then 
$$e^{\frac{a}{2}}=2\cosh\frac{x}{2}e^{\frac{y}{2}}-e^{\frac{z}{2}}.$$
\end{enumerate}
In addition, the formulae hold when the loop component in the separating resolution of $\alpha$ is a circle around a puncture.
\end{lemma}

\begin{figure}[htbp]\centering
\includegraphics[width=9.5cm]{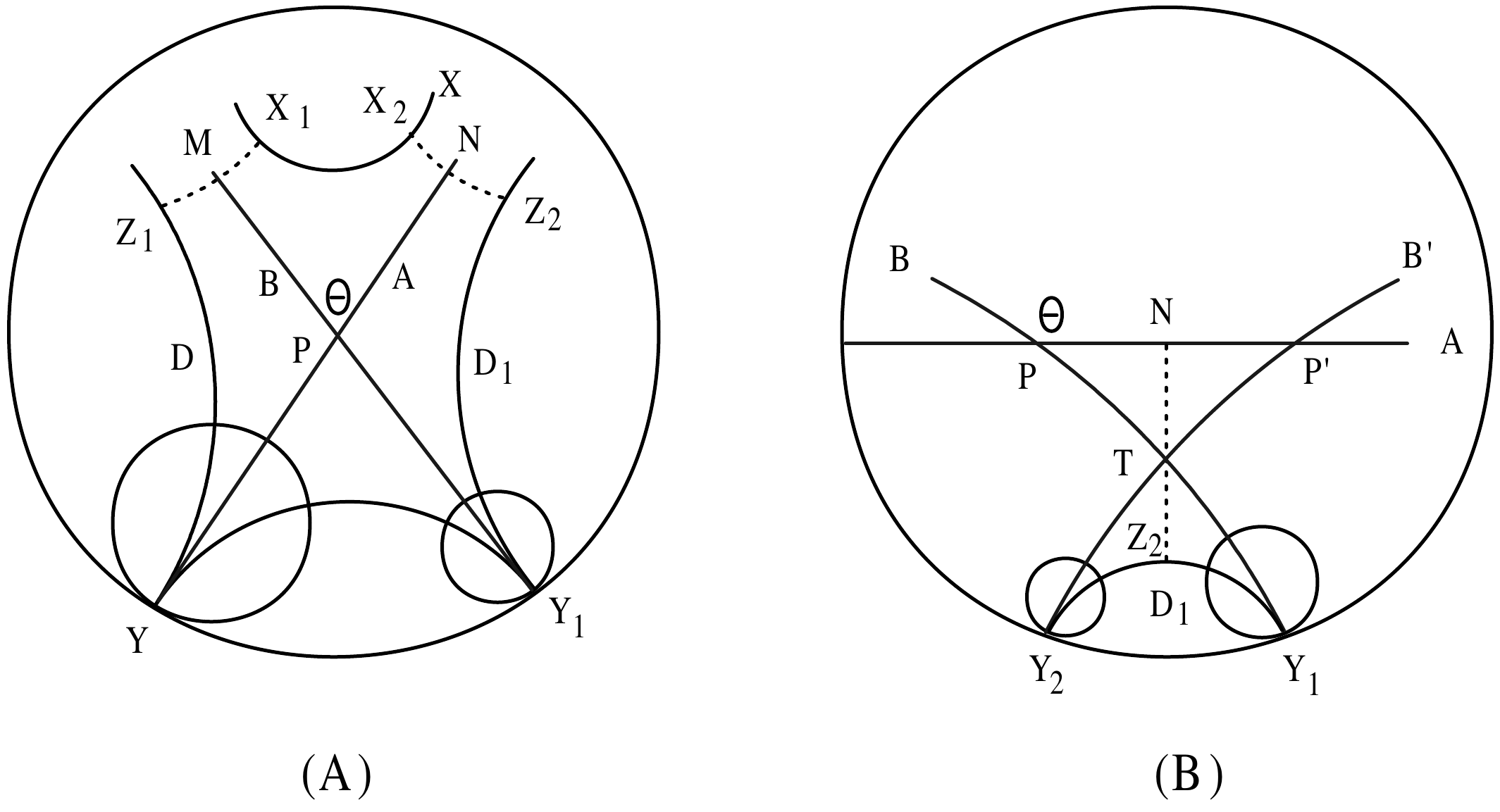}\\
\caption{}\label{arc}
\end{figure}

\begin{proof} For (1),  let $P$ in (A) of Figure \ref{arc} be a lift of $p$ in the universal cover $\mathbb{H}^2$, and let $A$ and $B$ be the two lifts of $\alpha$ passing through $P$ with $\theta$ the angle between them at $p$. Let the end point $Y$ of $A$ and the end point $Y_1$ of $B$ respectively be the lifts of the two end points of $\alpha$ so that $\overline{YY_1}$ is a lift of the geodesic representative of the arc component of the separating resolution of $\alpha$. Let $X$ be the corresponding lift of the geodesic representative of the loop component of the separating resolution of $\alpha$, and let $D$ and $D_1$ be the lifts of the geodesic representative of the non-separating resolution of $\alpha$.  We take points $X_1$ and $X_2\in X$,  $M\in B$ and $N\in A$ such that $|MX_1|$ realizes $d(B,X)$ and $|NX_2|$ realizes $d(A,X)$. Since $\overline{MX_1}$ and $\overline{NX_2}$ cover the same curve on $\s$, we have $|MX_1|=|NX_2|$. Applying the sine law to the generalized triangle $PMX_1X_2N$ of type $(-1,-1,1)$, we have $|PM|=|PN|\doteq\frac{a_3}{2}$. Suppose $M'\in A'$, $N'\in A$, $Z_1\in D$ and $Z_2\in D_1$ are the points such that $|M'Z_1|$ realizes $d(B,D)$ and $|N'Z_2|$ realizes $d(A,D_1)$. Then $\overline{M'Z_1}$ and $\overline{N'Z_2}$ cover the same curve on $\s$, hence $|M'Z_1|=|N'Z_2|$. Applying the cosine law to the generalized ideal triangles $PYZ_1M'$ and $PY_1Z_2N'$ of type $(-1,0,1)$, we see that
\begin{equation*}
\begin{split}
\sinh|PM'|&=\frac{1+\cos(\pi-\theta)\cosh|M'Z_1|}{\sin(\pi-\theta)\sinh|M'Z_1|}\\
&=\frac{1+\cos(\pi-\theta)\cosh|N'Z_2|}{\sin(\pi-\theta)\sinh|N'Z_2|}=\sinh|PN'|\,.
\end{split}
\end{equation*}
Therefore we have $|PM'|=|PN'|$. Hence $M'=M$, $N'=N$ and $|PM'|=|PN'|=\frac{a_3}{2}$. Let $H_Y$ and $H_{Y_1}$ respectively be the horocycles centered at $Y$ and $Y_1$, and $a_1=d(P,H_Y)$, $a_2=d(P,H_{Y_1})$, $z_1=d(Z_1,H_Y)$ and $z_2=d(Z_2,H_{Y_1})$. Then $a=a_1+a_2+a_3$ and $z=z_1+z_2$. Applying the cosine law to the generalized ideal triangle $PYZ_1M$, we have \[\cos(\pi-\theta)=\frac{-e^{z_1}+e^{a_1}\sinh\frac{a_3}{2}}{e^{a_1}\cosh\frac{a_3}{2}}\,.\]
From the sine law applied to the generalized ideal triangles $PYZ_1M$ and $PY_1Z_2N$, we have \[\frac{e^{z_1}}{e^{a_1}}=\frac{\sin(\pi-\theta)}{\sinh|MZ_1|}=\frac{\sin(\pi-\theta)}{\sinh|NZ_2|}=\frac{e^{z_2}}{e^{a_2}}\,,\ \text{ hence }\ \frac{a_2-a_1}{2}=\frac{z_2-z_1}{2}\,.\]
Using this, the cosine law above becomes \[\cos(\pi-\theta)=\frac{-e^{\frac{z}{2}}+e^{\frac{a_1+a_2}{2}}\sinh\frac{a_3}{2}}{e^{\frac{a_1+a_2}{2}}\cosh\frac{a_3}{2}}\,,\]
hence \[e^{\frac{z}{2}}=e^{\frac{a_1+a_2}{2}}(\sinh\frac{a_3}{2}+\cosh\frac{a_3}{2}\cos\theta)\,.\]
Applying the cosine law to the generalized triangle $PMX_1X_2N$, we have \[\cos\theta=\frac{-\cosh x+\sinh^2\frac{a_3}{2}}{\cosh^2\frac{a_3}{2}}\,,\ \text{hence }\ 2\cosh\frac{x}{2}=2\cosh\frac{a_3}{2}\sin\frac{\theta}{2}\,,\]
and the cosine law applied to the generalized ideal triangle $PYY_1$ of type $(0,0,1)$ gives \[e^{\frac{y}{2}}=e^{\frac{a_1+a_2}{2}}\sin\frac{\theta}{2}\,.\]
Therefore, we have
\begin{equation*}
\begin{split}
2\cosh\frac{x}{2}e^{\frac{y}{2}}+e^{\frac{z}{2}}&=e^{\frac{a_1+a_2}{2}}(\sinh\frac{a_3}{2}+\cosh\frac{a_3}{2}\cos\theta+2\cosh\frac{a_3}{2}\sin^2\frac{\theta}{2})\\
&=e^{\frac{a_1+a_2}{2}}e^{\frac{a_3}{2}}=e^{\frac{a}{2}}\,.
\end{split}
\end{equation*} 

For (2), let us look at (B) of Figure \ref{arc}. Applying similarly the cosine law to the generalized triangle $PMX_1X_2N$, we have 
\[\cos\theta=\frac{-\cosh x+\sinh^2\frac{a_3}{2}}{\cosh^2\frac{a_3}{2}}\,,\]
which implies
\[2\cosh\frac{x}{2}=2\cosh\frac{a_3}{2}\sin\frac{\theta}{2}\,,\]
and the cosine law applied to the generalized ideal triangle $PYY'$ of type $(0,0,1)$ gives \[e^{\frac{y}{2}}=e^{\frac{a_1+a_2}{2}}\sin\frac{\theta}{2}.\]
When $c(\beta)=1$, there is an intersection $T$ between $A'$ and $A''$ and the generalized triangles $PNZ_2Y'$ and $P'NZ_2Y''$ of type $(0,1,-1)$ are twisted. Applying the cosine law to $PNZ_2Y'$, we have \[\cos(\pi-\theta)=\frac{e^{z_1}+e^{a_1}\sinh\frac{a_3}{2}}{e^{a_1}\cosh\frac{a_3}{2}}\,.\]
From the sine law for the generalized ideal triangles $PNZ_2Y'$ and $PNZ_2Y''$, we have \[\frac{e^{z_1}}{e^{a_1}}=\frac{\sin(\pi-\theta)}{\sinh|NZ_2|}=\frac{e^{z_2}}{e^{a_2}}\,,\ \text{which implies }\ \frac{a_2-a_1}{2}=\frac{z_2-z_1}{2}\,.\]
Using this, the cosine law above becomes \[\cos(\pi-\theta)=\frac{e^{\frac{z}{2}}+e^{\frac{a_1+a_2}{2}}\sinh\frac{a_3}{2}}{e^{\frac{a_1+a_2}{2}}\cosh\frac{a_3}{2}}\,,\]
hence \[-e^{\frac{z}{2}}=e^{\frac{a_1+a_2}{2}}(\sinh\frac{a_3}{2}+\cosh\frac{a_3}{2}\cos\theta)\,.\] Therefore, we obtain
\[2\cosh\frac{x}{2}e^{\frac{y}{2}}-e^{\frac{z}{2}}=e^{\frac{a}{2}}\,.\] 
\end{proof}

\begin{lemma}\label{3.8} Let $\alpha$ be a geodesic arc of length $a$ both of whose ends meet at a puncture $v$, and let $r$ be the length of the horocycle centered at $v$. Let also $x$ and $y$ be the lengths of the geodesic representatives of the two resolutions $\alpha_1$ and $\alpha_2$ of $\alpha$ at $v$. Then we have
\[e^{\frac{a}{2}}=\frac{2}{r}(\cosh\frac{x}{2}+\cosh\frac{y}{2})\,.\]
In addition, the formula holds when some of the components of the resolutions of $\alpha$ are circles around a puncture.
\end{lemma}

\begin{figure}[htbp]\centering
\includegraphics[width=4.5cm]{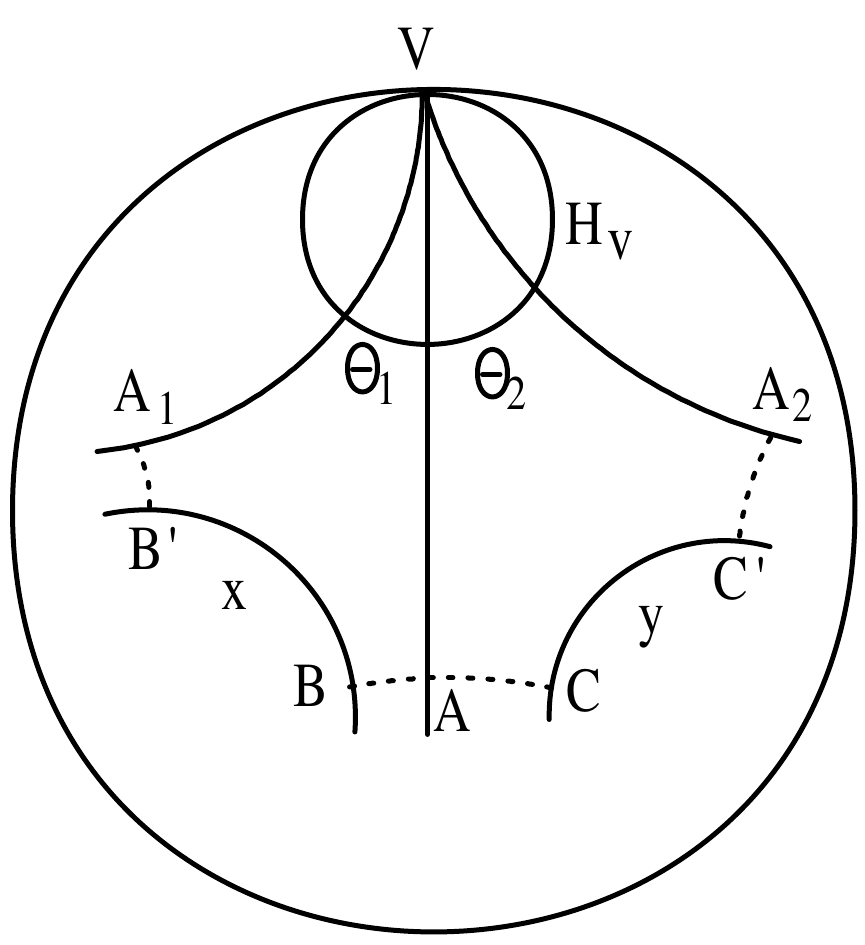}\\
\caption{}\label{self-arc}
\end{figure}

\begin{proof} In Figure \ref{self-arc}, let $V$ be the lift of $v$ and let $H_V$ be the lift of the horocycle centered at $V$, let  $\overline{AV}$ and $\overline{A_1V}$ be the lifts of $\alpha$ passing through $V$ in the universal cover $\mathbb{H}^2$. Let $\theta_1$ be the generalized angle between $\overline{AV}$ and $\overline{A_1V}$ and let $\overline{BB'}$ be the corresponding lift  of the geodesic representative of the homotopy class of $\alpha_1$. We take the point $A$, $A_1$, $B$ and $B'$ such that $|AB|$ realizes the distance from $\overline{AV}$ to $\overline{BB'}$ and $|A_1B'|$ realizes the distance from $\overline{A_1V}$ to $\overline{BB'}$. Since $\overline{AB}$ and $\overline{A_1B'}$ cover the same line in $\s$, we have $|AB|=|A_1B'|$ and $|BB'|=x$. By the sine law for the generalized triangle $CABB'A'$ of type $(0,-1,-1)$, we have \[\frac{e^{d(A,H_V)}}{\sinh|A_1B'|}=\frac{e^{d(A_1,H_V)}}{\sinh|AB|}=1,\]
which implies that $d(A,H_V)=d(A_1,H_V)=\frac{a}{2}$. Applying the cosine law to the generalized triangle $CABB'A'$, we have
\[\theta_1^2=\frac{\cosh x+1}{\frac{e^{a}}{2}},\]
which implies that
\[\theta_1=\frac{2\cosh\frac{x}{2}}{e^{\frac{a}{2}}}.\]
Similarly, if we let $\overline{A_2V}$ be the other lift of $\alpha$ adjacent to $\overline{AV}$ and let $\theta_2$ be the generalized angle between $\overline{AV}$ and $\overline{A_2V}$, we obtain 
\[\theta_2=\frac{2\cosh\frac{y}{2}}{e^{\frac{a}{2}}},\]
which together with the previous identity implies the formula.\end{proof}


\subsection{Generalized trace identities and the algebra homomorphism}\label{algebra}

Combining the results from the previous section, we obtain the following generalized trace identities. 

\begin{proposition}\label{prop:induction} 
\begin{enumerate}[(a)]

\item For a generalized curve $\alpha$ with $p$ one of its self-intersection points in $\s$, let $\alpha_1$ and $\alpha_2$ be the components of the separating resolution of $\alpha$ at $p$ and let $\beta$ be the non-separating resolution of $\alpha$ at $p$. Then we have $$(-1)^{c(\alpha)}\lambda(\alpha)=(-1)^{c(\alpha_1)+c(\alpha_2)}\lambda(\alpha_1)\lambda(\alpha_2)+(-1)^{c(\beta)}\lambda(\beta).$$
\item Let $\alpha$ and $\beta$ be two generalized curves with $p\in \s$ one of their intersections. If $\gamma_1$ and $\gamma_2$ are the resolutions of $\alpha$ and $\beta$ at $p$, then we have 
$$(-1)^{c(\alpha)+c(\beta)}\lambda(\alpha)\lambda(\beta)=(-1)^{c(\gamma_1)}\lambda(\gamma_1)+(-1)^{c(\gamma_2)}\lambda(\gamma_2).$$
\end{enumerate}
\end{proposition}

\begin{proof} First note that by part (1) of Lemma \ref{3.1}\,-\,\ref{3.3}, Lemma \ref{3.6} and \ref{3.7}, the formulae (a) and (b) are true when $\alpha$ and $\beta$ are geodesics. In the general case when the generalized curves are not all geodesics, we use induction on the number of intersection points. If a generalized curve $\alpha$ has only one self-intersection $p$ which is a vertex of a curl $C$, then we let $\alpha'$ be the generalized curve obtained from $\alpha$ by removing $C$ via a Reidemeister Move I. We let $\alpha_1$ and $\alpha_2$ be the components of the separating resolution of $\alpha$ at $p$, and let $\beta$ be the non-separating resolution of $\alpha$ at $p$. Note that one of $\alpha_1$ and $\alpha_2$, say $\alpha_1$, is a trivial loop since $p$ is the vertex of the curl $C$; and $\alpha_2$ and $\beta$ are regularly isotopic to $\alpha'$. We have $(-1)^{c(\alpha)}\lambda(\alpha)=-(-1)^{c(\alpha')}\lambda(\alpha')$ and $(-1)^{c(\alpha_1)+c(\alpha_2)}\lambda(\alpha_1)\lambda(\alpha_2)+(-1)^{c(\beta)}\lambda(\beta)=-(-1)^{c(\alpha_2)}2\lambda(\alpha_2)+(-1)^{c(\beta)}\lambda(\beta)=-(-1)^{c(\alpha')}\lambda(\alpha').$ Hence (a) is true in this case. If $\alpha$ has only one self-intersection and no curl, then $\alpha$ is regularly isotopic to a geodesic, and (a) is true by Lemma \ref{3.6} and \ref{3.7}. If two simple generalized curves $\alpha$ and $\beta$ have only one intersection, then $\alpha$ and $\beta$ are regularly isotopic to geodesics, and (b) is true by (1) of Lemma \ref{3.1}\,-\,\ref{3.3}. Now we assume that formula (a) holds when the number of self-intersections of $\alpha$ is less than $n$, and formula (b) holds when the number of crossings $\alpha\cup\beta$ is less than $n$. 

For (a),  if the number of self-intersections of $\alpha$ is equal to $n$, we have to consider the following two cases: 
\begin{enumerate}[(1)]
\item $p$ is not a vertex of a bi-gon bounded by $\alpha$; and
\item $p$ is a vertex of a bi-gon $B$ bounded by $\alpha$. 
\end{enumerate}
We denote by $\overline{\gamma}$ the unique geodesic in the homotopy class of a generalized curve $\gamma$. In case (1),  we have that $c(\alpha)=c(\alpha_1)+c(\alpha_2)$ and $0\leqslant c(\beta)-c(\alpha)\leqslant1$. Moreover, $c(\beta)-c(\alpha)=1$ if and only if the non-separating resolution $\beta'$ of $\overline{\alpha}$ contains a curl, i.e., $c(\beta)-c(\alpha)=c(\beta')$. The way to see this is that we first push all the curls together by regular isotopy so that away from the curls the curve is equivalent to a geodesic and then apply Proposition \ref{lemma:nonzero}. In this case, we have 
\begin{equation*}
\begin{split}
(-1)^{c(\alpha)}\lambda(\alpha)&=(-1)^{c(\alpha)}\lambda(\overline{\alpha})\\
&=(-1)^{c(\alpha)}\big(\lambda(\overline{\alpha_1})\lambda(\overline{\alpha_2})+(-1)^{c(\beta')}\lambda(\beta')\big)\\
&=(-1)^{c(\alpha_1)+c(\alpha_2)}\lambda(\alpha_1)\lambda(\alpha_2)+(-1)^{c(\beta)}\lambda(\beta).
\end{split}
\end{equation*}
In case (2), there is a curl $C$ generated from the bi-gon $B$ whose vertex is the other vertex $p'$ of $B$ . If $C$ is in one of the component of the separating resolution of $\alpha$, say $\alpha_1$,  then let $\alpha_1'$ be the curve obtained from $\alpha_1$ by removing $C$ via a Reidemeister Move I. Let $\alpha'$ be the curve obtained from $\alpha$ by removing $B$ via a Reidemeister Move II. Then we have $c(\alpha')=c(\alpha)$ and $c(\alpha_1')=c(\alpha_1)-1$, and formula (a) is equivalent to 

$$(-1)^{c(\beta)}\lambda(\beta)=(-1)^{c(\alpha')}\lambda(\alpha')+(-1)^{c(\alpha_1')+c(\alpha_2)}\lambda(\alpha_1')\lambda(\alpha_2),$$
which holds by the inductive assumption to formula $(a)$. Indeed, the generalized curves $\alpha_1'$ and $\alpha_2$ are regularly isotopic to the components of the separating resolution of $\beta$ at $p'$ and $\alpha'$ is regularly isotopic to the non-separating resolution of $\beta$ at $p'$, and the number of self-intersections of $\beta$ is less than $n$. If $C\subset\beta$, then let $\beta'$ be the curve obtained from $\beta$ by removing $C$ via a Reidemeister Move I. We have $c(\beta')=c(\beta)-1$, and that formula (a) is equivalent to 
$$(-1)^{c(\alpha_1)+c(\alpha_2)}\lambda(\alpha_1)\lambda(\alpha_2)=(-1)^{c(\alpha')}\lambda(\alpha')+(-1)^{c(\beta')}\lambda(\beta'),$$
which holds by the induction assumption to formula (b). Indeed, the generalized curves $\alpha'$ and $\beta'$ are regularly isotopic to resolutions of $\alpha_1\cup\alpha_2$ at $p'$. 

For (b), we have to consider the following two cases:
\begin{enumerate}[(1)]
\item $p$ is not a vertex of a bi-gon bounded by $\alpha$ and $\beta$; and

\item $p$ is a vertex of a bi-gon $B$ bounded by $\alpha$ and $\beta$.
\end{enumerate}
We still denote by $\overline{\gamma}$ the unique geodesic in the homotopy class of a generalized curve $\gamma$. In case (1), the resolutions $p$ do not change the number of curls. We have $c(\alpha)+c(\beta)=c(\gamma_1)=c(\gamma_2),$ and 

\begin{equation*}
\begin{split}
(-1)^{c(\alpha)+c(\beta)}\lambda(\alpha)\lambda(\beta)&=(-1)^{c(\alpha)+c(\beta)}\lambda(\overline{\alpha})\lambda(\overline{\beta})\\
&=(-1)^{c(\alpha)+c(\beta)}\big(\lambda(\overline{\gamma_1})+\lambda(\overline{\gamma_2})\big)\\
&=(-1)^{c(\gamma_1)}\lambda(\gamma_1)+(-1)^{c(\gamma_2)}\lambda(\gamma_2).
\end{split}
\end{equation*}
In case (2), there is a curl $C$ generated from the bi-gon $B$ whose vertex is the other vertex $p'$ of $B$. If $C\subset \gamma_1$, say, then let $\gamma_1'$ be the curve obtained from $\alpha_1$ be removing $C$ via a Reidemeister Move I. Let $\alpha'$ and $\beta'$ be the curve obtained from $\alpha$ and $\beta$ by removing $B$ via a Reidemeister Move II. We have $c(\alpha')=c(\alpha)$, $c(\beta')=c(\beta)$ and $c(\gamma_1')=c(\gamma_1)-1$, and the result is equivalent to 

$$(-1)^{c(\gamma_2)}\lambda(\gamma_2)=(-1)^{c(\alpha')+c(\beta')}\lambda(\alpha')\lambda(\beta')+(-1)^{c(\gamma_1')}\lambda(\gamma_1'),$$
which holds by formula (a), since $\alpha'$ and $\beta'$ are regularly isotopic to the components of the separating resolution of $\gamma_2$ at $p'$ and $\gamma_1'$ is regularly isotopic to the non-separating resolution of $\gamma_2$ at $p'$.
\end{proof}

\begin{proposition}\label{prop:puncture} 
\begin{enumerate}[(a)]
\item For an arc $\alpha$ both of whose end points are at the same puncture $v$, let $\beta$ and $\gamma$ be the resolutions of $\alpha$ at $v$, and let $r(v)$ be the length of the horocycle centered at $v$. Then we have 
$$(-1)^{c(\alpha)}\lambda(\alpha)=\frac{1}{r(v)}\big((-1)^{c(\beta)}\lambda(\beta)+(-1)^{c(\gamma)}\lambda(\gamma)\big).$$

\item Let $\alpha$ and $\beta$ be two arcs intersecting at a puncture $v$. If $\gamma_1$ and $\gamma_2$ are the resolutions of $\alpha$ and $\beta$ at $v$, then we have 
$$(-1)^{c(\alpha)+c(\beta)}\lambda(\alpha)\lambda(\beta)=\frac{1}{r(v)}\big((-1)^{c(\gamma_1)}\lambda(\gamma_1)+(-1)^{c(\gamma_2)}\lambda(\gamma_2)\big).$$
\end{enumerate}
\end{proposition}

\begin{proof} By Lemma \ref{3.8}, part (a) is true when $\alpha$ is a geodesic. In the general case when $\alpha$ is not a geodesic, we have the following two cases:
\begin{enumerate}[(1)]
\item $v$ is not a vertex of a generalized bi-gon bounded by $\alpha$, that is, a bi-gon with one of its vertices a puncture; and
\item $v$ is a vertex of a generalized bi-gon $B$ bounded by $\alpha$.
\end{enumerate}
In case (1), we have $c(\alpha)=c(\beta)=c(\gamma)$, and the formula follows from the case that $\alpha$ is geodesic. In case (2), we see one of the resolutions of $\alpha$ at $v$, say $\gamma$, contains a curl from the generalized bi-gon $B$, then the other resolution $\beta$ is the one enclosing the puncture $v$. We let $\beta_1$ be the non-separating resolution of $\beta$ and let $\beta_2$ be the component of the separating resolution of $\beta$ which is not a circle around $v$. Then we have $c(\gamma)=c(\beta_2)+1$ and $\lambda(\gamma)=\lambda(\beta_2)$, hence $(-1)^{c(\gamma)}\lambda(\gamma)=-(-1)^{c(\beta_2)}\lambda(\beta_2)$. By Lemma \ref{prop:induction}, we have $(-1)^{c(\beta)}\lambda(\beta)=(-1)^{c(\beta_1)}\lambda(\beta_1)+(-1)^{c(\beta_2)}2\lambda(\beta_2),$ which implies that $(-1)^{c(\beta)}\lambda(\beta)+(-1)^{c(\gamma)}\lambda(\gamma)=(-1)^{c(\beta_1)}\lambda(\beta_1)+(-1)^{c(\beta_2)}\lambda(\beta_2).$
Let $\alpha'$ be the curve obtained from $\alpha$ by removing the generalized bi-gon $B$ via a Reidemeister Move II$'$, then $\beta_1$ and $\beta_2$ are regularly isotopic to the resolutions of $\alpha'$ at $v$. By case (1) and the last equation above, we have 
\begin{equation*}
\begin{split}
(-1)^{c(\alpha)}\lambda(\alpha)&=(-1)^{c(\alpha')}\lambda(\alpha')\\
&=\frac{1}{r(v)}\big((-1)^{c(\beta_1)}\lambda(\beta_1)+(-1)^{c(\beta_2)}\lambda(\beta_2)\big)\\
&=\frac{1}{r(v)}\big((-1)^{c(\beta)}\lambda(\beta)+(-1)^{c(\gamma)}\lambda(\gamma)\big)
\end{split}
\end{equation*}
Formula (b) is a consequence of Lemma \ref{3.4}\,(1); and the proof is similar to that of (a).
\end{proof}

Combining Propositions \ref{prop:induction} and \ref{prop:puncture}, we obtain the following intermediate theorem.

\begin{theorem}\label{main'} The map $\Phi\colon\C(\s)\rightarrow C^{\infty}(\T^d(\s))$ defined in Theorem \ref{main} is a well-defined commutative algebra homomorphism.
\end{theorem} 

In \cite{Bu}, Bullock conjectured that the map he constructed from the non-quantum skein algebra to the coordinate ring $X(S)$ of the character variety was in fact an isomorphism, which he reduced to the question of showing that there are no non-zero nilpotent elements in $X(S)$. This question was later settled by Przytycki and Sikora \cite{PS}. We thus state the following conjecture.

\begin{conjecture}
The map $\Phi\colon\C(\s)\rightarrow C^{\infty}(\T^d(\s))$ is injective.
\end{conjecture}
\subsection{The homomorphism of Poisson algebras}\label{Poissonalg}

To complete the proof of Theorem \ref{main}, we need the following lemma.

\begin{lemma}\label{product}
Let $T$ be an ideal triangulation of a punctured surface $\s$, and $E$ be its set of edges. Suppose $\alpha$ is a generalized curve on $\s$ and $i(\alpha,e)$ is the number of intersection points of $\alpha$ and $e\in E$. Then the product $\alpha\cdot\prod_{e\in E}e^{i(\alpha,e)}$ in $\C(\s)$ can be expressed as a polynomial $P_{\alpha}$ with variables in $E$.
\end{lemma}

\begin{figure}[htbp]\centering
\includegraphics[width=11cm]{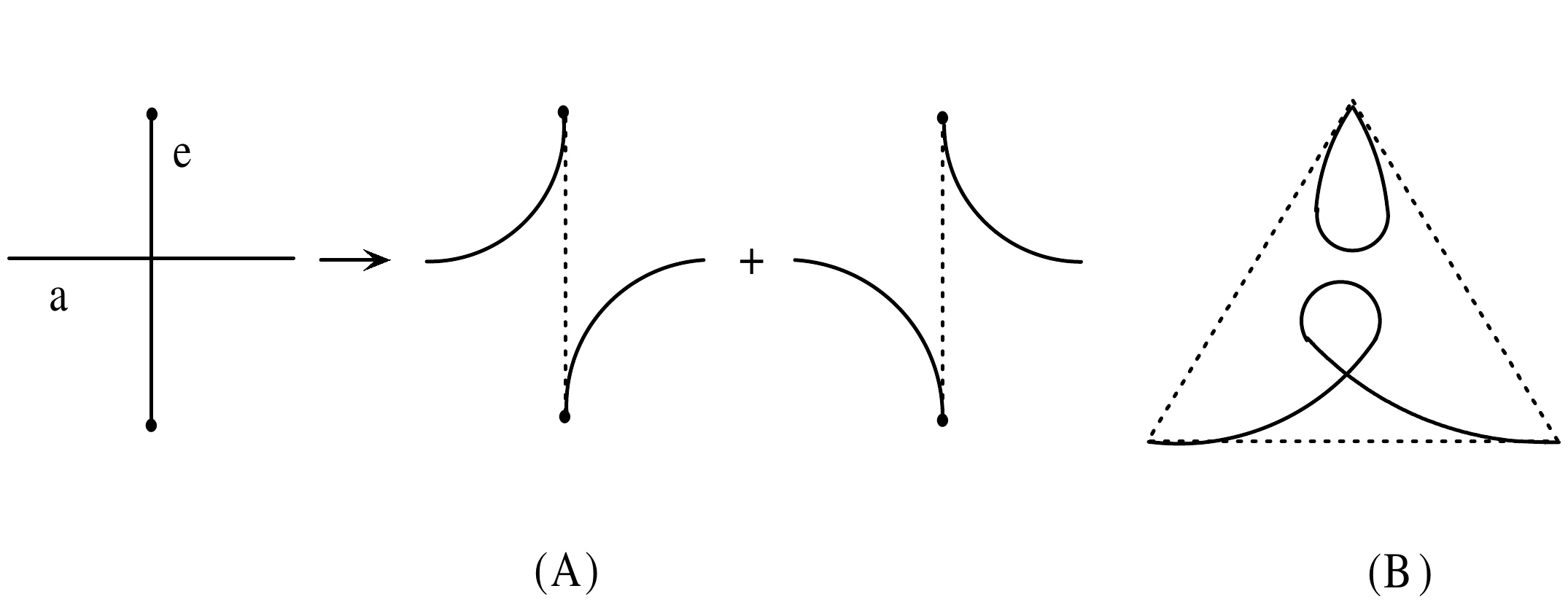}\\
\caption{}\label{pproduct}
\end{figure}

\begin{proof} Let $e\in E$ such that $\alpha\cap e\neq\emptyset$ and $p\in \alpha\cap e$. As in (A) of Figure \ref{pproduct}, each resolution of $\alpha\cdot e$ at $p$ has less intersection number with $e$ than $\alpha$ does. Resolving the product $\alpha\prod_{e\in E}e^{i(\alpha,e)}$ at each point of intersection $p\in\alpha\cap(\bigcup_{e\in E} e)\cap \s$, we see that each component of the final resolution has no intersection with each edge $e\in E$ in the surface, hence must lie in a triangle in $T$. Since a triangle is contractible, each component of the final resolution is either $0$ or up to sign an edge $e\in E$\,(as in (B) of Figure \ref{pproduct}). 
\end{proof}

As a direct consequence of Theorem \ref{main'} and Lemma \ref{product}, we have the following
\begin{proposition}\label{Laurent}
Let $T$ be an ideal triangulation of a punctured surface $\s$ endowed with a decorated hyperbolic metric, and $E$ be its set of edges. Then the $\lambda$-length $\lambda(\alpha)$ of any generalized curve $\alpha$ on $\s$  is a Laurent polynomial in $\{\lambda(e)\ |\ e\in E\}$.
\end{proposition}

\begin{remark} In Appendix~\ref{B} we give explicit formulae for the generalized trace $\lambda(\alpha)$ in terms of the $\lambda$--lengths associated to the edges of an ideal triangulation, $\alpha $ being either a loop or an arc.
\end{remark}

\begin{proof}[Proof of Theorem \ref{main}] For the Poisson structures, we let $T$ be a triangulation of $\s$ with a set of edges $E$. If $e$ and $e'$ are two edges in $E$ each having exactly one end meeting at a puncture $v$, and $x$ and $y$ are the resolutions of $e$ and $e'$ at $v$, then $\{e,e'\}=\frac{1}{4}v^{-1}(x-y)$. By Theorem \ref{main'} and Lemma \ref{3.4}\,(2), we have $\Phi(\{e,e'\})=\frac{1}{4r(v)}(e^{\frac{l(x)}{2}}-e^{\frac{l(y)}{2}})=\frac{1}{4}\frac{\theta'-\theta}{r(v)}e^{\frac{l(e)}{2}}e^{\frac{l(e')}{2}}.$ We also have that $\Pi_{WP}(e^{\frac{l(e)}{2}},e^{\frac{l(e')}{2}})=\frac{1}{16}\frac{\theta'-\theta}{r(v)}e^{\frac{l(e)}{2}}e^{\frac{l(e')}{2}}.$ If either edge has more than one end meeting at $v$ the same computation holds, replacing $x$ and $y$ by the corresponding sums of resolutions and taking the sums of their $\lambda$--lengths instead (see Remark~\ref{multiple} following Lemma~\ref{3.4}). If $e$ and $e'$ are two disjoint edges in $E$, then $\Phi(\{e,e'\})=4\Pi_{WP}(e^{\frac{l(e)}{2}},e^{\frac{l(e')}{2}})=0$. Therefore, we have that $\Phi(\{e,e'\})=4\Pi_{WP}(\lambda(e),\lambda(e'))$ for all pairs of edges in $E$. Now for each generalized curve $\alpha$, by lemma \ref{product}, we have $\alpha\prod_{e\in E}e^{i(\alpha,e)}=P_{\alpha}$ for some polynomial $P_{\alpha}$ with variables in $E$. Since $\{,\}$ is a Poisson bracket, we have $\{\alpha\prod_{e\in E}e^{i(\alpha,e)},e_0\}=\prod_{e\in E}e^{i(\alpha,e)}\{\alpha,e_0\}+\sum_{e'\in E}\alpha\prod_{e\neq e'}e^{i(\alpha,e)}\{e',e_0\}$ for each edge $e_0$ in $E$, from which we see that $\prod_{e\in E}e^{i(\alpha,e)}\{\alpha,e_0\}=\{P_{\alpha},e_0\}-Q$ for some polynomial $Q$ in $\alpha$, $e$ and $\{e,e_0\}$ in which the degrees of $\alpha$ and $\{e,e_0\}$ are equal to $1$. Since $\Phi$ is a $\mathbb{C}$--algebra homomorphism and $\Pi_{WP}$ is a bi-vector field, we have $$P_{\alpha}(\lambda(e))=(-1)^{c(\alpha)}\lambda(\alpha)\prod_{e\in E}\lambda(e)^{i(\alpha,e)},$$ and
\begin{equation*}
\begin{split}
&\Pi_{WP}\big((-1)^{c(\alpha)}\lambda(\alpha)\prod_{e\in E}\lambda(e)^{i(\alpha,e)},\lambda(e_0)\big)\\
=&\prod_{e\in E}\lambda(e)^{i(\alpha,e)}\Pi_{WP}\big((-1)^{c(\alpha)}\lambda(\alpha),\lambda(e_0)\big)+(-1)^{c(\alpha)}Q_{\lambda},
\end{split}
\end{equation*}
where $Q_{\lambda}$ is the value of $Q$ at $\lambda(\alpha)$, $\lambda(e)$ and $\Pi_{WP}(\lambda(e),\lambda(e_0))$.
As a consequence, since $\lambda(e)\neq0$ for each $e\in E$, we have 
\begin{equation*}
\begin{split}
\Phi(\{\alpha,e_0\})&=\frac{\Phi(\{P_{\alpha},e_0\})-\Phi(Q)}{\prod_{e\in E}\Phi(e)^{i(\alpha,e)}}\\
&=\frac{4\Pi_{WP}\big(P_{\alpha}(\lambda(e)),\lambda(e_0)\big)-(-1)^{c(\alpha)}4Q_{\lambda}}{\prod_{e\in E}\lambda(e)^{i(\alpha,e)}}\\
&=4\Pi_{WP}\big((-1)^{c(\alpha)}\lambda(\alpha),\lambda(e_0)\big).
\end{split}
\end{equation*}
For two generalized curves $\alpha$ an $\beta$, we let $\alpha\prod _{e\in E}e^{i(\alpha,e)}=P_{\alpha}$ and $\beta\prod_{e\in E}e^{i(\beta,e)}=P_{\beta}$ as in Lemma \ref{product}. Then we have $\prod_{e\in E}e^{i(\alpha,e)+i(\beta,e)}\{\alpha,\beta\}=\{P_{\alpha},P_{\beta}\}-R$, where $R$ is a polynomial in $\alpha$, $\beta$, $e$, $\{\alpha,e\}$, $\{e,\beta\}$ and $\{e,e'\}$ such that the degrees of $\alpha$, $\beta$, $\{\alpha,e\}$, $\{e,\beta\}$ and $\{e,e'\}$ are all equal to $1$. Therefore, we have $\Phi(R)=4R_{\lambda}$, where $R_{\lambda}$ is the value of $R$ at $\lambda(\alpha)$, $\lambda(\beta)$, $\lambda(e)$, $\Pi_{WP}(\lambda(\alpha)$, $\lambda(e))$, $\Pi_{WP}(\lambda(e),\lambda(\beta))$ and $\Pi_{WP}(\lambda(e),\lambda(e'))$, and 
\begin{equation*}
\begin{split}
\Phi(\{\alpha,\beta\})&=\frac{\Phi(\{P_{\alpha},P_{\beta}\})-\Phi(R)}{\prod_{e\in E}\Phi(e)^{i(\alpha,e)+i(\beta,e)}}\\
&=\frac{4\Pi_{WP}\big(P_{\alpha}(\lambda(e)),P_{\beta}(\lambda(e))\big)-(-1)^{c(\alpha)+c(\beta)}4R_{\lambda}}{\prod_{e\in E}\lambda(e)^{i(\alpha,e)+i(\beta,e)}}\\
&=4\Pi_{WP}\big((-1)^{c(\alpha)}\lambda(\alpha),(-1)^{c(\beta)}\lambda(\beta)\big).
\end{split}
\end{equation*}
Let $\pi\colon \T^d(\s)\rightarrow\mathbb{R}_{>0}^V$ be the projection onto the fiber.  By Mondello\,\cite{Mondello}, the kernel of $\Pi_{WP}$ is the pull-back $\pi^*(T^*\mathbb{R}_{>0}^V)$ of the cotangent space of $\mathbb{R}_{>0}^V$. Since $d(r(v))=\pi^*(dv)\in \pi^*(T^*\mathbb{R}_{>0}^V)$, we have $$\Phi(\{v,\alpha\})=4\Pi_{WP}\big(r(v),(-1)^{c(\alpha)}\lambda(\alpha)\big)=0$$ for each puncture $v$ and each generalized curve $\alpha$.
\end{proof}

As a consequence of Theorem \ref{main}, Wolpert's cosine formula generalizes to the bi-vector field $\Pi_{WP}$ as follows:

\begin{corollary} Let $\theta_p$ be the angle from $\alpha$ to $\beta$ at $p\in \alpha\cap\beta$ in $\s$. If $\alpha$ and $\beta$ are two geodesic arcs, then let $\theta_v$ be the generalized angle from $\alpha$ to $\beta$ and let $\theta'_v$ be the generalized angle from $\beta$ to $\alpha$ at a puncture $v\in \alpha\cap\beta$. We have

$$\Pi_{WP}(l(\alpha),l(\beta))=\frac{1}{2}\sum_{p\in \alpha\cap \beta\cap\s}\cos\theta_p+\frac{1}{4}\sum_{v\in \alpha\cap\beta\cap V}\frac{\theta'_v-\theta_v}{r(v)}.$$

\end{corollary}

\begin{proof} We let $\lambda'(\alpha)=\sinh\frac{l(\alpha)}{2}$ if $\alpha$ is a loop on $\s$, and let $\lambda'(\alpha)=\frac{1}{2}e^{\frac{l(\alpha)}{2}}$  if $\alpha$ is an arc on $\s$. By Theorem \ref{main} and (2) of Lemma \ref{3.1}\,-\,\ref{3.4}, we have 
 \begin{equation*}
 \begin{split}
&\Pi_{WP}(l(\alpha),l(\beta))\\
=\,&\frac{1}{\lambda'(\alpha)\lambda'(\beta)}\Pi_{WP}(\lambda(\alpha),\lambda(\beta))\\
 =\,&\frac{1}{4\lambda'(\alpha)\lambda'(\beta)}\Phi(\{\alpha,\beta\})\\
 =\,&\frac{1}{4\lambda'(\alpha)\lambda'(\beta)}\Phi\big(\frac{1}{2}\sum_{p\in\alpha\cap\beta\cap\s}(\alpha_p\beta^+-\alpha_p\beta^-)+\frac{1}{4}\sum_{v\in\alpha\cap\beta\cap V}\frac{1}{v}(\alpha_v\beta^+-\alpha_v\beta^-)\big)\\
 =\,&\frac{1}{8}\sum_{p\in\alpha\cap\beta\cap\s}\frac{\lambda(\alpha_p\beta^+)-\lambda(\alpha_p\beta^-)}{\lambda'(\alpha)\lambda'(\beta)}+\frac{1}{16}\sum_{v\in\alpha\cap\beta\cap V}\frac{1}{r(v)}\frac{\lambda(\alpha_v\beta^+)-\lambda(\alpha_v\beta^-)}{\lambda'(\alpha)\lambda'(\beta)}\\
 =\,&\frac{1}{2}\sum_{p\in \alpha\cap \beta\cap\s}\cos\theta_p+\frac{1}{4}\sum_{v\in \alpha\cap\beta\cap V}\frac{\theta'_v-\theta_v}{r(v)}.
\end{split}
\end{equation*}
where if $\alpha_v\beta^\pm$ corresponds to a sum of resolutions seen as an element in $\C(\s)$, then $\lambda(\alpha_v\beta^\pm)$ consists of the sum of their $\lambda$--lengths.

\end{proof}


\appendix
\section{Cosine and sine laws of twisted generalized triangles}\label{A}
\begin{enumerate}
\item Type $(1,1,-1)$:
\[\includegraphics[width=3.5cm]{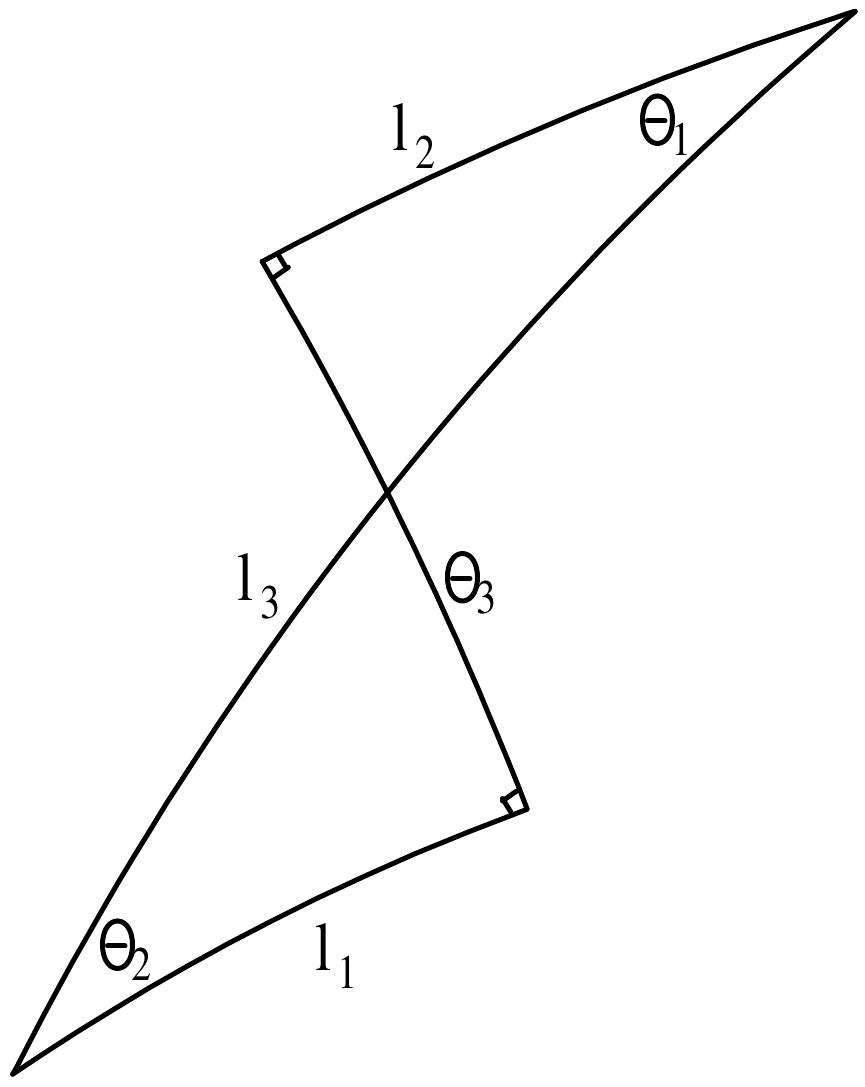}\]
\begin{equation*}
\begin{split}
&\sinh l_1=\frac{-\cos\theta_1+\cos\theta_2\cosh\theta_3}{\sinh\theta_2\sin\theta_3}\quad\quad\cos\theta_1=\frac{\sinh l_1+\sinh l_2\cosh l_3}{\cosh l_2\sinh l_3}\\
&\sinh l_2=\frac{-\cos\theta_2+\cos\theta_1\cosh\theta_3}{\sin\theta_1\sinh\theta_3}\quad\quad\cos\theta_2=\frac{\sinh l_2+\sinh l_1\cosh l_3}{\cosh l_1\sinh l_3}\\
&\cosh l_3=\frac{\cosh\theta_3-\cos\theta_1\cos\theta_2}{\sin\theta_1\sin\theta_2}\quad\quad\quad\cosh\theta_3=\frac{\cosh l_3-\sinh l_1\sinh l_2}{\cosh l_1\cosh l_2}\\
&\frac{\sin\theta_1}{\cosh l_1}=\frac{\sin\theta_2}{\cosh l_2}=\frac{\sinh\theta_3}{\sinh l_3}
\end{split}
\end{equation*}

\item Type $(1,0,-1)$:
\[\includegraphics[width=3.7cm]{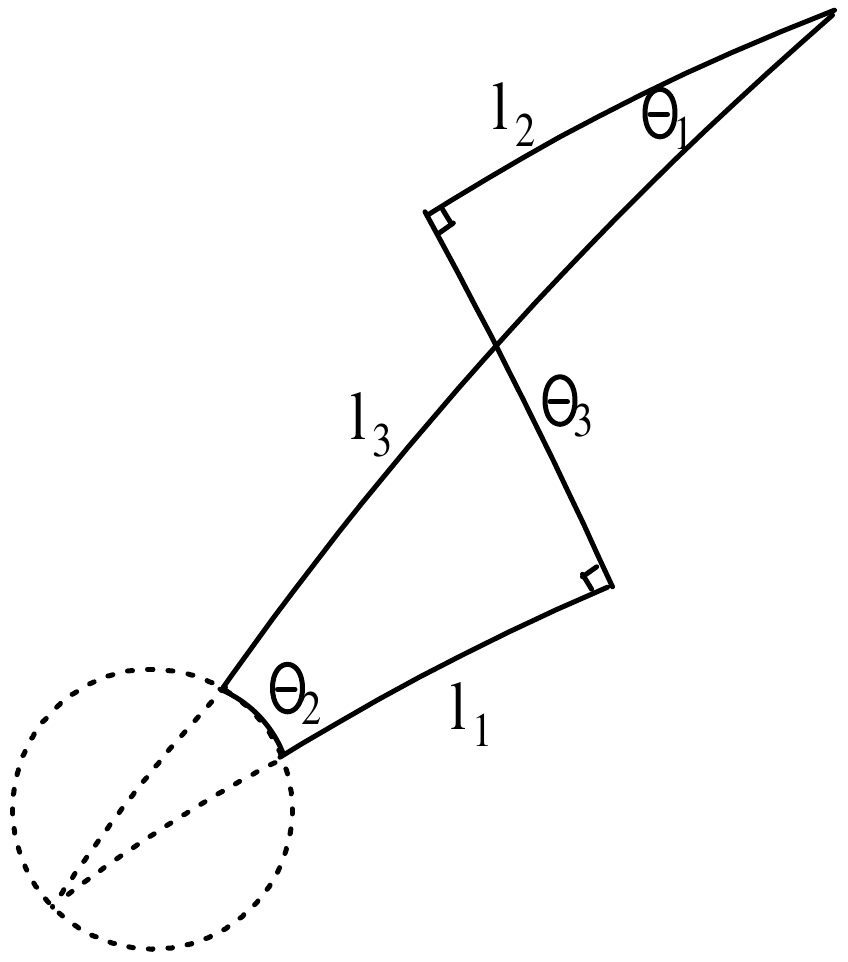}\]
\begin{equation*}
\begin{split}
&e^{l_1}=\frac{-\cos\theta_1+\cosh\theta_3}{\theta_2\sinh\theta_3}\quad\quad\quad\quad\cos\theta_1=\frac{e^{l_1}+e^{l_3}\sinh l_2}{e^{l_3}\cosh l_2}\\
&\sinh l_2=\frac{-1+\cos\theta_1\cosh\theta_3}{\sin\theta_1\sinh\theta_3}\quad\quad\theta_2^2=\frac{-\sinh l_2+\sinh(l_3-l_1)}{\frac{e^{l_1+l_3}}{2}}\\
&e^{l_3}=\frac{\cosh\theta_3-\cos\theta_1}{\theta_2\sin\theta_1}\quad\quad\quad\quad\quad\cosh\theta_3=\frac{e^{l_3}-e^{l_1}\sinh l_2}{e^{l_1}\cosh l_2}\\
&\frac{\sin\theta_1}{e^{l_1}}=\frac{\theta_2}{\cosh l_2}=\frac{\sinh\theta_3}{e^{l_3}}
\end{split}
\end{equation*}

\item Type $(1,-1,-1)$: 
\[\includegraphics[width=3.5cm]{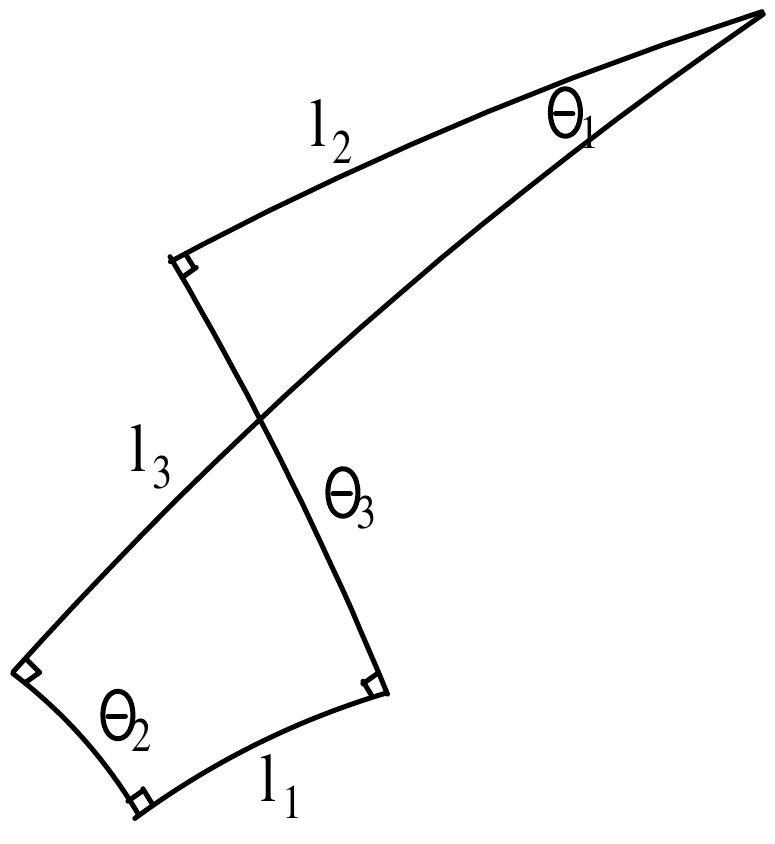}\]
\begin{equation*}
\begin{split}
&\cosh l_1=\frac{-\cos\theta_1+\cosh\theta_2\cosh\theta_3}{\sinh\theta_2\sinh\theta_3}\quad\quad\cos\theta_1=\frac{\cosh l_1+\sinh l_2\sinh l_3 }{\cosh l_2\cosh l_3}\\
&\sinh l_2=\frac{-\cosh\theta_2+\cos\theta_1\cosh\theta_3}{\sin\theta_1\sinh\theta_3}\quad\quad\cosh\theta_2=\frac{-\sinh l_2+\cosh l_1\sinh l_3}{\sinh l_1\cosh l_3}\\
&\sinh l_3=\frac{\cosh\theta_3-\cos\theta_1\cosh\theta_2}{\sin\theta_1\sinh\theta_2}\quad\quad\quad\cosh\theta_3=\frac{\sinh l_3-\cosh l_1\sinh l_2}{\sinh l_1\cosh l_2}\\
&\frac{\sin\theta_1}{\sinh l_1}=\frac{\sinh\theta _2}{\cosh l_2}=\frac{\sinh\theta _3}{\cosh l_3}
\end{split}
\end{equation*}

\item Type $(0,0,-1)$:
\[\includegraphics[width=4cm]{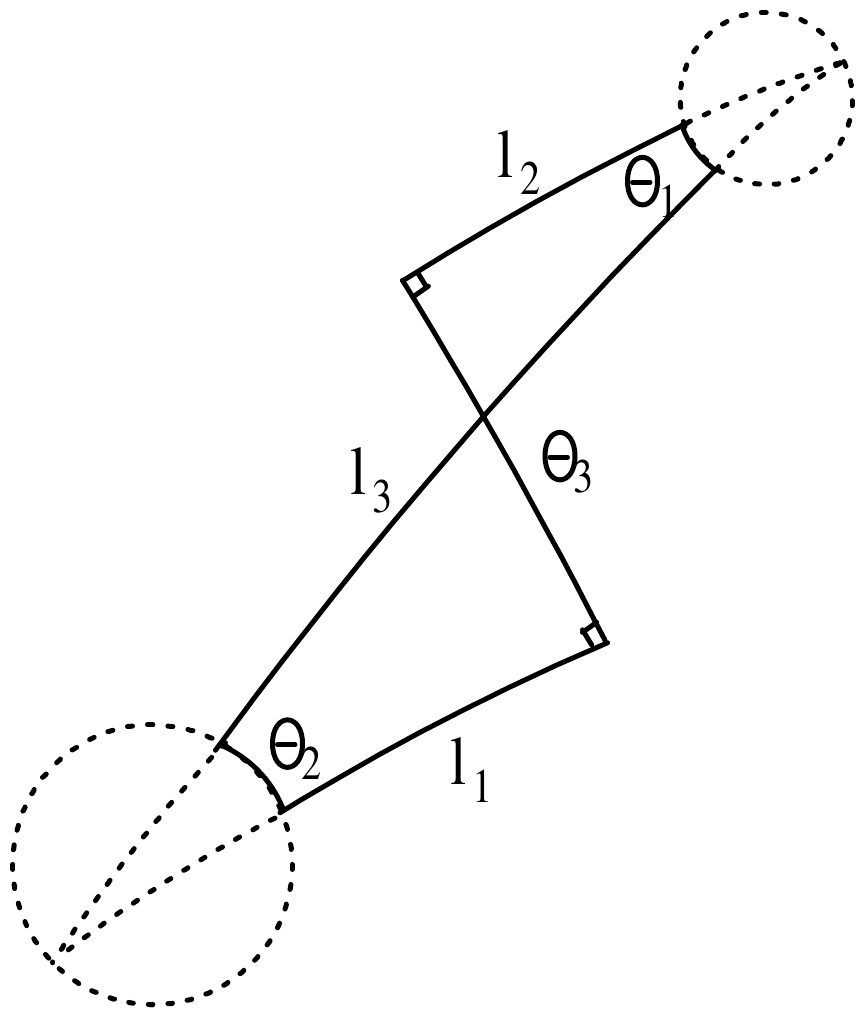}\]
\begin{equation*}
\begin{split}
&e^{l_1}=\frac{-1+\cosh\theta_3}{\theta_2\sinh\theta_3}\quad\quad\quad\ \theta_1^2=\frac{-e^{l_1}+e^{l_3-l_2}}{e^{l_2+l_3}}\\
&e^{l_2}=\frac{-1+\cosh\theta_3}{\theta_1\sinh\theta_3}\quad\quad\quad\ \theta_2^2=\frac{-e^{l_2}+e^{l_3-l_1}}{e^{l_1+l_3}}\\
&e^{l_3}=\frac{\cosh\theta_3-1}{2\theta_1\theta_2}\quad\quad\quad\quad\cosh^2\frac{\theta_3}{2}=e^{l_3-l_1-l_2}\\
&\frac{\theta_1}{e^{l_1}}=\frac{\theta_2}{e^{l_2}}=\frac{\sinh\theta_3}{2e^{l_3}}
\end{split}
\end{equation*}

\item Type $(0,-1,-1)$:
\[\includegraphics[width=3.2cm]{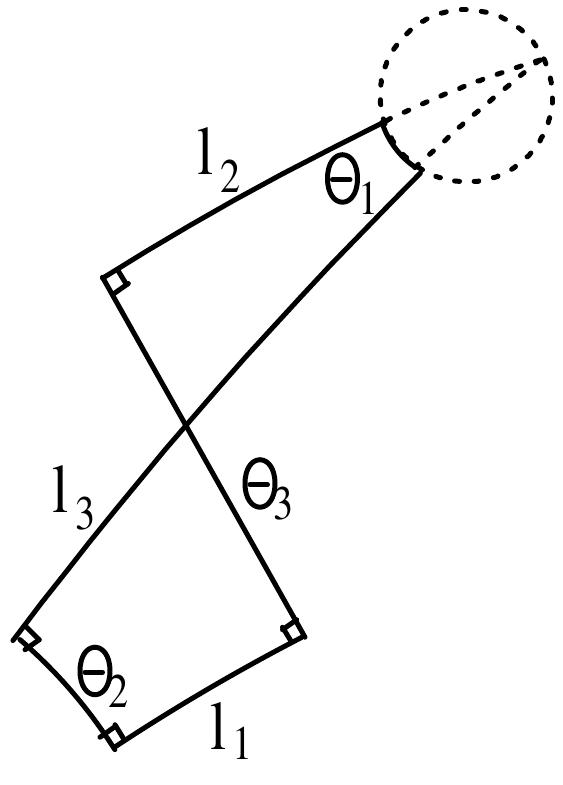}\]
\begin{equation*}
\begin{split}
&\cosh l_1=\frac{-1+\cosh\theta_2\cosh\theta_3}{\sinh\theta_2\sinh\theta_3}\quad\quad\ \theta_1^2=\frac{-\cosh l_1+\cosh(l_3-l_2)}{\frac{e^{l_2+l_3}}{2}}\\
&e^{l_2}=\frac{-\cosh\theta_2+\cosh\theta_3}{\theta_1\sinh\theta_3}\quad\quad\quad\quad\ \ \cosh\theta_2=\frac{-e^{l_2}+e^{l_3}\cosh l_1}{e^{l_3}\sinh l_1}\\
&e^{l_3}=\frac{\cosh\theta_3-\cosh\theta_2}{\theta_1\sinh\theta_2}\quad\quad\quad\quad\quad\ \  \cosh\theta_3=\frac{e^{l_3}-e^{l_2}\cosh l_1}{e^{l_2}\sinh l_1}\\
&\frac{\theta_1}{\sinh l_1}=\frac{\sinh\theta_2}{e^{l_2}}=\frac{\sinh\theta_3}{e^{l_3}}
\end{split}
\end{equation*}

\item Type $(-1,-1,-1)$:
\[\includegraphics[width=3.2cm]{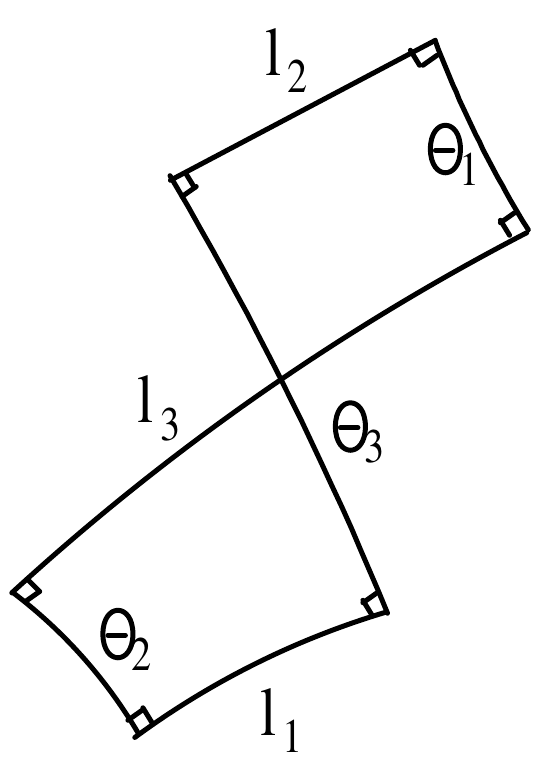}\]
\begin{equation*}
\begin{split}
&\cosh l_1=\frac{-\cosh\theta_1+\cosh\theta_2\cosh\theta_3}{\sinh\theta_2\sinh\theta_3}\quad\quad\cosh\theta_1=\frac{-\cosh l_1+\cosh l_2\cosh l_3}{\sinh l_2\sinh l_3}\\
&\cosh l_2=\frac{-\cosh\theta_2+\cosh\theta_1\cosh\theta_3}{\sinh\theta_1\sinh\theta_3}\quad\quad\cosh\theta_2=\frac{-\cosh l_2+\cosh l_1\cosh l_3}{\sinh l_1\sinh l_3}\\
&\cosh l_3=\frac{\cosh\theta_3-\cosh\theta_1\cosh\theta_2}{\sinh\theta_1\sinh\theta_2}\quad\quad\quad\cosh\theta_3=\frac{\cosh l_3-\cosh l_1\cosh l_2}{\sinh l_1\sinh l_2}\\
&\frac{\sinh\theta_1}{\sinh l_1}=\frac{\sinh\theta_2}{\sinh l_2}=\frac{\sinh\theta_3}{\sinh l_3}
\end{split}
\end{equation*}
\end{enumerate}

\section{Trace functions and $\lambda$--lengths: an explicit formula}\label{B}
Let $\s$ be a surface with at least one puncture and such that $\chi(\s)<0$, and choose a decorated hyperbolic metric $(m,r)$ on $\s$. Fix an ideal triangulation $T$ of $\s$ with triangles $\{\Delta_1,\ldots,\Delta_s\}$ and edges $\{e_1,\ldots,e_t\}$. We let $\alpha$ be a loop or an arc on $\s$ and let $\overline{\alpha}$ be the unique geodesic homotopic to $\alpha$ for the given metric. In addition, we choose an arbitrary orientation for $\overline{\alpha}$. If $\overline{\alpha}$ is a loop, we list the successive triangles that $\overline{\alpha}$ crosses by $\Delta_{i_1}$ through $\Delta_{i_n}$ starting at some arbitrary point inside the first triangle. Similarly we list the edges that $\overline{\alpha}$ crosses $e_{j_1}$ through $e_{j_n}$ in such a way that $e_{j_k}$ is shared by $\Delta_{i_k}$ and $\Delta_{i_{k+1}}$  for $1\leq k\leq n-1$ and $e_{j_n}$ is shared by $\Delta_{i_n}$ and $\Delta_{i_1}$. If $\overline{\alpha}$ is an arc, we list triangles and edges accordingly, from $\Delta_{i_1}$ containing the beginning vertex to $\Delta_{i_{n+1}}$ containing the end vertex. Note that in either case the same triangles and edges can occur several times. We denote by $\lambda_i=\lambda(e_i)$ the $\lambda$-length of each edge for $(m,r)$.

If $\overline{\alpha}$ makes a left turn in $\Delta_{i_k}$ as in (A) of Figure \ref{rl}, then we let 
\begin{equation*} M_{k} = \left(
\begin{array}{cc} \lambda_b & \lambda_r\\ 
0 & \lambda_l\\
\end{array} \right)
\end{equation*}
for the corresponding labels $b$, $l$ and $r$ and if $\overline{\alpha}$ makes a right turn in $\Delta_{i_k}$ as in (B) of Figure \ref{rl}, then we let 
\begin{equation*} M_{k} = \left(
\begin{array}{cc} \lambda_r & 0\\ 
\lambda_l & \lambda_b\\
\end{array} \right).
\end{equation*}
In addition, If $\alpha$ is an arc starting at $\Delta_{i_1}$ and ending at $\Delta_{i_{n+1}}$ as in ({C}) and (D) of Figure \ref{rl}, then we let 
\begin{equation*} M_{1}=(\ \lambda_r\ \ \lambda_l\ ) \text{\quad and\quad } M_{n+1} = \left(
\begin{array}{c} \lambda_r \\ 
\lambda_l \\
\end{array} \right).
\end{equation*}

\begin{figure}[htbp]\centering
\includegraphics[width=12cm]{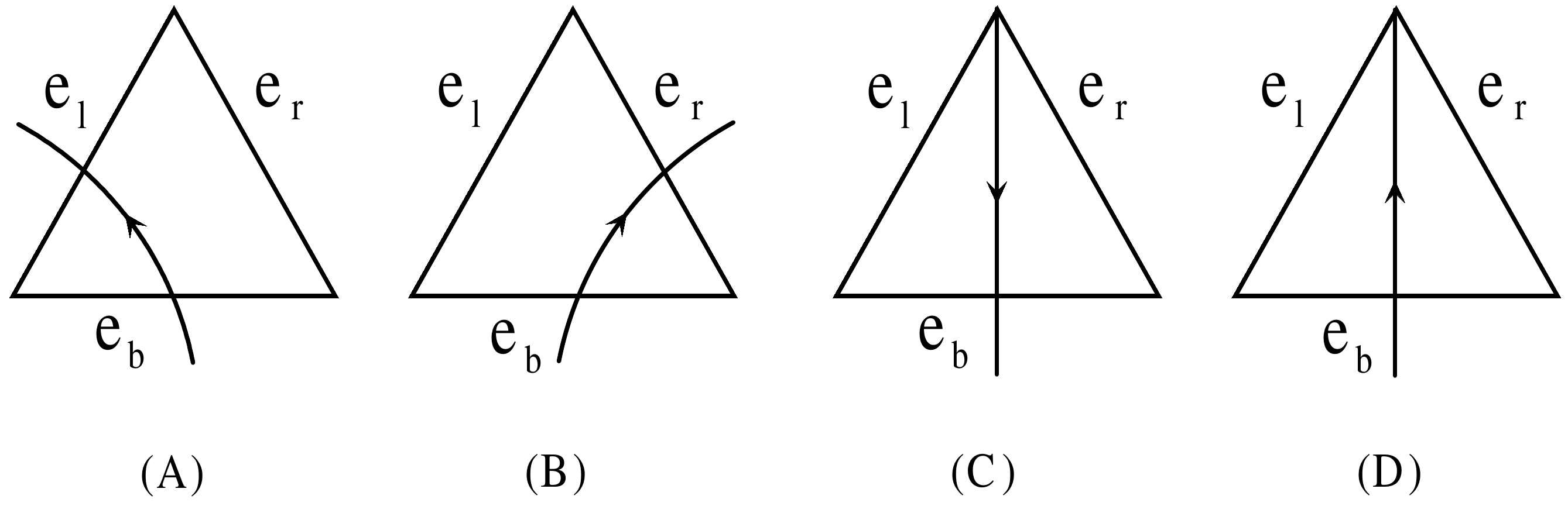}\\
\caption{intersections of $\overline{\alpha}$ with ideal triangles}\label{rl}
\end{figure}
Keeping track of the resolutions at the intersections of $\overline{\alpha}\cdot\prod_ke_{j_k}$ in the order given by the orientation of $\overline{\alpha}$, we obtain the following result.
\begin{theorem}
Let $\alpha$ be a generalized curve on a decorated hyperbolic surface $\s$ with a given ideal triangulation $T$, and let $\lambda_{j_k}$ and $M_{k}$ be as above. Then we have
\begin{equation*} 
\lambda(\alpha) = \left\{
\begin{array}{cl}
\frac{\displaystyle{tr(M_{1}\cdots M_{n})}}{\displaystyle{\lambda_{j_1}\cdots\lambda_{j_n}}}, & \text{if } \alpha \text{ is a loop},\\ 
 \\
\frac{\displaystyle{M_{1}\cdots M_{n+1}}}{\displaystyle{\lambda_{j_1}\cdots\lambda_{j_n}}}, & \text{if } \alpha \text{ is an arc}.
\end{array} \right.
  \end{equation*}
\end{theorem}

In the case of a loop, one should compare this formula with the expression of trace functions in terms of shear coordinates which can be found in \cite{BW2} Lemma 3.

\noindent
\noindent
Julien Roger\\
Department of Mathematics, Rutgers University\\
New Brunswick, NJ 08854, USA\\
(juroger@math.rutgers.edu)
\\

\noindent 
Tian Yang\\
Department of Mathematics, Rutgers University\\
New Brunswick, NJ 08854, USA\\
(tianyang@math.rutgers.edu)

\end{document}